\documentclass[12pt,letterpaper]{article}

\usepackage[margin=1in]{geometry}
\usepackage{setspace}

\usepackage{multirow} %
\usepackage{booktabs}  
\usepackage{siunitx} 
\usepackage[utf8]{inputenc}
\usepackage{tikz}
\usetikzlibrary{matrix,calc}

\usepackage{graphicx}
\usepackage{comment}
\usepackage{xcolor}
\usepackage{float}
\usepackage{mathtools,amsmath,amssymb,amsthm,accents}
\usepackage{bbm}
\usepackage{enumitem}

\newcommand{\TVert}[1]{{\left\vert\kern-0.25ex\left\vert\kern-0.25ex\left\vert #1 
    \right\vert\kern-0.25ex\right\vert\kern-0.25ex\right\vert}}

\newcommand\deq{\stackrel{\mathclap{\scriptsize\mbox{d}}}{=}}

\newcommand\convd{\stackrel{\mathclap{\scriptsize\mbox{d}}}{\rightarrow}}
\newcommand\convp{\stackrel{\mathclap{\scriptsize\mbox{p}}}{\rightarrow}}
\newcommand\convas{\stackrel{\mathclap{\scriptsize\mbox{a.s.}}}{\rightarrow}}
\newtheorem{theorem}{Theorem}[section]
\newtheorem{corollary}{Corollary}[theorem]
\newtheorem{lemma}[theorem]{Lemma}
\newtheorem{prop}[theorem]{Proposition}
\newtheorem{definition}[theorem]{Definition}
\newtheorem{assumption}[theorem]{Assumption}

\usepackage{xcolor,lipsum}
\usepackage{titlesec}

\titleformat{name=\paragraph}[block]
  {\sffamily\large}
  {}
  {0pt}
  {\colorparagraph}
\titlespacing*{\paragraph}{0pt}{\baselineskip}{\baselineskip}

\newcommand{\colorparagraph}[1]{%
  \colorbox{blue!20}{\parbox{\dimexpr\textwidth-2\fboxsep}{#1}}}

\usepackage[normalem]{ulem}
\usepackage[
backend=biber,
style=alphabetic,
sorting=ynt,
maxbibnames=99
]{biblatex}

\newcommand{\E}{\mathbbm{E}}
\newcommand{\Tr}{\text{Tr}}
\newcommand{\Id}{\text{Id}}
 \usepackage{mathrsfs}
\usepackage{authblk}

\title{CLT for Linear Spectral Statistics in High-Dimensional Random Effects Models}

\author{Ran Xie}
\author{Iain Johnstone}
\affil{Department of Statistics, Stanford University}

\date{}
\newcommand{\deteqv}[1]{\tilde{#1}}

\begin{document}

\maketitle
\begin{abstract}
    We study sample covariance matrices arising from multi-level components of variance. Thus, let
    $ B_n=\frac{1}{N}\sum_{j=1}^NT_{j}^{1/2}x_jx_j^TT_{j}^{1/2}$, where $x_j\in\mathbbm{R}^n$ are i.i.d. standard Gaussian, and $T_{j}=\sum_{r=1}^kl_{jr}^2\Sigma_{r}$ are $n\times n$ real symmetric matrices with bounded spectral norm, corresponding to $k$ levels of variation. As the matrix dimensions $n$ and $N$ increase proportionally, we show that the linear spectral statistics (LSS) of $B_n$ have Gaussian limits. The CLT is expressed as the convergence of a set of LSS to a standard multivariate Gaussian after centering by a mean vector $\Gamma_n$ and a covariance matrix $\Lambda_n$ which depend on $n$ and $N$ and may be evaluated numerically. 
    Our work is motivated by the estimation of high-dimensional covariance matrices between phenotypic traits in quantitative genetics, particularly within nested linear random-effects models with up to $k$ levels of randomness. Our proof builds on the Bai-Silverstein \cite{baisilverstein2004} martingale method with some innovation to handle the multi-level setting.
\end{abstract}
\section{Introduction}
Consider a matrix 
\begin{align}\label{eq:Bn}
   B_n=\frac{1}{N}\sum_{j=1}^NT_{j}^{1/2}x_jx_j^TT_{j}^{1/2}, 
\end{align}
where $x_j\in\mathbbm{R}^n$ are i.i.d. standard Gaussian, and $T_{j}$ are $n \times n$ real symmetric matrices given by
\begin{align}\label{eq:Tn}
    T_{j}=\sum_{r=1}^kl_{jr}^2\Sigma_{r}.
\end{align}
Above, $l_{jr}$ are real scalars and $\Sigma_{r}$ represent $n\times n$ real symmetric matrices. In this paper, we aim to establish a central limit theorem for the linear spectral statistics (LSS) of matrix $B_n$. For a function $f$ on $[0,\infty)$, the corresponding LSS is defined as
\begin{align*}
    \frac{1}{n}\sum_{i=1}^nf(\lambda_i)=\int f(x)dF^{B_n}(x),
\end{align*}
where $\lambda_1,\ldots,\lambda_n$ denote the eigenvalues of $B_n$, and $F^{B_n}$ denotes the empirical spectral distribution (ESD) of $B_n$ given by
\begin{align*}
    F^{B_n}(x)=\frac{1}{n}\sum_{i=1}^n\mathbbm{1}[\lambda_i\leq x].
\end{align*}
\subsection*{Motivation}
Our work is motivated by estimating high-dimensional genetic covariance matrices (G-matrix) within random effects models. For example, consider a one-way full sibling model with $F$ families and $S_f$ siblings in each family.\footnote{Please that $F$ and $S_f$ are not the usual notation; they are adopted in this context to avoid notation collision.} For each of the $s$th sibling in the $f$th family, we measure $p$ gene expression traits $y_{fs}\in\mathbbm{R}^p$, modeled by 
\begin{align}\label{eq:full.sib.design}
    y_{fs}=\alpha_f + \epsilon_{fs},\quad \alpha_f \sim N(0,\Sigma_A),\quad \epsilon_{fs}\sim N(0,\Sigma_E),
\end{align}
with independent family and individual effects $\alpha_f$, $\epsilon_{fs}$. We focus on settings in which both $F$ and $p$ are large and of comparable order. For instance, at a phenome-wide scale, there could be up to thousands of gene expression traits measured for each sample.

Suppose we stack the trait measurements $y_{fs}$ by row as a matrix $Y\in\mathbbm{R}^{n_s\times p}$, where the number of samples $n_s=\sum_{f=1}^FS_f$. Thus, we can express the model in the matrix form $Y=U\alpha + \epsilon$, where $U= \text{diag}(\mathbbm{1}_{S_1},...,\mathbbm{1}_{S_F})\in\mathbbm{R}^{n_s\times F}$ is the membership matrix that assigns each sibling to the corresponding family, $\alpha\in\mathbbm{R}^{F\times p}$, $\epsilon\in\mathbbm{R}^{n_s\times p}$ are the family and individual effects. To estimate the eigenvalue distributions of $\Sigma_A$ and $\Sigma_E$, we rely on the sum of squares matrix
\begin{align}\label{eq:Bn.full.sib}
    B_p=\frac{1}{F}Y^T\pi Y,
\end{align}
with $\pi=U(U^TU)^{-1}U^T$ being the projection matrix on the column space of $U$. Let us see how $B_p$ is in distribution equivalent to the form (\ref{eq:Bn}), with the matrix dimensions $(n,N)=(p,F)$ here. Plugging in the matrix form of $Y$, we can show that $(U^TU)^{-1/2}U^T(U\alpha)\deq L_1X_1\Sigma_A^{1/2}$, $(U^TU)^{-1/2}U^T(\epsilon)\deq L_2X_2\Sigma_E^{1/2}$, with $L_1=\text{diag}\{\sqrt{S_1},...,\sqrt{S_F}\}$, $L_2=\text{Id}_F$  derived from the membership matrix $U$, $X_1=[x_{11},\ldots,x_{1F}]^T$, $X_2=[x_{21},\ldots,x_{2F}]^T\in\mathbbm{R}^{F\times p}$ are independent random matrices with i.i.d. standard Gaussian entries. Therefore, we can equivalently express $B_p$ as
\begin{align*}
    B_p\deq\frac{1}{F}(L_1X_1\Sigma_A^{1/2}  +L_2X_2\Sigma_E^{1/2})^T(L_1X_1\Sigma_A^{1/2}  +L_2X_2\Sigma_E^{1/2}).
\end{align*}
Let $T_{f}=l_{1f}^2\Sigma_A+l_{2f}^2\Sigma_E$, then by the independence of the family and individual effects, there further exist i.i.d. standard Gaussian vectors 
$x_f\in\mathbbm{R}^p$ such that $l_{1f}\Sigma_A^{1/2}x_{1f}+l_{2f}\Sigma_E^{1/2}x_{2f}\deq T_{f}^{1/2}x_f$. Therefore, we can alternatively express $B_p$ as
\begin{align*}
    B_p&\deq\frac{1}{F}\sum_{f=1}^F T_{f}^{1/2}x_fx_f^TT_{f}^{1/2}, \quad T_{f}=l_{1f}^2\Sigma_A+l_{2f}^2\Sigma_E.
\end{align*}
The full sibling model contains two levels of variation and corresponds to $k=2$ in our general formulation of $T_f$ as in (\ref{eq:Tn}). Introducing half-siblings, corresponding to $k=3$, extends the model to the full-sib half-sib design. This model is widely used for estimating breeding values and heritability, which are essential for optimizing selection schemes in poultry breeding programs \cite{falconer1996introduction}. Further discussions on nested multivariate random effects models are given in Section \ref{subsec:multi.random}.
\subsection*{Approach and prior work}
We aim to establish a central limit theorem for the difference between the LSS of $B_n$ and that of its deterministic equivalent $\deteqv{B}_n$. The latter is deterministic, it can be computed from $T_{j}$, $j=1,\ldots,N$, and shares the same limiting spectral distribution as $B_n$. A formal definition of this matrix will be presented in Section \ref{sec:prelim}.

 Informally, we prove the existence of sequences of biases $\mu_n$ and variances $\sigma_n$, such that 
\begin{align*}
  \sigma_n^{-1}\left[ \sum_{i=1}^n\left(f(\lambda_i)-f(\deteqv{\lambda}_i)\right)-\mu_n\right]
\end{align*}
converges to a standard normal distribution. Here, $\deteqv{\lambda}_1,\ldots,\deteqv{\lambda}_n$ denote the eigenvalues of $\deteqv{B}_n$. Moreover, we generalize our proof to establish the asymptotic normality for an arbitrary finite list of functions $\{f_\nu\}$.

For sample covariance matrices corresponding to $B_n$ as in (\ref{eq:Bn}) with all $T_{j}\equiv T$, Bai and Silverstein \cite{baisilverstein2004} established the CLT of linear spectral statistics using a three-step strategy:
\renewcommand\labelitemi{\scriptsize$\bullet$}
\begin{itemize}
    \item Reduce problem to showing the CLT for scaled Stieltjes transform 
    \begin{align*}
        M_n(z)&=\text{Tr}(B_n-z)^{-1}-\text{Tr}(\deteqv{B}_n-z)^{-1}\\
        &=M_n^1(z)+M_n^2(z),
    \end{align*}
    where
    \begin{align*}
        M_n^1(z)&=\text{Tr}(B_n-z)^{-1}-\E \text{Tr}(B_n-z)^{-1},\\
        M_n^2(z)&=\E \text{Tr}(B_n-z)^{-1}-\text{Tr}(\deteqv{B}_n-z)^{-1}.
    \end{align*}
    \item Prove CLT for the centralized term $M_n^1(z)$ via the Martingale Central Limit Theorem.
    \item Compute the limit of the deterministic term $M_n^2(z)$.
\end{itemize}
This strategy has since been extensively adapted to show CLT of linear spectral statistics for broader classes of large sample covariance matrices. For information-theoretic type sample covariance matrix $S_n=(\frac{1}{\sqrt{N}}L_1XL_2+A)(\frac{1}{\sqrt{N}}L_1XL_2+A)^T$ with $L_1$, $L_2$ diagonal and $A$ deterministic, \cite{silverstein2012} established the CLT for log-determinant statistics, a special class of linear spectral statistics. For $S_n=(\frac{\sigma}{\sqrt{N}}X+A)(\frac{\sigma}{\sqrt{N}}X+A)^T$ with deterministic $A$, 
\cite{BANNA20202250} showed a CLT for linear spectral statistics corresponding to general functions $f$. \cite{bai2019central} establishes a result most closely related to ours; they consider matrices $B_n$ as in (\ref{eq:Bn}) with $T_{j}=l_{j}^2\Sigma$, which can be viewed as corresponding to a variance component estimator within linear random-effects models with a single level of randomness.
\subsection*{Our contributions}

We generalize the three-way strategy outlined above in three directions: Firstly, we adapt the strategy to accommodate a matrix $B_n$ of type (\ref{eq:Bn}) having $k$ levels of variation with a more complicated spectral distribution than previously studied. As we shall detail in Section \ref{sec:prelim}, \cite{fan2017eigenvalue} constructed the deterministic equivalent Stieltjes transform of $B_n$ building upon the solution to a system of $2k$ equations, which plays a key role in our construction of the limiting biases and variances. To the best of our knowledge, in past literature, the CLT has been established for sample covariance matrices with Stieltjes transforms determined by a system of no more than two equations \cite{bai2019central}, which corresponds to $k=1$ in our context for Gaussian random variables.

Secondly, we establish a CLT for the deviation of the linear spectral statistics calculated using the empirical spectral distribution of $B_n$, relative to its deterministic equivalent, instead of to the limiting spectral distribution \cite{baisilverstein2004,bai2019central}. In this formulation, we shall construct the bias $\mu_n$ and covariance $\sigma_n$ for the CLT using the deterministic equivalent matrix $\deteqv{B}_n$. This approach enables computation of $\mu_n$ and $\sigma_n$ based on finite sample quantities without assuming the convergence of the aspect ratio $n/N$ or the empirical spectral distributions of $\Sigma_{r}$ and $L_r$. Previous studies such as \cite{BANNA20202250,silverstein2012} have also established the CLT with respect to the deterministic equivalent, but have provided only limited analysis of the limiting bias. Specifically, they demonstrated that the bias asymptotically vanishes for complex Gaussian random variables. Our research focuses on real Gaussian random variables, for which the limiting bias persists.  

Thirdly, our proof of the tightness of $M_n^1$ simplifies significantly by utilizing the concentration results from \cite{Guionnet2000}, applied to the bulk distribution of $B_n$.

Finally, we propose a method to numerically evaluate the limiting bias and covariance in the central limit theorem, building upon an iterative algorithm proposed in \cite{fan2017eigenvalue} and the trapezoidal rule \cite{trefethen2014exponentially}.



\subsection*{Outline of paper}
Section \ref{sec:model.main.results} formalizes the model setup and main results, transitioning the discussion to the proof of the CLT for the scaled Stieltjes transform, as the first step of our three-step strategy. This section also specializes our result to the nested multivariate linear random effects models with up to $k$ levels of randomness. Section \ref{sec:prelim} presents essential preliminaries and tools. Section \ref{sec:conv.mn1} establishes the CLT for the centralized term $M_n^1$ (step 2 of the strategy). Section \ref{sec:bias} details the convergence of the deterministic term $M_n^2$ and computes the limiting bias for the final CLT (step 3 of the strategy). The remainder of the proof and additional details are deferred to the supplementary appendices.

\section{Model and Main results}\label{sec:model.main.results}
\subsection{Model setup}\label{subsec:model.setup}
In this paper, we make the following assumptions:
\begin{assumption}\label{assump}
\begin{enumerate}
\item $x_{j}=(x_{ji})_{i=1}^n\in\mathbbm{R}^n$ are i.i.d. standard Gaussian, for $j=1,\ldots, N$. 
    \item $c<n/N<C$. 
    \item $\Sigma_{r}$, $r=1,\ldots k$, are real symmetric. 
    \item $\sup_n\max_{j,r}l_{jr}<\infty$, $\sup_n\max_r\lVert \Sigma_{r}\rVert_2<\infty.$
\end{enumerate}
\end{assumption}
\noindent For simplicity of notations, define $L_r=\text{diag}(l_{1r},\ldots, l_{Nr})$, $r=1,\ldots,k$, $\mathrm{s}_L=\sup_n\max_r\lVert L_{r}\rVert_2$, $\mathrm{s}_{\Sigma}=\sup_n\max_r\lVert \Sigma_{r}^{1/2}\rVert_2$. Then, we can equivalently represent the matrix $B_n$ defined in (\ref{eq:Bn}) as
\begin{align}
    B_n&=\frac{1}{N}\sum_{r,s=1}^k\Sigma_r^{1/2}X_r^TL_rL_sX_s\Sigma_s^{1/2}\label{eq:Bn.1}\\
    &=\frac{1}{N}(\sum_{r=1}^kL_rX_r\Sigma_r^{1/2})^T(\sum_{r=1}^kL_rX_r\Sigma_r^{1/2}).\label{eq:Bn.2}
\end{align}
Above, $X_r\in\mathbbm{R}^{N\times n}$ are random matrices with i.i.d. Gaussian entries.
\subsection{Main result: Central limit theorem}
Define
\begin{align*}
   G_n(x)=n \left(F^{B_n}(x)-F^{\deteqv{B}_n}(x)\right),
\end{align*}
where $B_n$ takes the form (\ref{eq:Bn}), $\deteqv{B}_n$ is the deterministic equivalent of $B_n$, as defined in (\ref{eq:def.deteqv}).
 \begin{theorem}\label{thm:main}
       Under Assumption \ref{assump}, let $f_1,\ldots,f_l$ be functions on $\mathbbm{R}$ analytic on an open interval containing 
\begin{align}
  \left[0,k^2(1+\sqrt{C})^2\mathrm{s}_L^2\mathrm{s}_{\Sigma}^2\right].\label{eq:open.int}
\end{align}
Then
\begin{align*}
  \Lambda_n^{-1/2}  \left[  \left(\int f_1(x)dG_n(x),...,\int f_l(x)dG_n(x)\right)-\Gamma_n\right]
\end{align*}
converges weakly to $N(0,\text{Id}_l)$,with mean 
\begin{align}
    \Gamma_{n}[i]=-\frac{1}{2\pi i}\oint f_i(z)\mu_n(z)dz,\label{eq:main.mean}
\end{align}
and covariance 
\begin{align}\label{eq:main.cov}
    \Lambda_n[i,j]=-\frac{1}{2\pi^2}\oint\oint f_i(z_1)f_j(z_2)\sigma_n^2(z_1,z_2)dz_1dz_2,
\end{align}
where $\mu_n$ and ${\sigma_n}$ are respectively defined in (\ref{eq:mu}) and (\ref{eq:sigma}) in Lemma \ref{lem:main}. The contours in (\ref{eq:main.mean}) and (\ref{eq:main.cov}) are closed and taken in the positive direction in the complex plane, each containing the open interval (\ref{eq:open.int}), with the two contours in (\ref{eq:main.cov}) taken to be nonoverlapping.
    \end{theorem}
  For every $x$ in interval (\ref{eq:open.int}), by Cauchy's theorem,
    \begin{align*}
        f(x) = \frac{1}{2\pi i}\oint\frac{f(z)}{z-x}\mathrm{d}z.
    \end{align*}
    We can thus rewrite $G_n(f)=\int f(x)dG_n(x)$ in terms of $M_n(z)$ as follows
    \begin{align*}
        G_n(f) 
        &= \frac{1}{2\pi i}\int\oint\frac{f(z)}{z-x}n[F^{B_n}-F^{\deteqv{B}_n}](dx)dz\\
        &= \frac{1}{2\pi i}\oint f(z)dz\int\frac{1}{z-x}n[F^{B_n}-F^{\deteqv{B}_n}](dx)\\
        &= -\frac{1}{2\pi i}\oint f(z)\left[\text{Tr}(B_n-z)^{-1}-\text{Tr}(\deteqv{B}_n-z)^{-1}\right]dz\\
        &= -\frac{1}{2\pi i}\oint f(z)M_n(z)dz
    \end{align*}
Therefore, Theorem \ref{thm:main} immediately follows from the limiting results of $M_n$ formalized below.
\begin{lemma}\label{lem:main}
  Under Assumption \ref{assump}:\\
  (i) For any $v_0>0$, $\mathbbm{C}_0=\{z:|\text{Im }z|>v_0\}$,  $\{M_n(\cdot)\}$ forms a tight sequence on $\mathbbm{C}_0$.\\
  (ii) There exist $\mu_n(\cdot)$ and $\sigma_n(\cdot,\cdot)$ such that for any function $g$ analytic on an interval containing (\ref{eq:open.int}), we have that 
 \begin{align}\label{eq:oint.g}
    \frac{\oint g(z)(M_n(z)-\mu_n(z))dz}{\sqrt{2\oint\oint g(z_1)g(z_2)\sigma^2_n(z_1,z_2)dz_1dz_2}}
\end{align}
  converge weakly to a standard complex Gaussian random variable, where
 \begin{align}
      \mu_n(z)&=\deteqv{d}_{n0}(z)+n/N\sum_{r=1}^k\nu_r(z){\sum_{s=1}^k}h_{N+1}^{rs}(z,z)\Xi_0^s(z),\label{eq:mu}\\
      \sigma_n^2(z_1,{z}_2)&=\frac{\partial^2}{\partial z_2\partial z_1}\sum_{r=1}^k\left(N^{-2}{\sum_{j=1}^N}l_{rj}^2\deteqv{b}_j(z_1)\deteqv{b}_j(z_2)\deteqv{w}_{jr}(z_1,z_2)\right)\label{eq:sigma}.
  \end{align}
Above, the quantities $\deteqv{d}_{n0}$, $\nu_r$, $\Xi_0^a$ are defined in (\ref{eq:dnr.bias.deteqv}), (\ref{eq:nu.r}) and (\ref{eq:xi.0.a}). The quantities $h_j^{ab}$, $\deteqv{b}_j$, $\deteqv{w}_{jr}$ are defined in (\ref{eq:xi.h.Lam}), Lemma \ref{lem:Mn1.prelim} and (\ref{eq:deteqv.w}). The contours in (\ref{eq:oint.g}) are closed and taken in the positive direction in the complex plane, each containing the open interval (\ref{eq:open.int}), with the two contours in the denominator taken to be nonoverlapping.
\end{lemma}
\subsection{{Overview of proof}}
As outlined in the introduction, our proof generalizes a three-step strategy developed by Bai and Silverstein \cite{baisilverstein2004}. This approach consists of three main components: 
\begin{enumerate}
    \item[(a)] finite-dimensional CLT of the centralized term $M_n^1(z)$, as formalized in equation (\ref{Mn1:finite.sum});
    \item[(b)] tightness of $M_n^1(z)$ as in (\ref{eq:tightness.def});
    \item[(c)] convergence of the deterministic term $M_n^2(z)$ as in (\ref{eq:mn2.conv.def}).
\end{enumerate}

In addressing the first component, we establish the CLT by applying the martingale central limit theorem. Initially, we write $M_n^1(z)$ as a series of martingale differences and demonstrate that this sequence exhibits light-tailed behavior, thereby satisfying condition (ii) of Theorem \ref{thm:martingale}. The primary distinction in our approach concerns the calculation of limiting covariance. This task is challenging due to the complexity of matrix $B_n$ defined in (\ref{eq:Bn}). Specifically, our calculations are reduced to evaluating the limits of the normalized traces:
\begin{align}\label{eq:norm.trace}
    \frac{1}{N}\text{Tr}(\E_jD_j^{-1}(z_1)\Sigma_r\E_jD_j^{-1}(z_2)\Sigma_s), \quad {\forall r,s=1,\ldots,k}.
\end{align}
Above, $D_j(z)=B_n-N^{-1}T_j^{1/2}x_jx_j^TT_j^{1/2}-zI$, and $\E_j$ denotes the conditional expectation with respect to the $\sigma$-field generated by $x_1,\ldots,x_j$. This deviates from the formulations in \cite{baisilverstein2004} and \cite{bai2019central}, where the focus was on:
\begin{align*}
    \frac{1}{N}\text{Tr}(\E_jD_j^{-1}(z_1)\Sigma \E_jD_j^{-1}(z_2)\Sigma),
\end{align*}
with $B_n=\frac{1}{N}\sum_{i=1}^N\Sigma^{1/2}x_ix_i^T\Sigma^{1/2}$ in \cite{baisilverstein2004}, and $B_n=\frac{1}{N}\sum_{i=1}^Nl_i^2\Sigma^{1/2}x_ix_i^T\Sigma^{1/2}$ in \cite{bai2019central}. 

As a result, we derive asymptotic equivalents for the normalized traces as delineated in equation (\ref{eq:norm.trace}), where these equivalents are defined such that their differences with the original terms vanish in probability. These equivalents result from solving a system of $k$ equations, each composed of elements calculated from $L_r$ and $\Sigma_r$, $r=1,\ldots,k$. Specifically, the system incorporates terms $\deteqv{g}_2^{(r)}$ for $r = 1, \ldots, k$, which are also critical in forming the equations that determine the deterministic equivalent of $B_n$, formally defined in (\ref{eq:def.deteqv}). These terms can be computed using the iterative algorithm proposed in \cite{fan2017eigenvalue}. In contrast, earlier works \cite{baisilverstein2004,bai2019central} addressed only the case where $k=1$, thus eliminating the need to solve a system of equations.

Our proof of tightness for the sequence $\{M_n^1(z)\}$ is distinct from the approaches described in \cite{baisilverstein2004,bai2019central}. From the fact that Gaussian random variables satisfy logarithmic Sobolev inequalities, we apply  Theorem 1.1 from \cite{Guionnet2000} to establish the concentration of the spectral measure of $B_n$. Based on this result, it is straightforward to verify tightness.

Analogous to our calculation of the limiting covariance, the limiting bias--which corresponds to the limit of $M_n^2(z)$--is derived from components that solve systems of $k$ equations. Each of these systems comprises terms calculated from $L_r$ and $\Sigma_r$, $r=1,\ldots,k$.

\subsection{Numerical evaluation of bias and covariance}\label{subsec:numeric}
The formulas for bias and covariance presented in Theorem \ref{thm:main} and Lemma \ref{lem:main} may initially appear computationally challenging. In this section we show that these quantities can be accurately evaluated numerically. An illustrative example is presented in Section \ref{subsubsec:full.sib.sim} under the full sibling model.
\subsubsection{Evaluation of $\mu_n$ and $\sigma_n$}
From equations (\ref{eq:mu}) and (\ref{eq:sigma}), it turns out that $\mu_n$ and $\sigma_n$ are determined by the three quantities $L_r$, $\Sigma_r$ for $r=1,\dots,k$, and $\Tilde{b}_j$ for $j=1,\ldots,N$. According to Lemma \ref{lem:Mn1.prelim}, $\Tilde{b}_j$ depends on both $L_r$ and $\deteqv{g}_2^{(r)}$ for each $r=1,\dots,k$. Given that $L_r$ and $\Sigma_r$ are known by design, the challenge lies in numerically evaluating $\deteqv{g}_2^{(r)}$ for $r=1,\dots,k$, which are solutions to the system of equations presented in (\ref{eq:m0}). To address this, we use an iterative algorithm proposed by \cite{fan2017eigenvalue}, formalized in the subsequent lemma.

\begin{lemma}[Paraphrased from Theorem 1.4 in \cite{fan2017eigenvalue}]
    For each $z\in\mathbbm{C}^+$, the values $\deteqv{g}_i^{(r)}$, $i=1,2$, $r=1,\ldots k$ in the system of equations (\ref{eq:m0}) are the limits, as $t\rightarrow\infty$, of
the iterative procedure which arbitrarily initializes $\deteqv{g}_{1,0}^{(1)},\ldots, \deteqv{g}_{1,0}^{(k)}\in{\mathbbm{C}^+}$ and iteratively computes (for $t=0,1,2,\ldots$) $\deteqv{g}_{2,t}^{(r)}$ from $\deteqv{g}_{1,t}^{(r)}$ using (\ref{eq:m0}) and $\deteqv{g}_{1,t+1}^{(r)}$ from $\deteqv{g}_{2,t}^{(r)}$ using (\ref{eq:m0}).
\end{lemma}

\subsubsection{Evaluation of $\Gamma_n$ and $\Lambda_n$}
Having obtained accurate estimates of $\mu_n$ and $\sigma_n$, we evaluate the contour integrals for $\Gamma_n$  and $\Lambda_n$ given in (\ref{eq:main.mean}) and (\ref{eq:main.cov}) using the trapezoidal rule, a widely-employed numerical computation technique. We shall follow a formulation presented in the survey paper \cite{trefethen2014exponentially}. Informally, if we take the contours in (\ref{eq:main.mean}) and (\ref{eq:main.cov}) to be circular containing the interval (\ref{eq:open.int}), then the resulting approximation, assuming known values of $\mu_n$ and $\sigma_n$, achieves exponential accuracy relative to the number of function evaluations, $R$. This result is a direct corollary of Theorem \ref{lem:trefethen2014} formalized below.
    \begin{theorem}[Theorem 2.2 in \cite{trefethen2014exponentially}]\label{lem:trefethen2014}
  {\small  
 If $u(z)$ is analytic and satisfies  $|u(z)|\leq M$ in the annulus $r^{-1}<|z|<r$, for $r>1$, define
    \begin{align*}
        I=\int_{|z|=1}u(z)dz,\quad \textrm{and} \quad I_R=\frac{2\pi i}{
R}{\sum_{k=1}^R}z_ku(z_k),
    \end{align*}
    where 
    $z_k=\exp{2\pi ik/R}$. Then for any $R\geq 1$,
    \begin{align*}
       \left|I_R-I\right|\leq \frac{4\pi M}{r^R-1}.
 \end{align*}}
\end{theorem}
\subsection{Application to multivariate linear random-effects models}\label{subsec:multi.random}
Genetic covariance matrices are central to the study of
multivariate quantitative genetics, providing insights into how traits genetically covary with each other
due to pleiotropy and linkage disequilibrium. Accurate estimation of these matrices in the context of linear random effects models is fundamental to predicting organism response to
selection and understanding evolutionary limits. Nested random-effects models, in particular, have gained prominence for their applicability in improving selective breeding programs, notably within the poultry industry.

Consider a nested random-effects model with $n_s$ samples and $p$ traits measured for each sample. The observed traits $Y\in\mathbbm{R}^{n_s\times p}$ are modeled by a $k$-level Gaussian random-effects model
\begin{align}\label{eq:Y}
    Y=U_1\alpha_1+\ldots+U_k\alpha_k.
\end{align}
Above, each $\alpha_r\in\mathbbm{R}^{F_r\times p}$ represents the $r$-th level of random effect, with i.i.d. rows each distributed as $N(0,\Sigma_r)$. Each $U_r$ denotes the deterministic membership matrix that assigns each individual to the corresponding group of random effect in the $r$-th level. For instance, in the first level of randomness, if the $i$th individual is assigned to the $j$th group, then $U_1[i,j]=1$, and $U_1[i,l]=0$ when $l\neq j$. By the nested nature of the model, these membership matrices satisfy
\begin{align*}
&F_1\leq ...\leq F_k\leq n_s,\\
    &\text{col}(U_1)\subset \ldots \subset\text{col}(U_k).
\end{align*}
For simplicity, here we omit possible fixed effects. In this context, the genetic covariance of interest is a known linear combination $\Sigma_1,\ldots,\Sigma_k$ \cite{falconer1996introduction}. We are interested in estimating its eigenvalue distribution, which contains important information on evolutionary dynamics, e.g. null space dimension, largest eigenvalues, etc. To achieve this goal, one possible approach is to formulate parametric assumptions on the eigenvalue distributions of each $\Sigma_r$, and obtain method of moments estimators of these parameters based on the sum of squares matrix
\begin{align}\label{eq:nested.ori}
    B_p=\frac{1}{F_1}Y^TU_1(U_1^TU_1)^{-1}U_1^TY.
\end{align}
In particular, if we model the eigenvalue distributions of $\Sigma_1,\ldots,\Sigma_k$ as parameterized by $\tau_1,\ldots,\tau_l$ for a positive integer $l$, we can obtain method of moments estimators from a mapping between the parameters and moments of $B_p$ in the form of
\begin{align}\label{eq:mom.formula}
    (\hat{\tau}_1,\ldots,\hat{\tau}_l)=\mathcal{F}(\text{Tr}B_p,\ldots,\text{Tr}B_p^l).
\end{align}
We refer to a working manuscript \cite{xie2024} for the detailed expression of the mapping $\mathcal{F}$.

Here we show how $B_p$ may be written in the form (\ref{eq:Bn}) with $(n,N)=(p,F_1)$ so that Theorem \ref{thm:main} may be applied. By nature of membership matrices, we have that $U_1^TU_1$ is diagonal. Under the nested model, we further have that $U_1^TU_r$ has full row rank and that $U_1^TU_rU_r^TU_1$ is diagonal and positive definite. Additionally, considering a matrix $A\in\mathbbm{R}^{m_1\times m_2}$, $m_1<m_2$ and a standard Gaussian vector $x\in\mathbbm{R}^{m_2}$, there exists a standard Gaussian vector $z\in\mathbbm{R}^{m_1}$ such that  $Ax\deq (AA^T)^{1/2}z$.
As a result, by Gaussianity of the random effects, 
\begin{align*}
    (U_1^TU_1)^{-1/2}U_1^TU_r\alpha_r\deq L_rX_r\Sigma_r^{1/2},
\end{align*}
where $L_r=(U_1^TU_1)^{-1/2}(U_1^TU_rU_r^TU_1)^{1/2}$ is diagonal and positive definite, and $X_r\in\mathbbm{R}^{F_1\times p}$ has i.i.d. standard Gaussian entries. Thus, we can express the sum of squares matrix as
\begin{align}\label{eq:nested.sim.1}
    B_p\deq\frac{1}{F_1}{\sum_{r,s=1}^k\Sigma_{r}^{1/2}}X_{r}^TL_rL_sX_{s}\Sigma_{s}^{1/2}.
\end{align}
Let $T_{f}=\sum_{r=1}^kl_{rf}^2\Sigma_r$ and denote the $f$th row of $X_r$ by $x_{rf}$, then by the independence of the random effects there further exist i.i.d. standard Gaussian vectors 
$x_f\in\mathbbm{R}^p$ such that $\sum_{r=1}^kl_{rf}\Sigma_r^{1/2}x_{rf}\deq T_{f}^{1/2}x_f$. Therefore, we can alternatively express $B_p$ as
\begin{align}\label{eq:nested.sim.2}
    B_p&\deq\frac{1}{F_1}\sum_{f=1}^{F_1} T_{f}^{1/2}x_fx_f^TT_{f}^{1/2}, \quad T_{f}=\sum_{r=1}^kl_{rf}^2\Sigma_r.
\end{align}
Thus, from a central limit theorem on the LSS of $B_p$ such as the moments $(\text{Tr}B_p,\ldots ,\text{Tr}B_p^l)$, we can apply the Delta method to recover limiting results on the method of moments estimators from the mapping defined in (\ref{eq:mom.formula}). 
\subsubsection{Simulations under the full sibling design}\label{subsubsec:full.sib.sim}

Consider a full sibling design (\ref{eq:full.sib.design}) with $F=500$ families and $p=500$ gene expression traits measured for each individual. In each family, let the number of siblings $J_i$ be either 1 or 2 with equal probability. We model the eigenvalues of $\Sigma_A$, the covariance matrix of the family effects, to be exponentially decreasing, with $\sigma_i=\tau_1e^{-\tau_2 i}$, $i=1,\ldots,p$. For simplicity, take $\Sigma_E$, the covariance matrix of the individual effects, to be $\Sigma_E=\tau_e\text{Id}$. We take $\tau_1=1$, $\tau_2=0.3$ and $\tau_e=1$.

Based on the relationship outlined in (\ref{eq:mom.formula}), we derive the method of moments estimators $\hat{\tau}_1$, $\hat{\tau}_2$, and $\hat{\tau}_e$. These estimators are computed from the first and second moments of $B_p$ as specified in (\ref{eq:Bn.full.sib}), and the first moment of $D_p=(n_s-F)^{-1}Y^T(\text{Id}-\pi)Y\deq (n_s-F)^{-1}\sum_{i=1}^{n_s-F}\Sigma_E^{1/2}z_iz_i^T\Sigma_E^{1/2}$, where $\pi$ is defined in the same equation, and $z_i\in\mathbbm{R}^{p}$ are i.i.d. standard Gaussian vectors. \footnote{Obtaining the method of moments estimators directly from the first three moments of $B_p$ can be slightly more complicated computationally, resulting in longer runtime, and perhaps has higher variance than the choice made here.} We repeat the experiment $R=1000$ times from data generation. Figure \ref{fig:enter-label} shows histograms of the estimators. Denote the estimators by $\{\hat{\tau}^{(j)}=(\hat{\tau}_1^{(j)},\hat{\tau}_2^{(j)},\hat{\tau}_e^{(j)})\}_{j=1}^{R}$. The empirical biases $\Bar{\hat{\tau}}-\tau$ and standard deviations $(R-1)^{-1/2}\sqrt{\sum_{j=1}^{R}({\hat{\tau}^{(j)}}^2-\Bar{\hat{\tau}})^2}$ of the estimators are evaluated and presented in the first row of Table \ref{tab:table1}. To obtain the theoretical biases and standard deviations, we first apply Theorem \ref{thm:main} to the moments of $B_p$, $D_p$, which we denote by $\hat{\alpha}$. In other words, there exists $\Lambda_n$, $\mu_n'$ such that $\Lambda_n'^{-1/2}\left[n(\hat{\alpha}-\deteqv{\alpha})-{\mu}_n'\right]\rightarrow N(0,\Id_3)$, where $\deteqv{\alpha}$ is computed from the deterministic equivalent matrices $\deteqv{B}_p$, $\deteqv{D}_p$. Applying the Delta method for the mapping $\mathcal{F}$ defined in (\ref{eq:mom.formula}), we further have $
    \Lambda_n^{-1/2}\left[n(\hat{\tau}-\tau)-{\mu}_n\right]\rightarrow N(0,\Id_3),$
where $\tau=(\tau_1,\tau_2,\tau_e)$, $\Lambda_n=J_{\mathcal{F}}(\deteqv{\alpha})\Lambda_n'J_{\mathcal{F}}(\deteqv{\alpha})^T$, ${\mu}_n=J_{\mathcal{F}}(\deteqv{\alpha})\mu_n'+n({\mathcal{F}}(\deteqv{\alpha})-\tau)$, with $J_{\mathcal{F}}(\deteqv{\alpha})$ representing the Jacobian of mapping ${\mathcal{F}}$ evaluated at $\deteqv{\alpha}$. The resulting expressions $\Lambda_n$ and $\mu_n$ are further numerically evaluated using the techniques introduced in Section \ref{subsec:numeric}. Upon comparison, it is evident that the empirical and theoretical values closely align.

\begin{figure}
    \centering
    \includegraphics[trim={0cm 0cm 0cm 1cm},clip,width=\linewidth]{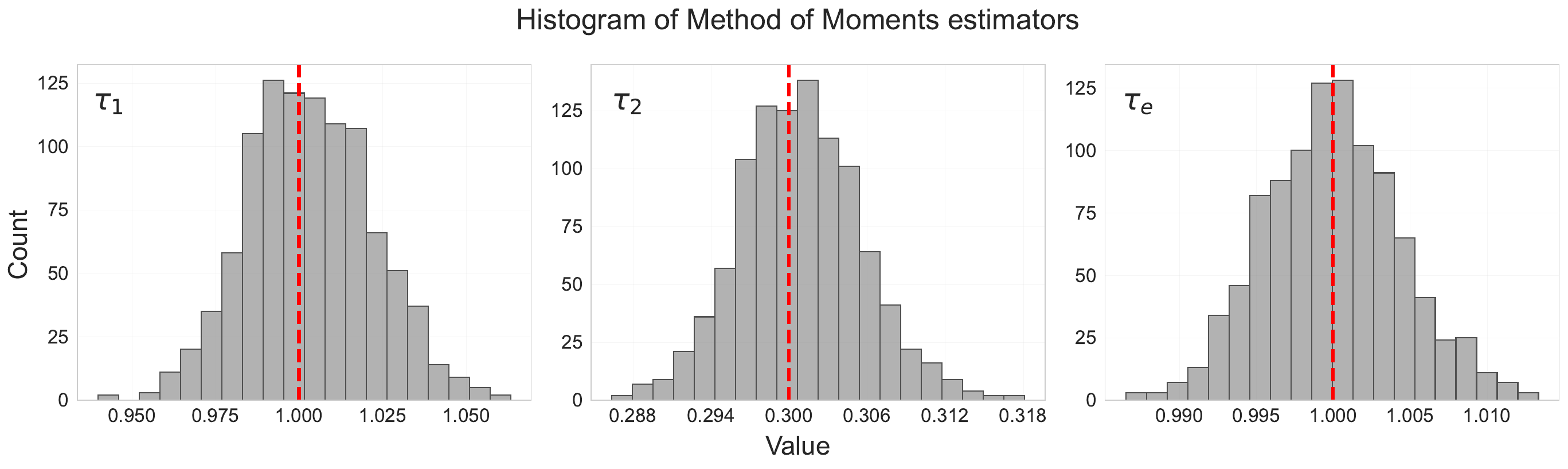}
    \caption{Histograms of the method of moments estimators for $\tau_1$, $\tau_2$ and $\tau_e$, which parameterize the eigenvalues of $\Sigma_A$ and $\Sigma_E$.}
    \label{fig:enter-label}
\end{figure}
\begin{table}[htb]
  \begin{center}
    \caption{Biases and standard deviations of method of moments estimators.
    }
    \label{tab:table1}
    \vspace{0.1cm}
    \begin{tabular}{S S S S S S S}
    \toprule
      \multicolumn{1}{c }{}& \multicolumn{3}{c}{Bias} &  \multicolumn{3}{c}{2 Standard Deviation}  \\ 
      \cmidrule(r){2-4}\cmidrule(r){5-7}
      \multicolumn{1}{c}{} & \multicolumn{1}{c}{$\tau_1$} & $\text{$\tau_2$}$& $\text{$\tau_e$}$& \multicolumn{1}{c}{$\tau_1$} & $\text{$\tau_2$}$& $\text{$\tau_e$}$\\ 
      \hline
      $\text{Empirical}$ &0.0033 &0.0009& -0.0001&0.0376& 0.0095& 0.0088\\ 
      $\text{Theoretical}$ &0.0032  &0.0009 & -0.0003&0.0380& 0.0084& 0.0085\\
      \bottomrule
    \end{tabular}
  \end{center}
\end{table}
\section{Preliminaries and tools}\label{sec:prelim}
\subsection{Deterministic equivalent spectral distribution}
To characterize the empirical spectral distribution $F^{B_n}$ of $B_n$, we study its deterministic equivalent matrix, $\deteqv{B}_n$. Here, equivalence is in the sense that $F^{B_n}-F^{\deteqv{B}_n}\rightarrow 0$ weakly almost surely. Modifying the construction proposed in \cite{fan2017eigenvalue}, define $\deteqv{B}_n$ as 
\begin{align}\label{eq:def.deteqv}
    \deteqv{B}_n=-z\sum_{r=1}^k\deteqv{g}_1^{(r)}\Sigma_r.
\end{align}
Here, $\deteqv{g}_1^{(r)}$, $r=1,\ldots,k$, are determined by the following system of equations:
\begin{align}
\begin{cases}
    z\deteqv{g}_1^{(r)}&=-\frac{1}{N}\text{Tr}\left(({\sum_{s=1}^k}\deteqv{g}_2^{(s)}L_s^2+\text{Id})^{-1}L_r^2\right),r=1,\ldots,k\\
    z\deteqv{g}_2^{(r)}&=-\frac{1}{N}\text{Tr}\left(({\sum_{s=1}^k}\deteqv{g}_1^{(s)}\Sigma_s+\text{Id})^{-1}\Sigma_r\right),r=1,\ldots,k.
    \end{cases}\label{eq:m0}
\end{align}
The detailed construction is deferred to Section \ref{sec:comp.stieltjes} in the supplementary appendices. 
Thus, the Stieltjes transform $\deteqv{m}_n$ may be equivalently expressed as:
\begin{align}
\deteqv{m}_n&=\frac{1}{n}\text{Tr}(\deteqv{B}_n-z\text{Id})^{-1},\nonumber\\
z\deteqv{m}_n&=-\frac{1}{n}\text{Tr}\left(\sum_{r=1}^k\deteqv{g}_1^{(r)}\Sigma_r+\text{Id} \right)^{-1},\label{eq:mn.deteqv.g1}\\
    z\deteqv{m}_n&=-1-\frac{N}{n}z\sum_{r=1}^k \deteqv{g}_1^{(r)}\deteqv{g}_2^{(r)},\nonumber\\
    z\deteqv{m}_n&=-1+\frac{N}{n}-\frac{1}{n}\text{Tr}\left(1+\sum_{r=1}^kL_r^2\deteqv{g}_2^{(r)}\right)^{-1}.\nonumber
\end{align}
Similarly, for $B_n$, define:
\begin{align}
\begin{cases}\label{eq:random.stiel}
m_n&=\frac{1}{n}\text{Tr}(B_n-z\text{Id})^{-1},\\
z{g}_1^{(r)}&=-\frac{1}{N}\text{Tr}\left(({\sum_{s=1}^k}{g}_2^{(s)}L_s^2+\text{Id})^{-1}L_r^2\right),r=1\ldots,k\\
{g}_2^{(r)}&=\frac{1}{N}\text{Tr}\left((B_n-z\text{Id})^{-1}\Sigma_r\right),r=1\ldots,k.
\end{cases}
\end{align}
Adapting the proof of Theorem 1.2 from \cite{fan2017eigenvalue}, we establish the following equivalence between $B_n$ and $\deteqv{B}_n$.
\begin{lemma}\label{lem:det.equiv}
    Under Assumption \ref{assump}, as $n, N\rightarrow\infty$, for each $z\in\mathbbm{C}^+$, $i=1,2$, and $r=1,\ldots,k$, 
    \begin{align*}
    m_{n}-\deteqv{m}_{n}&\rightarrow 0,\\
       {g}_i^{(r)}-\deteqv{g}_i^{(r)}&\rightarrow 0
    \end{align*}
    pointwise almost surely.
\end{lemma}
\subsection{Truncation of $x_{ji}$}
\begin{lemma}\label{lem:trunc}
    Under Assumption \ref{assump}, there exist $\delta_n$ and $\hat{x}_{ji}$, $j\in [N]$, $i\in[n]$, such that 
    \begin{enumerate}
        \item $|\hat{x}_{ji}|<\delta_n\sqrt{n}$
        \item $\E [\hat{x}_{ji}]=0$, $\E |\hat{x}_{ji}|^2=1$, $\E |\hat{x}_{ji}|^4=3+o(1)$
        \item $\delta_n\rightarrow 0$, $n\delta_n^4\rightarrow\infty$
        \item Let $\hat{x}_j=(\hat{x}_{ji})_{i=1}^n$, $\hat{B}_n=\frac{1}{N}\sum_{j=1}^NT_{j}^{1/2}\hat{x}_j\hat{x}_j^TT_{j}^{1/2}$, $\hat{G}_n=n \left(F^{\hat{B}_n}(x)-F^{\deteqv{B}_n}(x)\right)$.
        For any function $f$ on $\mathbbm{R}$ analytic on an open interval containing (\ref{eq:open.int}),
        \begin{align*}
            \int fdG_n(x)=\int fd\hat{G}_n(x)+o_p(1),
        \end{align*}
        where $o_p(1)$ represents convergence in probability to 0.
    \end{enumerate}
\end{lemma}
Building upon the above lemma, without loss of generality, we shall subsequently work under the following set of assumptions:
\begin{assumption}\label{assump_trunc}
\begin{enumerate}
\item[1'] $x_{ji}$, $j\in [N]$, $i\in[n]$, are i.i.d. random variables satisfying properties (1-3) in Lemma \ref{lem:trunc}.
    \item[2-4] Remain consistent with Assumption \ref{assump}.
\end{enumerate}
\end{assumption}
After truncation, we derive a crucial concentration result that forms the basis for our limiting covariance calculation.
\begin{lemma}[Concentration]\label{lem:concen}
   Under Assumption \ref{assump_trunc}, for $x_1=[x_{11},...,x_{1n}]^T$ and non-random $n\times n$ matrices $B_l$, $l=1,...,q$, we have that for $q\geq 2$,
\begin{align}\label{eq:martingale.concentr}
    \E \left|{\prod_{l=1}^q}(N^{-1}x_1^TB_lx_1-N^{-1}\text{Tr}B_l)\right|\leq KN^{-1}\delta_n^{2q-4}\prod_{l=1}^q\lVert B_l\rVert,
\end{align}
where $K$ is some constant that depends on $q$. 
\end{lemma}
\subsection{Concentration of extreme eigenvalues}
Recall that $B_n$ defined in (\ref{eq:Bn}) can be equivalently represented as (\ref{eq:Bn.2}).
Under Assumption 3.3, applying Theorem 3.1 in \cite{baisil1998}, we have that for each $r=1,\ldots, k$,
\begin{align*}
  N^{-1/2} \lVert X_r\rVert_2 \leq 1+\sqrt{C}\quad a.s.
\end{align*}
As a result, we get
\begin{align}
    \sqrt{\lambda_{\text{max}}(B_n)}
    &=\frac{1}{\sqrt{N}}\lVert \sum_{r=1}^kL_rX_r\Sigma_r^{1/2}\rVert_2\\
    &\leq \frac{1}{\sqrt{N}}\sum_{r=1}^k\lVert L_r \rVert_2 \lVert X_r\rVert_2 \lVert\Sigma_r^{1/2} \rVert_2 \label{eq:eigen.max.bd}\\
     &\leq (1+\sqrt{C})ks_Ls_{\Sigma}\text{ a.s.},
\end{align}
where $s_L$ and $s_{\Sigma}$ are defined in Section \ref{subsec:model.setup}. From this, the upper bound on the largest eigenvalue of $B_n$ is given by
\begin{align}
    \lambda_{\text{max}}(B_n)
    &\leq (1+\sqrt{C})^2k^2s_L^2s_{\Sigma}^2\text{ a.s.}.\label{eq:concen.bn}
\end{align}
Furthermore, we establish the following concentration result.
\begin{lemma}[Concentration of extreme eigenvalues of $B_n$]\label{lemma: concen.eig.sofs}
    Define $C_u:= (1+\sqrt{C})^2k^2s_L^2s_{\Sigma}^2$. Under Assumption \ref{assump_trunc}, for any $\delta>0$, positive integer $k>0$,
    \begin{align*}
        \mathbbm{P}(\lambda_{\max}(B_n)\geq C_u+\delta)\leq Cp^{-k}.
    \end{align*}
\end{lemma}
\begin{proof}
    Applying property (1.9a) of \cite{baisilverstein2004}, we have that for any $k$, $\eta_k>(1+\sqrt{C})^2$, 
    \begin{align*}
        \mathbbm{P}\left(\lambda_{\max}\left(\frac{1}{N}X_rX_r^T\right)>\eta_k \text{ for any }1\leq r\leq k\right)=o(p^{-l}).
    \end{align*}
    Taken together with equation (\ref{eq:eigen.max.bd}), we conclude the proof.
\end{proof}
\section{Convergence of centralized sequence $M_n^1$}\label{sec:conv.mn1}
Let us write $M_n^1$ as the scaled centralized Stieltjes transform
\begin{align*}
    M_n^1=\text{Tr}(B_n-z)^{-1}-\E \text{Tr}(B_n-z)^{-1}=n(m_n-\E [m_n]).
\end{align*}
Let $r_j=N^{-1/2}T_{j}^{1/2}x_j$, then we can express $B_n$ as
\begin{align*}
    B_n=\sum_{j=1}^Nr_jr_j^T.
\end{align*}
Introduce 
\begin{align}
    &D(z)=B_n-zI,\quad D_j(z)=D(z)-r_jr_j^T,\quad B_{nj}=B_n-r_jr_j^T,\label{eq:d.dj}\\
   &\epsilon_j(z)=r_j^TD_j^{-1}(z)r_j-N^{-1}\text{Tr}(D_j^{-1}(z)T_j),\quad \gamma_j(z)=r_j^TD_j^{-2}(z)r_j-N^{-1}\text{Tr}(D_j^{-2}(z)T_j)\label{eq:eps}\\
    &\beta_j(z)=\frac{1}{1+r_j^TD_j^{-1}(z)r_j},\quad \Tilde{\beta}_j(z)=\frac{1}{1+N^{-1}\text{Tr}(D_j^{-1}(z)T_j)},\\
    &b_j(z)=\frac{1}{1+N^{-1}\E [\text{Tr}(D_j^{-1}(z)T_j)]},\quad
    \psi_j(z)=\frac{1}{1+N^{-1}\E [\text{Tr}(D^{-1}(z)T_j)]}\label{eq:def.psi}\\
    &R_j(z)=zI-\frac{1}{N}\sum_{i\neq j}\psi_i(z)T_i,\quad R=\frac{1}{N}\sum_{j=1}^N\psi_j(z)T_j-zI.\label{eq:def.R}
\end{align}
It is straightforward to verify that $\gamma_j(z)=d\epsilon_j(z)/dz$. From the spectral decomposition of $D(z)$ and $D_j(z)$, it follows that
\begin{align}\label{eq:d.dj.opnorm}
    \lVert D(z)\rVert,  \lVert D_j(z)\rVert\leq v^{-1},
\end{align}
where $v=\text{Im}z$.
Based on the concentration result established in Lemma \ref{lem:concen}, we further have
\begin{lemma}\label{lem:Mn1.prelim}
    Under Assumption \ref{assump_trunc}, for any $z\in\mathbbm{C}^+$ with $\textrm{Im}z>v$, $v<1$, the following statements hold true:
    \begin{enumerate}
        \item Boundedness:
        \begin{gather}
            |\beta_j|,|\Tilde{\beta}_j|,|b_j|,|\psi_j|\leq |z|v^{-1},\quad |b_j(z)-\psi_j(z)|\leq CN^{-1}|z|^4v^{-7}\label{eq:bd.beta}\\
            \lVert R_j^{-1}(z)\rVert,\lVert R^{-1}(z)\rVert\leq v^{-1}\label{eq:norm.R.inv}
        \end{gather}
        \item Limit of $b_j(z)$, $\psi_j(z)$: let $\deteqv{b}_j(z)=(1+\sum_{r}l_{rj}^2\deteqv{g}_2^{(r)}(z))^{-1}$, then
        \begin{align}\label{eq:bk.lim}
   \left|b_j(z)-\deteqv{b}_j(z)\right|=o(1),\quad\left|\psi_j(z)-\deteqv{b}_j(z)\right|=o(1).
\end{align}
        \item  Concentration of $|\Tilde{\beta}_j(z)-b_j(z)|$:
        \begin{align}\label{eq:beta.tilde.b}
            \E |\Tilde{\beta}_j(z)-b_j(z)|^2\leq C|z|^4v^{-6}N^{-1}
        \end{align}
        \item Concentration of $\epsilon_j(z)$, $\gamma_j(z)$: For $q\geq 1$, 
        \begin{align}
            &\E |\epsilon_j(z)|^{2q}\leq KN^{-1}\delta_n^{4q-4}v^{-2q}\\
            &\E |\gamma_j(z)|^{2q}\leq KN^{-1}\delta_n^{4q-4}v^{-4q}.
        \end{align}
    \end{enumerate}
\end{lemma}
\subsection{Finite dimensional convergence}\label{subsec: mn1.finite}
For any $v_0>0$, set $C_0=\{z:|\text{Im}z|>v_0\}$.
In this section, we prove that for any finite set of $z_1,\ldots z_l\in C_0$, $\alpha_1,\ldots,\alpha_l\in\mathbbm{C}$, there exists deterministic $\check{\sigma}_n$, such that
\begin{align}\label{Mn1:finite.sum}
   \check{\sigma}_n^{-1} \sum_{\nu=1}^l\alpha_{\nu} M_n^1(z_{\nu})
\end{align}
converges to a standard complex Gaussian random variable. We shall accomplish this using the Martingale Central Limit Theorem \ref{thm:martingale}. In this section, we write $M_n^1(z)$ as the sum of martingale differences and verify condition (ii) in Theorem \ref{thm:martingale}. In the next section, we compute the covariance of the finite sum (\ref{Mn1:finite.sum}), and verify condition (i) in Theorem \ref{thm:martingale}.

Let $\E_0(\cdot)$ denote expectation and $\E_j(\cdot)$ denote conditional expectation with respect to the $\sigma$-field generated by $r_1$,...,$r_j$.  From the Woodbury formula, we have
\begin{align}
    D^{-1}&=(D_j(z)+r_jr_j^T)^{-1}\nonumber\\
    &=D_j^{-1}(z)-\frac{D_j^{-1}(z)r_jr_j^TD_j^{-1}(z)}{1+r_j^TD_j^{-1}(z)r_j}\nonumber\\
    &=D_j^{-1}(z)-\beta_j(z)D_j^{-1}(z)r_jr_j^TD_j^{-1}(z).\label{eq:D.Dj}
    \end{align}
Based on this property, and the definitions of $D$, $D_j$, $\beta_j$ and $\gamma_j$, we get
\begin{equation}\label{eq:tele}
\begin{aligned}
n(m_n-\E [m_n])
&=\text{Tr}[D^{-1}(z)-\E_0D^{-1}(z)]\\
&=\sum_{j=1}^N\text{Tr}\E_jD^{-1}(z)-\text{Tr}\E_{j-1}D^{-1}(z)\\
&=\sum_{j=1}^N(\E_j-\E_{j-1})\text{Tr}[D^{-1}(z)-D_j^{-1}(z)]\\
&=-\sum_{j=1}^N(\E_j-\E_{j-1})\beta_j(z)r_j^TD_j^{-2}(z)r_j,\\
&=-\sum_{j=1}^N(\E_j-\E_{j-1})\beta_j(z)\gamma_j+\sum_{j=1}^N(\E_j-\E_{j-1})\beta_j(z)N^{-1}\text{Tr}(D_j^{-2}(z)T_j)\\
&\triangleq I_1+I_2.
\end{aligned}
 \end{equation}
By definition, we immediately have
\begin{align}\label{eq:beta.prop1}
    \beta_j(z)=\Tilde{\beta}_j(z)-\beta_j(z)\Tilde{\beta}_j(z)\epsilon_j(z).
\end{align}
Applying this property again on the $\beta_j(z)$ on the right hand side, then
\begin{align}
    \beta_j(z)
    &=\Tilde{\beta}_j(z)-\left(\Tilde{\beta}_j(z)-\beta_j(z)\Tilde{\beta}_j(z)\epsilon_j(z)\right)\Tilde{\beta}_j(z)\epsilon_j(z),\\
    &=\Tilde{\beta}_j(z)-\Tilde{\beta}_j
    ^2(z)\epsilon_j(z)+\Tilde{\beta}_j
    ^2(z)\beta_j(z)\epsilon_j^2(z).\label{eq:beta.prop2}
\end{align}
Applying (\ref{eq:beta.prop1}) to $I_1$ gives
\begin{align*}
    I_1=-\sum_{j=1}^N(\E_j-\E_{j-1})\gamma_j\Tilde{\beta}_j(z)+\sum_{j=1}^N(\E_j-\E_{j-1})\gamma_j\beta_j(z)\Tilde{\beta}_j(z)\epsilon_j(z).
\end{align*}
Note that 
\begin{align}
    &\E_{j-1}\gamma_j\Tilde{\beta}_j(z)=\E [\gamma_j\Tilde{\beta}_j(z)\big|r_1,...,r_{j-1}]\nonumber\\
    =&\E_{r_{j+1},...,r_N}\left[\E_
{r_j}[\gamma_j\Tilde{\beta}_j(z)\big|r_1,...,r_{j-1},{r_{j+1},...,r_N}]\right]\nonumber\\
=&\E_{r_{j+1},...,r_N}\left[\Tilde{\beta}_j(z)\E_
{r_j}[\gamma_j\big|r_1,...,r_{j-1},{r_{j+1},...,r_N}]\right]\nonumber\\
=&\E_{r_{j+1},...,r_N}\left[\Tilde{\beta}_j(z)\E_
{r_j}[N^{-1}\text{Tr}\left(T_j^{1/2}D_j^{-2}T_j^{1/2}(x_jx_j^T-I)\right)\big|r_1,...,r_{j-1},{r_{j+1},...,r_N}]\right]=0.\label{eq:ej-1}
\end{align}
By the mutual independence of $(\E_j-\E_{j-1})\gamma_j\beta_j(z)\Tilde{ \beta}_j(z)\epsilon_j(z)$ the basic algebraic fact that $(x+y)^2\leq 2(x^2+y^2)$, we have
\begin{align}
    \E \left|\sum_{j=1}^N(\E_j-\E_{j-1})\gamma_j\beta_j(z)\Tilde{\beta}_j(z)\epsilon_j(z)\right|^2
    &= \sum_{j=1}^N\E \left|(\E_j-\E_{j-1})\gamma_j\beta_j(z)\Tilde{\beta}_j(z)\epsilon_j(z)\right|^2\label{eq:op1.start}\\
    &\leq 4\sum_{j=1}^N\E \left|\gamma_j\beta_j(z)\Tilde{\beta}_j(z)\epsilon_j(z)\right|^2.\nonumber
\end{align}
Apply  the boundedness results (\ref{eq:bd.beta}) on $\beta_j(z)$, $\Tilde{\beta}_j(z)$, then
\begin{align*}
    \E \left|\sum_{j=1}^N(\E_j-\E_{j-1})\gamma_j\beta_j(z)\Tilde{\beta}_j(z)\epsilon_j(z)\right|^2
    &\leq C_z\sum_{j=1}^N\E \left|\gamma_j\epsilon_j(z)\right|^2,
\end{align*}
where $C_z$ is a constant that depends on $z$.
By Cauchy Schwartz and Lemma \ref{lem:Mn1.prelim}, 
\begin{align*}
    \E \left|\sum_{j=1}^N(\E_j-\E_{j-1})\gamma_j\beta_j(z)\Tilde{\beta}_j(z)\epsilon_j(z)\right|^2
    &\leq C_z\sum_{j=1}^N\sqrt{\E \left|\gamma_j\right|^4}\sqrt{\E \left|\epsilon_j\right|^4}=o(1).
\end{align*}
As a result, applying Markov's Inequality, we can establish the convergence in probability as follows
\begin{align}
    \sum_{j=1}^N(\E_j-\E_{j-1})\gamma_j\beta_j(z)\Tilde{\beta}_j(z)\epsilon_j(z)=o_p(1).\label{eq:op1.end}
\end{align}
Collecting the terms yields 
\begin{align}\label{eq:I1}
    I_1=-\sum_{j=1}^N\E_j\gamma_j\Tilde{\beta}_j(z)+o_p(1).
\end{align}
Similarly, property (\ref{eq:beta.prop2}) implies
\begin{align*}
    I_2=&\sum_{j=1}^N(\E_j-\E_{j-1}) \Tilde{\beta}_j(z)N^{-1}\text{Tr}(D_j^{-2}(z)T_j)-\sum_{j=1}^N(\E_j-\E_{j-1})\Tilde{\beta}_j
    ^2(z)\epsilon_j(z)N^{-1}\text{Tr}(D_j^{-2}(z)T_j)\\
    &+\sum_{j=1}^N(\E_j-\E_{j-1})\Tilde{\beta}_j
    ^2(z)\beta_j(z)\epsilon_j^2(z)N^{-1}\text{Tr}(D_j^{-2}(z)T_j)
\end{align*}
Since $\Tilde{\beta}_j(z)\text{Tr}(D_j^{-2}(z)T_j)$ does not depend on $r_j$, we obtain
\begin{align*}
    \sum_{j=1}^N(\E_j-\E_{j-1}) \Tilde{\beta}_j(z)N^{-1}\text{Tr}(D_j^{-2}(z)T_j)=0.
\end{align*}
Following the same arguments leading to (\ref{eq:ej-1}) gives
\begin{align}\label{eq:ej-1.I2}
    \sum_{j=1}^N\E_{j-1}\Tilde{\beta}_j
    ^2(z)\epsilon_j(z)N^{-1}\text{Tr}(D_j^{-2}(z)T_j)=0.
\end{align}
Finally, arguing similarly to (\ref{eq:op1.start}-\ref{eq:op1.end}), we get
\begin{align*}
    \sum_{j=1}^N(\E_j-\E_{j-1})\Tilde{\beta}_j
    ^2(z)\beta_j(z)\epsilon_j^2(z)N^{-1}\text{Tr}(D_j^{-2}(z)T_j)=o_p(1).
\end{align*}
Collecting the terms gives
\begin{align}\label{eq:I2}
    I_2=\sum_{j=1}^N\E_j\Tilde{\beta}_j(z)\epsilon_j(z)N^{-1}\text{Tr}(D_j^{-2}(z)T_j)+o_p(1)
\end{align}
From (\ref{eq:tele}), (\ref{eq:I1}), and (\ref{eq:I2}), we conclude that
\begin{align}\label{eq:mart}
    n(m_n-\E [m_n])=\sum_{j=1}^Nh_j(z)+o_p(1),
\end{align}
where 
\begin{align}\label{eq:hj}
    h_j(z)=-\E_j\left(\Tilde{\beta}_j(z)\gamma_j(z)-\Tilde{\beta}_j
    ^2(z)\epsilon_j(z)N^{-1}\text{Tr}(D_j^{-2}(z)T_j)\right)=-\E_j\frac{d}{dz}\Tilde{\beta}_j(z)\epsilon_j(z).
\end{align}
From (\ref{eq:ej-1}) and (\ref{eq:ej-1.I2}), it is straightforward to check that $\E _{j-1}[h_j(z)]=0$.

Next, we verify condition (ii) in Theorem \ref{thm:martingale}. 
From the boundedness of $\Tilde{\beta}_j(z)$ and $N^{-1}\text{Tr}(D_j^{-2}(z)T_j)$, and the algebraic fact that $(x+y)^4\leq K(x^4+y^4)$ for some $K>0$, we have
\begin{align*}
    \E |h_j(z)|^4\leq K_1\E |\gamma_j(z)|^4+K_2\E |\gamma_j(z)|^4=o(N^{-1}),
\end{align*}
where the last step follows from Lemma \ref{lem:Mn1.prelim}. As a result, we obtain
\begin{align*}
     \sum_{j=1}^N\E [|h_j|^2(z)\mathbbm{1}_{|h_j|>\epsilon}]\leq \epsilon^{-2}\sum_{j=1}^N\E |h_j(z)|^4\rightarrow 0.
\end{align*}

\subsection{Covariance calculation}\label{subsec: mn1.finite2}
Let us start by substituting the martingale representation (\ref{eq:mart}) into the finite-dimensional sum (\ref{Mn1:finite.sum}). To verify condition (i) in Theorem \ref{thm:martingale}, our objective is to establish the existence of a deterministic sequence $\check{\sigma}_n$ such that
\begin{align*}
    \check{\sigma}_n^{-2}\sum_{j=1}^N\E_{j-1}\left[\left|\sum_{\nu=1}^l\alpha_{\nu}h_j(z_{\nu})\right|^2\right]\convp 1.
\end{align*}
To this end, it suffices to prove the existence of a sequence of $\sigma_n$ such that 
\begin{align}\label{eq:phi.z1.z2}
    {\sigma}_n^{-2}\Phi_n(z_1,z_2)={\sigma}_n^{-2}\sum_{j=1}^N\E_{j-1}[h_j(z_1)h_j(z_2)]\convp 1.
\end{align}
\subsubsection{Simplified representation}\label{subsec:sim.rep}
First, we shall derive a more tractable equivalent expression of (\ref{eq:phi.z1.z2}), where equivalence is defined in terms of the difference vanishing in probability. Incorporating (\ref{eq:hj}), and guided by the boundedness and integrability results established in Lemma \ref{lem:Mn1.prelim}, we are positioned to apply the Dominated Convergence Theorem to the difference quotient defined by $\deteqv{\beta}_j(z)\epsilon_j(z)$. Consequently, we obtain
\begin{align}
    \Phi_n(z_1,z_2)=\frac{\partial^2}{\partial z_2\partial z_1}\sum_{j=1}^N\E_{j-1}[\E_j(\Tilde{\beta}_j(z_1)\epsilon_j(z_1))\E_j(\Tilde{\beta}_j(z_2)\epsilon_j(z_2))].
\end{align}
Thus, as formalized in Lemma \ref{app:lem.der}, it suffices to consider
\begin{align}\label{eq:beta.eps}
   \sum_{j=1}^N\E_{j-1}[\E_j(\Tilde{\beta}_j(z_1)\epsilon_j(z_1))\E_j(\Tilde{\beta}_j(z_2)\epsilon_j(z_2))].
\end{align}
By Cauchy-Schwartz and Lemma \ref{lem:concen}, we establish in Lemma \ref{lem:tilde.beta.b} that 
the terms $\Tilde{\beta}_j$ can be equivalently replaced with the deterministic terms $b_j$. 
Consequently, we examine
\begin{align}\label{eq:b.eps}
    \sum_{j=1}^Nb_j(z_1)b_j(z_2)\E_{j-1}[\E_j(\epsilon_j(z_1))\E_j(\epsilon_j(z_2))].
\end{align}
In Lemma \ref{lem:Mn1.prelim}, we computed the limit in probability of $b_j(z)$. To study $\E_j(\epsilon_j(z_1))\E_j(\epsilon_j(z_2))$, we define $A=T_j^{1/2}D_j^{-1}(z_1)T_j^{1/2}$, $B=T_j^{1/2}D_j^{-1}(z_2)T_j^{1/2}$. Plugging in the definition of $\epsilon_j$ in (\ref{eq:eps}), we have
\begin{align*}
&\E_{j-1}[\E_j(\epsilon_j(z_1))\E_j(\epsilon_j(z_2))]\\
=&\E_{j-1}[\E_j(r_j^TD_j^{-1}(z_1)r_j-N^{-1}\text{Tr}(D_j^{-1}(z_1)T_j))\E_j(r_j^TD_j^{-1}(z_2)r_j-N^{-1}\text{Tr}(D_j^{-1}(z_2)T_j))]\\
=&N^{-2}\E_{j-1}[(x_j^T(\E_{j-1}A)x_j-\text{Tr}(\E_{j-1}A))(x_j^T(\E_{j-1}B)x_j-\text{Tr}(\E_{j-1}B))].
\end{align*}
Apply Lemma \ref{lemma:epsilon.exp}, then 
\begin{align*}
&\E_{j-1}[\E_j(\epsilon_j(z_1))\E_j(\epsilon_j(z_2))]=N^{-2}(2\text{Tr}\E_{j-1}(AB)+o(n))\\
=&N^{-2}2\text{Tr}(T_j^{1/2}\E_jD_j^{-1}(z_1)T_j\E_jD_j^{-1}(z_2)T_j^{1/2})+o(N^{-1}).
\end{align*}
Therefore, the goal is to prove that there exists $\sigma_n'$ such that
\begin{align}
    &\sigma_n'^{-2}N^{-2}\sum_{j=1}^Nb_j(z_1)b_j(z_2)\text{Tr}(\E_jD_j^{-1}(z_1)T_j\E_jD_j^{-1}(z_2)T_j)\\
    =&\sigma_n'^{-2}N^{-2}\sum_{r=1}^k\sum_{j=1}^Nl_{rj}^2b_j(z_1)b_j(z_2)w_{jr}(z_1,z_2)\rightarrow 1,
\end{align}
where $w_{jr}(z_1,z_2)=\text{Tr}(\E_jD_j^{-1}(z_1)T_j\E_jD_j^{-1}(z_2)\Sigma_r).$
\subsubsection{Convergence of $w_{jr}(z_1,z_2)$}\label{subsec:cov.replace}
The main idea is to sequentially substitute the random components in $w_{jr}(z_1,z_2)$, namely $D_j^{-1}(z_1)$ and $D_j^{-1}(z_2)$ by the deterministic $R_j^{-1}(z_1)$ and $R_j^{-1}(z_2)$ and compute the non-vanishing terms. To this end, we introduce a few notations: for $i\neq j$,
\begin{align}\label{eq:beta.ij}
    D_{ij}=B_n-r_ir_i^T-r_jr_j^T,\quad \beta_{ij}=\frac{1}{1+r_i^TD_{ij}^{-1}r_i},\quad b_{ij}=\frac{1}{1+N^{-1}E\text{Tr}(D_{ij}^{-1}T_i)}.
\end{align}
By definition,
\begin{align*}
    D_j(z)+R_j(z)=\sum_{i\neq j}r_ir_i^T-\frac{1}{N}\sum_{i\neq j}\psi_i(z)T_i,
\end{align*}
then
\begin{align}
    R_j^{-1}(z)+D_j^{-1}(z)
    &=R_j^{-1}(D_j(z)+R_j(z))D_j^{-1}\nonumber\\
    &=\sum_{i\neq j}R_j^{-1}r_ir_i^TD_j^{-1}-\frac{1}{N}\sum_{i\neq j}\psi_i(z)R_j^{-1}T_iD_j^{-1}\nonumber\\
    &=\sum_{i\neq j}R_j^{-1}r_ir_i^TD_j^{-1}-\sum_{i\neq j}\psi_i(z)R_j^{-1}r_ir_i^TD_{ij}^{-1}\label{eq:riDj}\\
    &+\sum_{i\neq j}\psi_i(z)R_j^{-1}r_ir_i^TD_{ij}^{-1}-\frac{1}{N}\sum_{i\neq j}\psi_i(z)R_j^{-1}T_iD_{ij}^{-1}\nonumber\\
    &+\frac{1}{N}\sum_{i\neq j}\psi_i(z)R_j^{-1}T_iD_{ij}^{-1}-\frac{1}{N}\sum_{i\neq j}\psi_i(z)R_j^{-1}T_iD_j^{-1}.\nonumber
\end{align}
From the fact that $\alpha^T(\Sigma+\beta\alpha^T)^{-1}=\alpha^T\Sigma^{-1}(1+\alpha^T\Sigma^{-1}\beta)^{-1}$,
we have
\begin{align*}
    r_i^TD_j^{-1}=r_i^T(D_{ij}+r_ir_i^T)^{-1}=\frac{r_i^TD_{ij}^{-1}}{1+r_i^TD_{ij}^{-1}r_i}=\beta_{ij}r_i^TD_{ij}^{-1}.
\end{align*}
Plugging this into equation (\ref{eq:riDj}), then
\begin{align}
    R_j^{-1}(z)+D_j^{-1}(z)
    &=\sum_{i\neq j}(\beta_{ij}(z)-\psi_i(z))R_j^{-1}r_ir_i^TD_{ij}^{-1}\nonumber\\
    &+\sum_{i\neq j}\psi_i(z)R_j^{-1}(r_ir_i^T-N^{-1}T_i)D_{ij}^{-1}\nonumber\\
    &+\frac{1}{N}\sum_{i\neq j}\psi_i(z)R_j^{-1}T_i(D_{ij}^{-1}-D_j^{-1})\triangleq A_2(z)+A_1(z)+A_3(z).\label{eq:decom.rj.dj}
\end{align}
Prior to substituting $D_j^{-1}(z_1)$ with $R_j^{-1}(z_1)$ in $w_{jr}(z_1,z_2)$, we establish several useful results regarding $A_1$, $A_2$ and $A_3$. We will demonstrate that terms involving $A_2$ and $A_3$ vanish, while $A_1$ requires more scrutiny.
\begin{lemma}\label{lem:A123}
    Under Assumption \ref{assump_trunc}, for any $z\in C_0$, the following hold:
    \begin{enumerate}
        \item For any (possibly random) matrix $M$ with  deterministic bound on its spectral norm, 
        \begin{align*}
            \E |\text{Tr}A_2(z)M|\leq O(N^{1/2}),\quad|\text{Tr}A_3(z)M|\leq O(1).
        \end{align*}
        \item For any deterministic matrix $M$ with  deterministic bound on its spectral norm, 
        \begin{align*}
            \E |\text{Tr}(A_1(z)M)|\leq O(N^{1/2}).
        \end{align*}
    \end{enumerate}
\end{lemma}
In Lemma \ref{lem:Mn1.prelim}, we proved that $\lVert R_j(z)\rVert\leq v^{-1}$. Therefore, substituting $D_j^{-1}(z_1)$ with $-R_j^{-1}(z_1)+A_1(z_1)+A_2(z_1)+A_3(z_1)$, from Lemma \ref{lem:A123} we immediately have
\begin{align*}
    w_{jr}(z_1,z_2)=&\text{Tr}(\Sigma_r\E_jD_j^{-1}(z_1)T_j\E_jD_j^{-1}(z_2))\\
   =&-\text{Tr}(\Sigma_rR_j^{-1}(z_1)T_j\E_jD_j^{-1}(z_2))+\text{Tr}(\Sigma_r\E_jA_1(z_1)T_j\E_jD_j^{-1}(z_2))+a(z_1,z_2),
\end{align*}
where $\E |a(z_1,z_2)|\leq O(N^{1/2})$.\footnote{Throughout the remainder of this section, we denote by $a(z_1,z_2)$ a term for which $\E |a(z_1,z_2)|\leq O(N^{1/2})$. It should be noted that $a(z_1,z_2)$ may not be the same in each occurrence.} Further substitute $D_j^{-1}(z_2)$ with $-R_j^{-1}(z_2)+A_1(z_2)+A_2(z_2)+A_3(z_2)$, then
\begin{align}
    w_{jr}(z_1,z_2)=&\text{Tr}(\Sigma_r\E_jD_j^{-1}(z_1)T_j\E_jD_j^{-1}(z_2))\nonumber\\
   =&\text{Tr}(\Sigma_rR_j^{-1}(z_1)T_jR_j^{-1}(z_2))+\text{Tr}(\Sigma_r\E_jA_1(z_1)T_j\E_jD_j^{-1}(z_2))+a(z_1,z_2).\label{eq:wjr.interm}
\end{align}

Finally, to study $\text{Tr}(\Sigma_r\E_jA_1(z_1)T_j\E_jD_j^{-1}(z_2))$, we combine the substitution technique with concentration results established in Lemma \ref{lem:concen}. 

Recall
\begin{align*}
    A_1(z_1)=\sum_{i\neq j}\psi_i(z_1)R_j^{-1}(z_1)(r_ir_i^T-N^{-1}T_i)D_{ij}^{-1}(z_1).
\end{align*}
First, note that for $i>j$,
\begin{align*}
   \E_j[\psi_i(z_1)R_j^{-1}(z_1)(r_ir_i^T-N^{-1}T_i)D_{ij}^{-1}(z_1)] = \psi_i(z_1)R_j^{-1}(z_1)\E_j[(r_ir_i^T-N^{-1}T_i)D_{ij}^{-1}(z_1)] =0,
\end{align*}
then
\begin{align*}
    A_1(z_1)=\sum_{i< j}\psi_i(z_1)R_j^{-1}(z_1)(r_ir_i^T-N^{-1}T_i)D_{ij}^{-1}(z_1).
\end{align*}
Now, write
\begin{align}
    &\text{Tr}(\E_jA_1(z_1)T_j\E_jD_j^{-1}(z_2)\Sigma_r)\nonumber\\
    =&\sum_{i< j}\psi_i(z_1)\text{Tr}\left[R_j^{-1}(z_1)(r_ir_i^T-N^{-1}T_i)\E_jD_{ij}^{-1}(z_1)T_j\E_jD_j^{-1}(z_2)\Sigma_r\right]\nonumber\\
    =&+\sum_{i< j}\psi_i(z_1)\text{Tr}\left[R_j^{-1}(z_1)(r_ir_i^T-N^{-1}T_i)\E_jD_{ij}^{-1}(z_1)T_j\E_jD_{ij}^{-1}(z_2)\Sigma_r\right]\label{eq:A1.4}\\
    &+\sum_{i< j}\psi_i(z_1)\text{Tr}\left[R_j^{-1}(z_1)(r_ir_i^T-N^{-1}T_i)\E_jD_{ij}^{-1}(z_1)T_j\E_j\left(D_{j}^{-1}-D_{ij}^{-1}\right)(z_2)\Sigma_r\right].\label{eq:A1.1}
\end{align}
where
\begin{align}
 (\ref{eq:A1.4})
    =&\sum_{i< j}\psi_i(z_1)(r_i^T\E_jD_{ij}^{-1}(z_1)T_j\E_jD_{ij}^{-1}(z_2)\Sigma_rR_j^{-1}(z_1)r_i\nonumber\\
    &-N^{-1}\text{Tr}\left[T_i\E_jD_{ij}^{-1}(z_1)T_j\E_jD_{ij}^{-1}(z_2)\Sigma_rR_j^{-1}(z_1)\right])\nonumber\\
    (\ref{eq:A1.1})
    =
    &-\frac{1}{N}\sum_{i< j}\psi_i(z_1)\text{Tr}\left[R_j^{-1}(z_1)T_i\E_jD_{ij}^{-1}(z_1)T_j\E_j\left(D_{j}^{-1}-D_{ij}^{-1}\right)(z_2)\Sigma_r\right]\label{eq:A1.5}\\
    &+\sum_{i< j}\psi_i(z_1)\text{Tr}\left[R_j^{-1}(z_1)r_ir_i^T\E_jD_{ij}^{-1}(z_1)T_j\E_j\left(D_{j}^{-1}-D_{ij}^{-1}\right)(z_2)\Sigma_r\right].\label{eq:A1.2}
\end{align}
By concentration result established in Lemma \ref{lem:concen}, it is straightforward to check that
\begin{align*}
    \E |(\ref{eq:A1.4})|\leq O(N^{1/2}).
\end{align*}
With the boundedness results established in (\ref{eq:d.dj.opnorm}) and Lemma \ref{lem:Mn1.prelim}, we also have
\begin{align*}
    \E |(\ref{eq:A1.5})|\leq O(1).
\end{align*} 
We proceed to analyze
\begin{align}
    \sum_{i< j}\psi_i(z_1)\text{Tr}\left[R_j^{-1}(z_1)r_ir_i^T\E_j(D_{ij}^{-1}(z_1))T_j\left(D_{j}^{-1}-D_{ij}^{-1}\right)(z_2)\Sigma_r\right]\label{eq:A1.3},
\end{align}
which satisfies $(\ref{eq:A1.2})=\E_j\left[(\ref{eq:A1.3})\right]$.  With property (\ref{eq:D.Dj}), we can expand $D_{j}^{-1}-D_{ij}^{-1}$ as
\begin{align}\label{eq:Dj.Dij}
    D_{j}^{-1}-D_{ij}^{-1}
    &=-\beta_{ij}D_{ij}^{-1}r_ir_i^TD_{ij}^{-1}.
\end{align}
This gives
\begingroup
\allowdisplaybreaks
\begin{align*}
    (\ref{eq:A1.3})
    =-&
    \sum_{i< j}\psi_i(z_1)\beta_{ij}(z_2)\text{Tr}\left[R_j^{-1}(z_1)r_ir_i^T\E_j(D_{ij}^{-1}(z_1))T_jD_{ij}^{-1}(z_2)r_ir_i^TD_{ij}^{-1}(z_2)\Sigma_r\right]\\
    =-&
    \sum_{i< j}\psi_i(z_1)\beta_{ij}(z_2)\left(r_i^T\E_j(D_{ij}^{-1}(z_1))T_jD_{ij}^{-1}(z_2)r_i\right)\left(r_i^TD_{ij}^{-1}(z_2)\Sigma_rR_j^{-1}(z_1)r_i\right)\\
    =-&
    \sum_{i< j}\psi_i(z_1)\beta_{ij}(z_2)\left(r_i^T\E_j(D_{ij}^{-1}(z_1))T_jD_{ij}^{-1}(z_2)r_i-N^{-1}\text{Tr}\left[\E_j(D_{ij}^{-1}(z_1))T_jD_{ij}^{-1}(z_2)T_i\right]\right)\\
    &\qquad\qquad\quad\quad\times \left(r_i^TD_{ij}^{-1}(z_2)\Sigma_rR_j^{-1}(z_1)r_i-N^{-1}\text{Tr}\left[D_{ij}^{-1}(z_2)\Sigma_rR_j^{-1}(z_1)T_i\right]\right)\\
    -&
   \sum_{i< j}\psi_i(z_1)\beta_{ij}(z_2)\text{Tr}\left[D_{ij}^{-1}(z_2)\Sigma_rR_j^{-1}(z_1)T_i\right]\\
    &\qquad\qquad\quad\quad\times \left(r_i^T\E_j(D_{ij}^{-1}(z_1))T_jD_{ij}^{-1}(z_2)r_i-N^{-1}\text{Tr}\left[\E_j(D_{ij}^{-1}(z_1))T_jD_{ij}^{-1}(z_2)T_i\right]\right)\\
    -&
    \sum_{i< j}\psi_i(z_1)\beta_{ij}(z_2)\text{Tr}\left[\E_j(D_{ij}^{-1}(z_1))T_jD_{ij}^{-1}(z_2)T_i\right]\\
    &\qquad\qquad\quad\quad\times \left(r_i^TD_{ij}^{-1}(z_2)\Sigma_rR_j^{-1}(z_1)r_i-N^{-1}\text{Tr}\left[D_{ij}^{-1}(z_2)\Sigma_rR_j^{-1}(z_1)T_i\right]\right)\\
    -&
    \frac{1}{N^2}\sum_{i< j}\psi_i(z_1)\beta_{ij}(z_2)\text{Tr}\left[\E_j(D_{ij}^{-1}(z_1))T_jD_{ij}^{-1}(z_2)T_i\right]\text{Tr}\left[D_{ij}^{-1}(z_2)\Sigma_rR_j^{-1}(z_1)T_i\right]\\
    \triangleq\quad &\alpha_1(z_1,z_2)+\alpha_2(z_1,z_2)+\alpha_3(z_1,z_2)+\alpha_4(z_1,z_2).
\end{align*}
\endgroup
Applying the concentration result established in Lemma \ref{lem:concen} gives
\begin{align*} \E |\alpha_1(z_1,z_2)+\alpha_2(z_1,z_2)+\alpha_3(z_1,z_2)|\leq O(N^{1/2}).
\end{align*}
Collecting the terms, we have
\begin{align}\label{eq:A1..}
    \text{Tr}(\E_jA_1(z_1)T_j\E_jD_j^{-1}(z_2)\Sigma_r)
    =\E_j\alpha_4(z_1,z_2)+a(z_1,z_2).
\end{align}  
In Lemma \ref{app:lem.prop}, it is established that
\begin{align*}
    \E \left|\beta_{ij}(z)-\psi_i(z)\right|\leq C_zN^{-1/2},
\end{align*}
for a constant $C$ that depends only on $z$. Combining this property with property (70) and the concentration results established in Lemma \ref{lem:concen}, we can substitute $\beta_{ij}$ with $\psi_i$, and substitute $D_{ij}$ with $D_j$ to obtain
\begin{align*}
    \alpha_4(z_1,z_2)=-
    \frac{1}{N^2}\sum_{i< j}\psi_i(z_1)\psi_i(z_2)\text{Tr}&\left[\E_j(D_{j}^{-1}(z_1))T_jD_{j}^{-1}(z_2)T_i\right]\times\\
    &\quad\text{Tr}\left[D_{j}^{-1}(z_2)\Sigma_AR_j^{-1}(z_1)T_i\right]+a(z_1,z_2).
\end{align*}
 Now, if we substitute $D_j^{-1}(z_2)$ with $-R_j^{-1}(z_2)+A_1(z_2)+A_2(z_2)+A_3(z_3)$, then taken together with the properties of $A_1$, $A_2$, $A_3$ established in Lemma \ref{lem:A123} and the non-randomness of $R_j$, we have
\begin{align}
     \alpha_4(z_1,z_2)=-
    \frac{1}{N^2}\sum_{i< j}\psi_i(z_1)\psi_i(z_2)\text{Tr}&\left[\E_j(D_{j}^{-1}(z_1))T_jD_{j}^{-1}(z_2)T_i\right]\times\nonumber\\
    &\text{Tr}\left[R_{j}^{-1}(z_2)\Sigma_AR_j^{-1}(z_1)T_i\right]+a(z_1,z_2).\label{eq:alpha.4}
\end{align}
 Substituting the above formula into (\ref{eq:A1..}) and (\ref{eq:wjr.interm}) yields
\begin{align}
    &w_{jr}(z_1,z_2)=\E_j \text{Tr}(\E_j(D_j^{-1}(z_1))T_jD_j^{-1}(z_2)\Sigma_r)\nonumber\\
   =&\text{Tr}(R_j^{-1}(z_1)T_jR_j^{-1}(z_2)\Sigma_r)+a(z_1,z_2)\nonumber\\
   &+\frac{1}{N^2}\E_j\sum_{i< j}\psi_i(z_1)\psi_i(z_2)\text{Tr}\left[\E_j(D_{j}^{-1}(z_1))T_jD_{j}^{-1}(z_2)T_i\right]\text{Tr}\left[R_{j}^{-1}(z_2)\Sigma_rR_j^{-1}(z_1)T_i\right]\nonumber\\
   =&\text{Tr}(R_j^{-1}(z_1)T_jR_j^{-1}(z_2)\Sigma_r)+a_6(z_1,z_2)\nonumber\\
   &+\frac{1}{N^2}{\sum_{s=1}^k}\E_jw_{js}(z_1,z_2)\sum_{i< j}l_{si}^2\psi_i(z_1)\psi_i(z_2)\text{Tr}\left[R_{j}^{-1}(z_2)\Sigma_rR_j^{-1}(z_1)T_i\right].\label{eq: interm.step}
\end{align}
For each $a,b=1,\ldots,k$, define 
\begin{align}
\Xi^{ab}_j&=\frac{1}{N}\text{Tr}\left[\deteqv{R}_{j}^{-1}(z_2)\Sigma_a\deteqv{R}_j^{-1}(z_1)\Sigma_b\right],\quad
    h^{ab}_j=\frac{1}{N}\sum_{i< j}l_{ai}^2l_{bi}^2\deteqv{b}_i(z_1)\deteqv{b}_i(z_2),\quad\Lambda^{ab}_j=\sum_{r=1}^kh^{ra}_j\Xi^{rb}_j,\label{eq:xi.h.Lam}
\end{align}
where
\begin{align}\label{eq:R.deteqv}
    \deteqv{R}=\frac{1}{N}\sum_{j}\deteqv{b}_j(z)T_j-zI.
\end{align}
Above, we have substituted $\psi_j(z)$  in (\ref{eq: interm.step}) by $\deteqv{b}_j(z)$ defined in Lemma \ref{lem:Mn1.prelim}, with $|\psi_j(z)-\deteqv{b}_j(z)|=o(1)$.
Let $\deteqv{w}_{jr}(z_1,z_2)$ be solutions to the system of equations
\begin{align}
    \begin{cases}\label{eq:deteqv.w}
        \deteqv{w}_{jr}(z_1,z_2)=&{\sum_{s=1}^k}l_{sj}^2\Xi^{sr}_j+{\sum_{s=1}^k}
    \deteqv{w}_{js}(z_1,z_2)\Lambda^{sr}_j,r=1\ldots,k.
    \end{cases}
\end{align}
Combined with the arguments in Section \ref{subsec:sim.rep}, we conclude that
\begin{align}
   \sigma_n^2(z_1,z_2)= 
        \frac{\partial^2}{\partial z_2\partial z_1}\sum_{r=1}^k\left(N^{-2}\sum_{j=1}^Nl_{rj}^2\deteqv{b}_j(z_1)\deteqv{b}_j(z_2)\deteqv{w}_{jr}(z_1,z_2)\right)
\end{align}
satisfies condition (\ref{eq:phi.z1.z2}). 
\subsection{Tightness of $M_n^1$}\label{subsec: mn1.tight}
By condition of tightness in Theorem 8.2 and 12.3 in Billingsley \cite{billingsley2013convergence}, we would like to prove that for any positive $\epsilon$, $\eta>0$, there exists $\delta\in\left[0,1\right]$ such that, for any $|z_1-z_2|\leq \delta$, we have
\begin{align}\label{eq:tightness.def}
    \mathbbm{P}\left[\left|M_n^1(z_1)-M_n^1(z_2)\right|>\epsilon\right]<\eta.
\end{align} 

First, building upon the analyses in \cite{Guionnet2000}, we obtain the following concentration result:
\begin{lemma}[Concentration of the spectral measure of $B_n$]\label{lem:concen.spectral.S}
For functions f such that $g(x)=f(x^2)$ is lipschitz, we have
    \begin{align}\label{eq: cocentration}
    \mathbbm{P}\left(\left| \int fdF_n-\E \int fdF_n\geq \epsilon\right|\right)\leq \exp\left(-\frac{C_{\epsilon}n^2}{|g|_L^2}\right),
    \end{align}
    where $C_{\epsilon}$ is a constant that depends on $\epsilon$.
    \end{lemma}

Using this concentration result, we establish the tightness of $M_n^1$ as follows:
\begin{lemma}
     On $C_0=\{z:\text{Im}z>v_0\}$, the stochastic process 
    $$\left\{M_n^1(z)=n(m_n(z)-\E [m_n(z)])\bigg|z\in C_0\right\}$$ 
    is tight.
\end{lemma}
\begin{proof}
When $\text{Im}z>v_0$, by definition of tightness we would like to prove that for an positive $\epsilon$, $\eta>0$, there exists $\delta\in\left[0,1\right]$ such that, for any $|z_1-z_2|\leq \delta$, we have
\begin{align*}
    \mathbbm{P}\left[\left|M_n^1(z_1)-M_n^1(z_2)\right|>\epsilon\right]<\eta.
\end{align*}
Define a function $\Tilde{f}(\lambda)=(\lambda-z_1)^{-1}-(\lambda-z_2)^{-1}$, then 
$$M_n^1(z_1)-M_n^1(z_2)=n\left[\int \Tilde{f}dF_n-\E \int \Tilde{f}dF_n\right].$$
Let $g(x)=\Tilde{f}(x^2)=(x^2-z_1)^{-1}-(x^2-z_2)^{-1}$. Observe that both the real and imaginary parts of $g$ are $C\delta v_0^{-4}$ Lipschitz. Therefore, applying Lemma \ref{lem:concen.spectral.S}, we have the bound
\begin{align*}
    \mathbbm{P}\left[\left|M_n^1(z_1)-M_n^1(z_2)\right|>\epsilon\right]
    =\mathbbm{P}\left[n\left|\int \Tilde{f}dF_n-\E \int \Tilde{f}dF_n\right|>\epsilon\right]<2\exp{\left(-\frac{C_\epsilon v_0^8}{\delta^2}\right)}.
\end{align*}
We conclude the proof by taking $\delta = \min\left\{\frac{C_{\epsilon}v_0^4}{\log{(2\eta^{-1})}},1\right\}$. 
\end{proof}

\subsection{Central Limit Theorem for $M_n^1$}\label{subsec: clt.mn1}
Previously, we have been working with $z\in C_0=\{z:\text{Im}z>v_0\}$. 
In Sections \ref{subsec: mn1.finite} and \ref{subsec: mn1.finite2}, we have proved the convergence of finite sums of the form
\begin{align*}
     \frac{\sum_{\nu=1}^l \alpha_{\nu}M_n^1(z_{\nu})}{\sqrt{{\sum_{\nu=1}^l\sum_{\nu'=1}^l}\alpha_{\nu}\alpha_{\nu'}\sigma_n^2(z_{\nu},\Bar{z}_{\nu'})}}
  \end{align*}
for $z_1,\ldots,z_l\in C_0$. In Section \ref{subsec: mn1.tight}, we established the tightness of $M_n^1$. In this section, we aim to prove the following result:
\begin{lemma}\label{cor:sos.gamma.gp}
Under Assumption \ref{assump_trunc}, for any function $g$ analytic on an interval containing (\ref{eq:open.int}), we have that 
 \begin{align*}
    \frac{\oint g(z)M_n^1(z)dz}{\sqrt{\oint\oint g(z_1)g(z_2)\sigma^2_n(z_1,z_2)dz_1dz_2}}\convd CN(0,1),
\end{align*}
where the contours are as specified in Lemma \ref{lem:main}.
\end{lemma}
Let $\mathcal{C}$ be a contour containing the interval (\ref{eq:open.int}), with endpoints at $(\pm r,\pm v_0)$, with $r=(C_u+\eta)^2+\eta$. Here $C_u$ is the upper spectral limit defined in (\ref{lemma: concen.eig.sofs}), $\eta,v_0>0$. We split $\mathcal{C}$ into the union of $\mathcal{C}_u$, $\overline{\mathcal{C}_u}$ and $\mathcal{C}_j$, where $\mathcal{C}_u=\{z=x+iv_0,|x|\leq r\}$, $\mathcal{C}_j=\{z=\pm r+iy,|y|\leq v_0\}$. {Combining the finite-dimensional convergence and tightness results, we have that}
\begin{align}\label{eq:cu.mn1.conv}
    \frac{\int_{\mathcal{C}_u} g(z)M_n^1(z)dz}{\sqrt{\int_{\mathcal{C}_u}\int_{\overline{\mathcal{C}_u}} g(z_1)g(z_2)\sigma^2_n(z_1,z_2)dz_1dz_2}}\convd CN(0,1).
\end{align}
The formal proof is deferred to Section \ref{app:subsec.cu.mn1.conv}.
Therefore, to complete the proof of Lemma \ref{cor:sos.gamma.gp} we need only establish 
\begin{align}\label{eq:mn1.cj}
    \int_{\mathcal{C}_j} g(z)M_n^1(z)dz\rightarrow 0.
\end{align}
When $\text{Im}z<v_0$, let $Q_n=\{\lambda_{max}(S)\leq C_u+\eta\}$. By Lemma \ref{lemma: concen.eig.sofs}, we have that for any $\eta>0$, positive integer $t>0$, 
\begin{align*}
    \mathbbm{P}[Q_n^c]=o(n^{-t}).
\end{align*}
By Cauchy-Schwartz, we can conclude (\ref{eq:mn1.cj}) from the following result.
\begin{lemma} \label{lem.xiqn}
Under the assumptions of Lemma \ref{lemma: concen.eig.sofs},
\begin{align*}
    \lim_{v_0\rightarrow 0}\limsup_{n}\int_{\mathcal{C}_j}\E [|M_n^1(z)\mathbbm{1}_{Q_n}|^2]dz= 0.
\end{align*}
\end{lemma}
    
\begin{proof}[{Proof of Lemma \ref{lem.xiqn}}]
For any $z\in\mathcal{C}_j$, and any $\lambda_i$, $i=1,...,n$, on $Q_n$, we have 
$$
|\lambda_i^2-z|\geq \eta.
$$
Define a trucation function $[\lambda]_K$, where $K=C_u+\eta$, as
\begin{align*}
    [\lambda]_K=\begin{cases}
      -K & \text{if }\lambda\leq -K,\\
      \lambda & \text{if }|\lambda|\leq K,\\
      K & \text{if }\lambda\geq K.
    \end{cases} 
\end{align*}
Additionally, define 
$$
g_z(\lambda)=g(\lambda;z)=\frac{1}{[\lambda]_K-z}.
$$
 Denote the real and imaginary parts of $g_z$ respectively as $g_{z,R}$ and $g_{z,I}$. It is not hard to see that for all $z\in \mathcal{C}_j$, $g_{z,R}(\lambda^2)$ and $g_{z,I}(\lambda^2)$ are both Lipschitz with constant $2K\eta^{-2}$. Further note that on the event $Q_n$, the following relation
 $$
 m_n(z)=n^{-1}\sum_i\frac{1}{\lambda_i-z}=n^{-1}\sum_ig(\lambda_i)=\int g(x) F_n(dx)
 $$
 is valid. Therefore, by Lemma \ref{lem:concen.spectral.S}, we have the following bound that
\begin{align*}
\text{Var}[m_n(z)\mathbbm{1}_{Q_n}]=O(\eta^{-4}n^{-2}).
\end{align*}
We complete the proof with
\begin{align*}
    \limsup_{n}\int_{\mathcal{C}_j}\E [|M_n^1(z)\mathbbm{1}_{Q_n}|^2]dz=O(v_0).
\end{align*}
\end{proof}

As an immediate corollary to Lemma \ref{cor:sos.gamma.gp}, we also have
\begin{corollary}\label{cor:f.mn1}
    Under Assumption \ref{assump_trunc}, let $f_1,\ldots,f_l$ be functions on $\mathbbm{R}$ analytic on an open interval containing (\ref{eq:open.int}), then there exists $\Lambda_n$, such that
    \begin{align*}
       ( \Lambda_n)^{-1/2}\left[\oint f_1(z)M_n^1(z)dz,\ldots, \oint f_l(z)M_n^1(z)dz\right]\rightarrow N(0,\text{Id}_l),
    \end{align*}
    with $\Lambda_n[i,j]=\oint\oint f_i(z_1)f_j(z_2)\sigma_n^2(z_1,z_2)dz_1dz_2$.The contours are closed, non-overlapping and taken in the positive direction in the complex plain, each containing the open interval (\ref{eq:open.int}).
\end{corollary}

\begin{proof}[Proof of Corollary \ref{cor:f.mn1}]
By the Cramer-Wold device, it suffices to prove that 
  \begin{align*}
      \frac{\sum \alpha_i\oint f_i(z)M_n^1(z)dz}{\sqrt{\alpha^T\Lambda_n\alpha}}=\frac{\oint \sum \alpha_i f_i(z)M_n^1(z)dz}{\sqrt{\alpha^T\Lambda_n\alpha}}\convd N(0,1).
  \end{align*}


\end{proof}

\section{Bias calculation}\label{sec:bias}
In this section, we study the deterministic component
\begin{align*}
    M_n^2(z)=\E \text{Tr}(B_n-z)^{-1}-\text{Tr}(\deteqv{B}_n-z)^{-1}=n(\E m_n(z)-\deteqv{m}_n(z)).
\end{align*}
For any $v_0>0$, define $C_0=\{z:\text{Im}z>v_0\}$.
 We shall verify that there exists $\mu_n(\cdot)$ computed from $L_r$, $\Sigma_r$, $n$ and $N$, such that for all $z\in C_0$,
\begin{align}\label{eq:mn2.conv.def}
    M_n^2(z)-\mu_n(z)=o_p(1).
\end{align}
Following a similar approach used for covariance calculations in Section \ref{subsec:cov.replace}, we revisit matrix $R$ defined in Equation (\ref{eq:def.R}). Our strategy involves substituting the random matrix $D$ with the deterministic $R$ and computing the residuals. 
Let
\begin{align*}
     M_n^2(z)=n[\E m_n(z)-\deteqv{m}_n(z)]=\left(\E \text{Tr}D^{-1}-\text{Tr}R^{-1}\right)+\left(\text{Tr}R^{-1}-n\deteqv{m}_n(z)\right).
\end{align*}
For the rest of this section, we analyze the convergence of these two components separately. To simplify the notation, define 
\begin{align*}
    d_{n0}&=\E \text{Tr}D^{-1}-\text{Tr}R^{-1}\\
    d_{nr}&=\E \text{Tr}(D^{-1}\Sigma_r)-\text{Tr}(R^{-1}\Sigma_r),r=1\ldots,k.
\end{align*}
\subsection{Convergence of $d_{nr}$}
By definition, we have
\begin{align*}
   D(z)-R(z)=\sum_{j=1}^N r_jr_j^T - \frac{1}{N}\sum_{j=1}^N\psi_j(z)T_j.
\end{align*}
Applying property (\ref{eq:D.Dj}) and the definition of $\beta_j$ gives
\begin{align}
    R^{-1}(z)-D^{-1}(z)
    &=R^{-1}(D-R)D^{-1}\\
    &=R^{-1}\left(\sum_{j=1}^N r_jr_j^TD^{-1} - \frac{1}{N}\sum_{j=1}^N\psi_j(z)T_jD^{-1}\right)\\
    &=R^{-1}\left(\sum_{j=1}^N\beta_j(z)r_jr_j^TD_j^{-1} - \frac{1}{N}\sum_{j=1}^N\psi_j(z)T_jD^{-1}\right).\label{eq:D.inv.R.inv}
\end{align}
Taking the trace and the expectation, then
\begin{align*}
    d_{n0}&=\E \text{Tr}D^{-1}-\text{Tr}R^{-1}\\
    &=-\sum_{j=1}^N\left(\E \beta_j(z)\text{Tr}r_j^TD_j^{-1}R^{-1}r_j - \frac{1}{N}\psi_j(z)\E \text{Tr}R^{-1}T_jD^{-1}\right)\\
    &=-\sum_{j=1}^N\E \beta_j(z)\left(\text{Tr}r_j^TD_j^{-1}R^{-1}r_j -\frac{1}{N}\E \text{Tr}R^{-1}T_jD_j^{-1}\right)\\
    &\quad - \frac{1}{N}\sum_{j=1}^N\E \beta_j(z)\left(\E \text{Tr}R^{-1}T_jD_j^{-1}-\E \text{Tr}R^{-1}T_jD^{-1}\right)\\
    &\quad - \frac{1}{N}\sum_{j=1}^N\E (\beta_j(z)-\psi_j(z))\left(\E \text{Tr}R^{-1}T_jD^{-1}\right)\\
    &:= \mathcal{J}_1+\mathcal{J}_2+\mathcal{J}_3.
\end{align*}
Adopting a similar approach outlined in Section \ref{subsec:cov.replace} for covariance calculation, we individually examine each of the three components. Detailed computations are deferred to Section \ref{subsec:sec:dn0.comp} in the supplementary appendices.
Collecting the terms yields
\begin{align*}
    d_{n0}
    &=\frac{1}{N^2}\sum_{j=1}^N\psi_j^2(z)\E [\text{Tr}R^{-1}T_jD^{-1}T_jD^{-1}]\\
    &\quad-\frac{1}{N^3}\sum_{j=1}^N\psi_j^3(z)\E [\text{Tr}D^{-1}T_jD^{-1}T_j]\E [\text{Tr}R^{-1}T_jD^{-1}]+o(1).
\end{align*}
Similarly, for each $r=1,\ldots,k$,  we have
\begin{align*}
    d_{nr}&=\E \text{Tr}\left(D^{-1}\Sigma_r\right)-\text{Tr}\left(R^{-1}\Sigma_r\right)\\
    &=\frac{1}{N^2}\sum_{j=1}^N\psi_j^2(z)\E [\text{Tr}D^{-1}T_jD^{-1}\Sigma_rR^{-1}T_j]\\
    &\quad-\frac{1}{N^3}\sum_{j=1}^N\psi_j^3(z)\E [\text{Tr}D^{-1}T_jD^{-1}T_j]\E [\text{Tr}R^{-1}T_jD^{-1}\Sigma_r]+o(1).
\end{align*}
Following a similar decomposition of $R^{-1}-D^{-1}$ as established in (\ref{eq:decom.rj.dj}), we apply Lemma \ref{lem:A123} to substitute $D^{-1}$ with $R^{-1}$ and conclude that
\begin{align}
    &\frac{1}{N}\E [\text{Tr}R^{-1}\Sigma_rD^{-1}\Sigma_s]=\frac{1}{N}\E [\text{Tr}R^{-1}\Sigma_rR^{-1}\Sigma_s]+o(1)\label{eq:dd,dr,rr}\\
   &\frac{1}{N}\E [\text{Tr}R^{-1}\Sigma_rR^{-1}\Sigma_sD^{-1}]=\frac{1}{N}\E [\text{Tr}R^{-1}\Sigma_rR^{-1}\Sigma_sR^{-1}]+o(1)\label{eq:rdd,rrd,rrr}\\
   &\frac{1}{N}\E [\text{Tr}R^{-1}\Sigma_rD^{-1}]=\frac{1}{N}\E [\text{Tr}R^{-1}\Sigma_rR^{-1}]+o(1).
\end{align} 
The remaining terms $\E [\text{Tr}D^{-1}\Sigma_rD^{-1}R^{-1}\Sigma_s]$,  $\E [\text{Tr}D^{-1}\Sigma_rD^{-1}\Sigma_s]$, $\E [\text{Tr}D^{-1}\Sigma_rD^{-1}\Sigma_sR^{-1}\Sigma_l]$ for $r$, $s$, $l=1,\ldots, k$, can be evaluated following a similar argument leading to the evaluation of $ \text{Tr}(\E_jA_1(z_1)T_j\E_jD_j^{-1}(z_2)\Sigma_r)$ in (\ref{eq:A1..}). We present the results in the following lemma. 
\begin{lemma}\label{lem:conv.dnr}
    Under Assumption \ref{assump_trunc}, for any $z\in C_0$, and any deterministic matrix $M$ with bounded spectral norm, for $a=1,\ldots,k$, we have
    \begin{align*}
        \text{Tr}(D^{-1}\Sigma_aD^{-1}M)
    &=\text{Tr}(R^{-1}\Sigma_aR^{-1}M)\\
    &+\frac{1}{N^2}\sum_{j=1}^N\psi_j^2\E \text{Tr}D^{-1}T_jD^{-1}\Sigma_a\E \text{Tr}R^{-1}T_jR^{-1}M+ a'(z),
    \end{align*}
    where $\E a'(z)\leq O(N^{1/2})$.
\end{lemma}
In addition to the terms $\deteqv{b}_j$, $h_{N+1}^{ab}$ and $\deteqv{R}$, which are defined in Lemma \ref{lem:Mn1.prelim}, (\ref{eq:xi.h.Lam}) and (\ref{eq:R.deteqv}), for each $a,b,c=1,\ldots,k$, further define
\begin{gather}
    \Xi_0^{a}=\frac{1}{N}\text{Tr}\left[\deteqv{R}^{-1}(z)\Sigma_a\deteqv{R}^{-1}(z)\right],\quad \Xi_1^{ab}=\frac{1}{N}\text{Tr}\left[\deteqv{R}^{-1}(z)\Sigma_a\deteqv{R}^{-1}(z)\Sigma_b\right],\label{eq:xi.0.a}\\ \Xi_2^{ab}=\frac{1}{N}\text{Tr}\left[\deteqv{R}^{-1}(z)\Sigma_a\deteqv{R}^{-1}(z)\deteqv{R}^{-1}(z)\Sigma_b\right],\\
    \Xi_3^{abc}=\frac{1}{N}\text{Tr}\left[\deteqv{R}^{-1}(z)\Sigma_a\deteqv{R}^{-1}(z)\Sigma_b\deteqv{R}^{-1}(z)\Sigma_c\right],\quad h^{abc}=\frac{1}{N}\sum_{j=1}^N\deteqv{b}_j^3(z)l_{aj}^2l_{bj}^2l_{cj}^2.
\end{gather}
For subsequent analyses, we set $h^{ab}_{N+1}=h^{ab}_{N+1}(z,z)$. Based on the formula derived in Lemma \ref{lem:conv.dnr}, let $\zeta_1^{ab}$, $\zeta_2^{ab}$, and $\zeta_3^{abc}$ denote the solutions to the systems of equations below:
\begin{align*}
    &\begin{cases}
        \zeta_1^{ab}=\Xi_1^{ab}+\sum_{r=1}^k\sum_{s=1}^kh_{N+1}^{rs}\zeta_1^{ar}\Xi_1^{bs}\quad a,b=1,\ldots,k.
    \end{cases}\\
    &\begin{cases}
        \zeta_2^{ab}=\Xi_2^{ab}+\sum_{r=1}^k\sum_{s=1}^kh_{N+1}^{rs}\zeta_1^{ar}\Xi_2^{bs}\quad a,b=1,\ldots,k.
    \end{cases}\\
    &\begin{cases}
        \zeta_3^{abc}=\Xi_3^{abc}+\sum_{r=1}^k\sum_{s=1}^kh_{N+1}^{rs}\zeta_1^{ar}\Xi_3^{bcs}\quad {a,b,c=1,\ldots,k.}
    \end{cases}
\end{align*}
We are now equipped to define the equivalents to ${d}_{nr}$, $r=0,\ldots,k$ as
\begin{align}\label{eq:dnr.bias.deteqv}
    \deteqv{d}_{n0}&=\sum_{a=1}^k\sum_{b=1}^kh_{N+1}^{ab}\zeta_2^{ab}-\sum_{a=1}^k\sum_{b=1}^k\sum_{c=1}^kh^{abc}\zeta_1^{ab}\Xi_0^c,\\
    \deteqv{d}_{nr}&=\sum_{a=1}^k\sum_{b=1}^kh_{N+1}^{ab}\zeta_3^{abr}-\sum_{a=1}^k\sum_{b=1}^k\sum_{c=1}^kh^{abc}\zeta_1^{ab}\Xi_1^{cr}, \quad {r=1,\ldots k},
\end{align}
satisfying ${d}_{nr}=\deteqv{d}_{nr}+o(1)$, $r=0,\ldots k$.
\subsection{Convergence of $\text{Tr}R^{-1}-n\deteqv{m}_n(z)$}
From the fact that $A^{-1}+B^{-1}=A^{-1}(A+B)B^{-1}$, along with property (\ref{eq:mn.deteqv.g1}),
\begin{align*}
    &\text{Tr}R^{-1}-n\deteqv{m}_n(z)(z)\\
    =&\text{Tr}\left({\sum_{s=1}^k}(N^{-1}\sum l_{sj}^2\psi_j(z))\Sigma_s-z\right)^{-1}+\text{Tr}\left(\sum_{s=1}^kz\deteqv{g}_1^{(s)}\Sigma_s+z \right)^{-1}\\ 
    =&-{\sum_{s=1}^k}(N^{-1}\sum_{j=1}^N l_{sj}^2\psi_j(z)+z\deteqv{g}_1^{(s)})\text{Tr}(R^{-1}\Sigma_s\deteqv{R}^{-1}),
\end{align*}
with $\deteqv{R}$ defined in (\ref{eq:R.deteqv}). 
By definitions of $\psi_j$ (\ref{eq:def.psi}), $g_i^{(r)}$ (\ref{eq:random.stiel}) and $\deteqv{g}_i^{(r)}$ (\ref{eq:m0}), we have
\begin{align*}
    \psi_j(z)=\frac{1}{1+\sum_{r=1}^kl_{rj}^2\E g_2^{(r)}},\quad z\deteqv{g}_1^{(s)}&=-\frac{1}{N}\sum_{j=1}^N\frac{l_{sj}^2}{1+\sum_{r=1}^kl_{rj}^2\deteqv{g}_2^{(r)}}.
\end{align*}
As a result, we get
\begin{align*}
    N^{-1}\sum_{j=1}^N l_{sj}^2\psi_j(z)+z\deteqv{g}_1^{(s)}
    &=\sum_{r=1}^k(\E {g}_2^{(r)}-\deteqv{g}_2^{(r)})\frac{1}{N}\sum_{j=1}^Nl_{sj}^2l_{rj}^2\deteqv{b}_j(z)\psi_j(z),
\end{align*}
where $\deteqv{b}$ is defined in Lemma \ref{lem:Mn1.prelim}.
Therefore,
\begin{align}\label{eq:mn2.interm}
     M_n^2(z)=&d_{n0}+\sum_{r=1}^kn(\E {g}_2^{(r)}-\deteqv{g}_2^{(r)}){\sum_{s=1}^k}\frac{1}{N}\sum_{j=1}^Nl_{sj}^2l_{rj}^2\deteqv{b}_j(z)\psi_j(z)\frac{1}{n}\text{Tr}(R^{-1}\Sigma_s\deteqv{R}^{-1}).
\end{align}
Similarly, for each $r=1,\ldots,k$,
\begin{align}
    &n[\E {g}_2^{(r)}-\deteqv{g}_2^{(r)}]=n/N\left[\text{Tr}\E D^{-1}\Sigma_r-\text{Tr}R^{-1}\Sigma_r\right]+n/N\left[\text{Tr}R^{-1}\Sigma_r-n\deteqv{g}_2^{(r)}\right]\nonumber\\
    =&n/Nd_{nr}-\sum_{t=1}^kn(\E {g}_2^{(t)}-\deteqv{g}_2^{(t)}){\sum_{s=1}^k}\frac{1}{N}\sum_{j=1}^Nl_{sj}^2l_{rj}^2\deteqv{b}_j(z)\psi_j(z)\frac{1}{N}\text{Tr}(R^{-1}\Sigma_s\deteqv{R}^{-1}\Sigma_r).\label{eq:ng}
\end{align}
From the equivalents $\deteqv{d}_{nr}$ defined in (\ref{eq:dnr.bias.deteqv}), we can further define the equivalents $\nu_r$ as solutions to the system of equations
\begin{align}
    \begin{cases}
        \nu_r=n/N\deteqv{d}_{nr}-\sum_{t=1}^T\nu_t{\sum_{s=1}^k}h_{N+1}^{st}\Xi_1^{sr}\quad r=1,\ldots, k,
    \end{cases}\label{eq:nu.r}
\end{align}
such that $n[\E {g}_2^{(r)}-\deteqv{g}_2^{(r)}]-\nu_r=o(1)$.
Finally, we define
\begin{align*}
    \mu_n(z)=\deteqv{d}_{n0}(z)+n/N\sum_{r=1}^k\nu_r(z){\sum_{s=1}^k}h_{N+1}^{rs}(z,z)\Xi_0^s(z),
\end{align*}
and conclude that $M_n^2(z)=\mu_n(z)+o(1)$.

\section*{Acknowledgements}
The authors are supported in part by NIH Grant R01 GM134483 and NSF Grant DMS-1811614.

\newpage
\printbibliography
\newpage
\appendix
\section{Preliminaries and tools}
\subsection{Construction of Stieltjes transform}\label{sec:comp.stieltjes}
Recall that matrix $B_n$ defined as (\ref{eq:Bn}) can be interpreted as a sum of squares matrix (\ref{eq:nested.sim.1}) within a nested $k$-level random effects model, with $(n,N)=(p,F_1)$. Following the notations defined in Section \ref{subsec:multi.random}, we can equivalently write $L_r$ as $L_r=V_1^TU_r(U_r^TU_r)U_r^TV_1$, where $V_1=U_1(U_1^TU_1)^{-1/2}$. Similarly, we write $$B_p=F_1^{-1}Y^TU_1(U_1^TU_1)^{-1}U_1^TY=F_1^{-1}Y^TV_1V_1^TY.$$
Let $\deteqv{b}_r=-z\deteqv{g}_1^{(r)}$, $\deteqv{a}_r=F_1/F_r\deteqv{g}_2^{(r)}$.
Following the construction proposed in Theorem 1.2 of \cite{fan2017eigenvalue}, define $\deteqv{B}_p$ as 
\begin{align}\label{app:eq.det.bn}
    \deteqv{B}_p=-z\sum_{r=1}^k\deteqv{g}_1^{(r)}\Sigma_r.
\end{align}
Then, we can immediately define
\begin{align}
\begin{cases}
    zm_{\deteqv{B}_p}&=-\frac{1}{p}\text{Tr}\left((\sum_{s=1}^k\deteqv{g}_1^{(s)}\Sigma_s+\text{Id} )^{-1}\right)\\
    z\deteqv{g}_2^{(r)}&=-\frac{1}{F_1}\text{Tr}\left((\sum_{s=1}^k\deteqv{g}_1^{(s)}\Sigma_s+\text{Id})^{-1}\Sigma_r\right),r=1,\ldots,k
    \end{cases}
\end{align}
To compute $\deteqv{g}_1$, we generalize the approach in Appendix A.1 of \cite{fan2017eigenvalue}. In particular, let $Q=\text{diag}(Q_1,\ldots, Q_k)\in\mathbbm{R}^{F_+\times F_+}$, where $Q_r=(U_r^TU_r)^{-1}U_r^T[V_1|V_r]$, $F^+=\sum_rF_r$. Here, $V_r$ represents the orthogonal basis of $\text{col}(U_r)\setminus \text{col}(U_1)$. In other words, $V_1^TV_r=0$, $r>1$. Now, define $U=(\sqrt{F_1}U_1|\ldots|\sqrt{F_k}U_k)$ and
\begin{align*}
    M=F_1^{-1}Q^TU^TV_1V_1^TUQ=RR^T,
\end{align*}
where $R^T=[L_1|\sqrt{F_2/F_1}L_2 \quad 0|\ldots|\sqrt{F_k/F_1}L_k \quad 0]$ is $F_1\times F_+$. Now, we would like to compute the block traces of
\begin{align*}
    S=(\text{Id}+MD(a))^{-1}M=(\text{Id}+RR^TD(a))^{-1}RR^T.
\end{align*}
By the Woodbury matrix identity,
\begin{align*}
    S=(\text{Id}-R(\text{Id}+R^TD(a)R)^{-1}R^TD(a))RR^T=R(\text{Id}+R^TD(a)R)^{-1}R^T.
\end{align*}
Therefore, $\deteqv{a}_r=F_1^{-1}\text{Tr}\left((\text{Id}+\sum_{s=1}^kF_s/F_1\deteqv{a}_sL_s^2)^{-2}L_r^2\right)$, and a linear transformation yields
\begin{align*}
    z\deteqv{g}_1^{(r)}=-\frac{1}{F_1}\text{Tr}\left((\sum_{s=1}^k\deteqv{g}_1^{(s)}\Sigma_s+\text{Id})^{-1}L_r^2\right),r=1,\ldots,k.
\end{align*}


\subsection{Proof of Lemma \ref{lem:det.equiv}}
We shall adapt the proof strategy for Theorem 1.2 in \cite{fan2017eigenvalue}. As derived in Section \ref{subsec:multi.random}, any matrix of the form (\ref{eq:Bn}) can be interpreted as a sum of squares matrix for a nested multivariate linear random effects model as defined in (\ref{eq:nested.ori}). Plugging in $Y$ defined in (\ref{eq:Y}), we have
\begin{align*}
    B_p=\sum_{r,s=1}^k \alpha_r^T\mathsf{F}_{rs}\alpha_s,
\end{align*}
where $\mathsf{F}_{rs}=\frac{\sqrt{F_rF_s}}{F_1}U_r^TU_1(U_1^TU_1)^{-1}U_1^TU_s$, and $\alpha_r=\sqrt{F_r}G_r\Sigma_r^{1/2}$, with $G_r\in\mathbbm{R}^{F_r\times p}$, $r=1,\ldots,k$ being independent matrices with i.i.d. Gaussian entries of variance $F_r^{-1}$. Now, define
\begin{align*}
     W=B_p=\frac{1}{F_1}\sum_{r,s=1}^k \Sigma_r^{1/2}G_r^T\mathsf{F}_{rs}G_s\Sigma_s^{1/2},
\end{align*}
Applying Theorem 1.2 in \cite{fan2017eigenvalue}, we immediately have 
\begin{align*}
    m_{n}-\deteqv{m}_{n}=o(1).
\end{align*}
To prove the convergence of $g_i^{(j)}-\deteqv{g}_i^{(j)}$, we adapt a similar proof. First, we define some notations. Recall the definitions of non-commutative probability space, rectangular probability space, and $\mathcal{B}$-valued probability space in Section 3 of \cite{fan2017eigenvalue}. 
Following the construction proposed in Section 4.1 of \cite{fan2017eigenvalue}, let $O_0$,...,$O_{2k}$ be independent Haar-distributed orthogonal matrices, and consider the transformations
    \begin{align*}
        \Sigma_r^{1/2}\rightarrow H_r:= O_r^T\Sigma_r^{1/2}O_0,\quad \mathsf{F}_{rs}\rightarrow O_{k+r}^T\mathsf{F}_{rs}O_{k+s}.
    \end{align*}
Now, $m_r$ and $n_r$ in Section 4.1 of \cite{fan2017eigenvalue} respectively correspond to $p$ and $F_r$ in our case. Similarly, let $N=p+\sum_rm_r+\sum_rn_r$, we define the elements $\Tilde{W}$, $\{\Tilde{F}_{rs}\}_{r,s=1}^k$, $\{\Tilde{H}_r\}_{r=1}^k$, $\{\Tilde{G}_r\}_{r=1}^k$, $P_0,\ldots, P_{2k}$, $\{f_{rs}\}_{r,s=1}^k$, $\{h_r\}_{r=1}^k$, $\{g_r\}_{r=1}^k$, $p_0,\ldots, p_{2k}$, and the two rectangular spaces 
$(\mathbbm{C}^{N\times N},N^{-1}\text{Tr},P_0,\ldots,P_{2k})$, $(\mathcal{A},\tau,p_0,\ldots,p_{2k})$, the  sub-* algebra $\mathcal{D}=\langle p_r:0\leq r\leq 2k\rangle$, and the von Neumann sub-algebra $\mathcal{H}=\langle \mathcal{D}, \{h_r\}\rangle_{W^*}
$, where $\langle\cdot\rangle_{W^*}$ denotes the ultraweak closure. 

For $z\in\mathbbm{D}_0=\{z\in\mathbbm{C}^+:\text{Im}z>C_0\}$ with $C_0$ large enough, define 
\begin{align*}
        \alpha_r(z)=\tau_r(h_rG_{\omega}^{\mathcal{H}}(z)h_r^*)
    \end{align*}
as in equation (4.8) of \cite{fan2017eigenvalue}, where $G_{\omega}^{\mathcal{H}}(z)$ stands for $\mathcal{H}$-valued Cauchy transform of $\omega$ for $z$ defined in Section 3.3 of \cite{fan2017eigenvalue}, and $\tau_r$ is defined as $\tau_r(a)=\tau(p_rap_r)/\tau(p_r)$ for any $a$. 
Following the argument leading to (4.23) in \cite{fan2017eigenvalue}, we can verify that
\begin{align*}
    \deteqv{g}_2^{(r)}=-\frac{1}{n_1}\text{Tr}\left((z\text{Id}-{\textstyle\sum}_{s=1}^k\deteqv{b}_s\Sigma_s)^{-1}\Sigma_r\right)=-\frac{m_r}{n_1}\alpha_r,
\end{align*}
where $\deteqv{b}_r$ is defined above equation (\ref{app:eq.det.bn}) in Section \ref{sec:comp.stieltjes}. Following the argument in Step 4 in the proof of Lemma 4.4 in \cite{fan2017eigenvalue}, we can further verify that 
    \begin{align*}
        \deteqv{g}_2^{(r)}(z)&=-\frac{N}{n_1}\sum_{l=0}^{\infty}z^{-(l+1)}\tau(\omega^lh_r^*h_r).
    \end{align*}
By definition,  we also have
    \begin{align*}
   {g}_2^{(2)}=\frac{N}{n_1}N^{-1}\text{Tr}\left[(W-z)^{-1}\Sigma_r\right]&=\frac{N}{n_1}\sum_{l=0}^{\infty} z^{-(l+1)}N^{-1}\text{Tr}\left[W^l\Sigma_r\right]\\
   &=\frac{N}{n_1}\sum_{l=0}^{\infty} z^{-(l+1)}N^{-1}\text{Tr}\left[W^lH_r^*H_r\right].
\end{align*}
Applying Theorem 3.10 and adapting the proof of Corollary 3.11 in \cite{fan2017eigenvalue}, we have for $z\in\mathbbm{D}$,
\begin{align*}
    \deteqv{g}_2^{(r)}(z)-{g}_2^{(r)}(z)\convas 0.
\end{align*}

Finally, since $\deteqv{g}_1^{(r)}$ and $g_1^{(r)}$ are respectively defined as linear functions of $\deteqv{g}_2^{(r)}$ and $g_2^{(r)}$, we immediately have that
\begin{align*}
    \deteqv{g}_1^{(r)}(z)-{g}_1^{(r)}(z)\convas 0.
\end{align*}
Since $\deteqv{g}_i^{(r)}-{g}_i^{(r)}$ is bounded over $\{z\in\mathbbm{C}^+:\text{Im}z>\epsilon\}$ for any $\epsilon>0$, Lemma C.1 in \cite{fan2017eigenvalue} implies that $\deteqv{g}_i^{(r)}-{g}_i^{(r)}\convas 0$ for any $z\in\mathbbm{C}^+$.

\subsection{Proof of Lemma \ref{lem:trunc}}
 Before proceeding with the proof, we note that our proof is self-contained, it simplifies the justification for the truncation step described by \cite{baisilverstein2004}, adapting it from moments assumptions to the Gaussian assumption.

Under Assumption \ref{assump}, by Gaussianity, we have $\E |x_{11}|^4=3<\infty$. For $m=1,2,\ldots$, find $n_m$ ($n_m>n_{m-1}$) satisfying
\begin{align*}
    m^4\E \left[|x_{11}|^4\mathbbm{1}[|x_{11}|\geq \sqrt{n}/m]\right]<2^{-m}.
\end{align*}
for all $n\geq n_m$. Define $\delta_n'=1/m$ for all $n\in [n_m,n_{m+1})$ (=1 for $n<n_1$). Then, as $n\rightarrow 0$, $\delta_n'\rightarrow 0$ and
\begin{align*}
    \delta_n'^{-4}\E \left[|x_{11}|^4\mathbbm{1}[|x_{11}|\geq \delta_n'\sqrt{n}]\right]\rightarrow 0.
\end{align*}
Define $\delta_n=\max\{\delta_n',n^{-1/8}\}$, then the following conditions hold:
\begin{align}\label{eq:deltan.prop}
    \delta_n\rightarrow 0, \quad\delta_nn^{1/4}\rightarrow \infty,\quad\delta_n^{-4}\E \left[|x_{11}|^4\mathbbm{1}[|x_{11}|\geq \delta_n\sqrt{n}]\right]\rightarrow 0.
\end{align}
Let $\check{x}_{ji}=x_{ji}\mathbbm{1}_{|x_{ji}|\leq \delta_n\sqrt{n}}$, and $\check{B}_n=\frac{1}{N}\sum_{j=1}^NT_j^{1/2}\check{x}_j\check{x}_j^TT_j^{1/2}$, then
\begin{align*}
    \mathbbm{P}(B_n\neq \check{B}_n)&\leq nN\mathbbm{P}(|x_{11}|\geq \delta_n\sqrt{n})\\
    &= nN\E \left[\mathbbm{1}[|x_{11}|\geq \delta_n\sqrt{n}]\right]\\
    &\leq nN\E \left[\frac{|x_{11}|^4}{\delta_n^4n^2}\mathbbm{1}[|x_{11}|\geq \delta_n\sqrt{n}]\right]\\
    &\leq K\delta_n^{-4}\E \left[|x_{11}|^4\mathbbm{1}[|x_{11}|\geq \delta_n\sqrt{n}]\right]\rightarrow 0
\end{align*}
for some constant $K$. In other words, for any function $f$ satisfying the assumptions in Theorem \ref{thm:main}, we have
\begin{align}\label{eq:gn.check.gn}
   \int fd G_n(x)=\int fd\check{G}_n(x)+o_p(1),
\end{align}
where $\check{G}_n$ is analogue of $G_n$ with the matrix $B_n$ replaced by $\check{B}_n$.

Now, let $\hat{x}_{ji}=\check{x}_{ji}/\sigma_n$, where 
\begin{align*}
    \sigma_n^2&=\text{Var}(\check{x}_{ji})=\E |\check{x}_{ji}|^2\\
    &=1-\frac{2\delta_n\sqrt{n}\varphi(\delta_n\sqrt{n})}{\Phi(\delta_n\sqrt{n})-\Phi(-\delta_n\sqrt{n})}\rightarrow 1
\end{align*}
follows from Gaussianity. Above, $\varphi$ and $\Phi$ respectively represent the probability density function and cumulative density function of the standard normal distribution. Equation  (\ref{eq:deltan.prop}) yields
\begin{align*}
    \E |\hat{x}_{11}|^4=\frac{\E |\check{x}_{11}|^4}{\sigma_n^4}=\frac{\E |{x}_{11}|^4+o(\delta_n^4)}{\sigma_n^4}=3+o(1).
\end{align*}
Collecting the terms gives
\begin{align*}
    \hat{x}_{ji}<\delta_n\sqrt{n},\quad \E [\hat{x}_{ji}]=0,\quad \E |\hat{x}_{ji}|^2=1,\quad \E |\hat{x}_{ji}|^4=3+o(1).
\end{align*}
Let $\hat{B}_n=\frac{1}{N}\sum_{j=1}^NT_j^{1/2}\hat{x}_j\hat{x}_j^TT_j^{1/2}=\sigma_n^{-2}\check{B}_n$ and similarly define $\hat{G}_n$. Denote the $i$th smallest eigenvalue of a PSD matrix $A$ by $\lambda_i^A$, then 
\begin{align*}
     \left|\int fd\check{G}_n(x)-\int fd \hat{G}_n(x)\right|&\leq K\sum_{i=1}^n|\lambda_i^{\check{B}_n}-\lambda_i^{\hat{B_n}}|\\
     &\leq K(\sigma_n^{2}-1)\sum_{i=1}^n|\lambda_i^{\hat{B}_n}|\\
    &\leq Kn\frac{2\delta_n\sqrt{n}\varphi(\delta_n\sqrt{n})}{(\Phi(\delta_n\sqrt{n})-\Phi(-\delta_n\sqrt{n}))}\lambda_{\max}^{\hat{B}_n}.
\end{align*}
Above, $K$ is the upper bound of $f'$, and by the property (\ref{eq:concen.bn}), we have that $\lambda_{\max}^{\hat{B}_n}$ is almost surely bounded by a deterministic finite value. Finally, direct calculations show
\begin{align*}
    \delta_n\sqrt{n}\varphi(\delta_n\sqrt{n})=C\frac{1}{\delta_n^5n^{3/2}}\frac{(\delta_n^2n)^3}{e^{\delta_n^2n}}\rightarrow 0.
\end{align*}
As a result, together with equation (\ref{eq:gn.check.gn}), we arrive at the conclusion
\begin{align*}
     \int fd G_n(x)=\int fd\hat{G}_n(x)+o_p(1).
\end{align*}



\subsection{Proof of Lemma \ref{lem:concen}}
Before proving the lemma, we first introduce a few helpful results.
\begin{prop}[Rephrased from Lemma 2.7 \cite{baisil1998}]\label{prop.concen}
 Let $A$ be an $n\times n$ non-random (possibly complex) matrix. Suppose $x_1=(x_{11},...,x_{1n})^T$ is a (possibly complex) vector with i.i.d mean zero entries with $\E |x_{1i}|^2=1$, $\E |x_{1i}|^l\leq \nu_l$, for $i=1,\ldots,n$, $l\leq 2p$, $p\geq 1$. Then
 \begin{align}
     \E |x_1^*Ax_1-\text{tr}A|^p\leq K_p[\nu_4\text{tr}(AA^*)]^{p/2}+K_p\nu_{2p}\text{tr}\left[(AA^*)^{p/2}\right]\label{eq:lem2.7}
 \end{align}
 for some constant $K_p$ depending on $p$ only.
\end{prop}
\begin{lemma}\label{lem:basic}
    For any matrix $A\in\mathbbm{C}^{n\times N}$ with singular values $\sigma_1(A)\geq ...\geq \sigma_{n\wedge	N}(A)\geq 0$, the following hold true
    \begin{enumerate}
        \item the spectral norm satisfies
        \begin{align*}
           \lVert A\rVert =\sigma_1(A)=|\lambda_{\max}(A^*A)|^{1/2}
        \end{align*}
        \item for $r\geq 1$,
        \begin{align*}
            \text{tr}\left[(AA^*)^{r/2}\right]=\sum_{k=1}^{n\wedge	N}\lambda_k^{r/2}(AA^*)\leq (n\wedge	N)\lVert A\rVert^r\leq 2c^{-1}n\lVert A\rVert^r.
        \end{align*}
    \end{enumerate}
\end{lemma}
\begin{proof}[Proof of Lemma \ref{lem:concen}]
    Recall from assumption that $x_{1i}$ are i.i.d., with $|x_{1i}|<\delta_n\sqrt{n}$, $\E [x_{1i}]=0$, $\E |x_{1i}|^2=1$, $\E |x_{1i}|^4=3+o(1)$, where $\delta_n\rightarrow 0$, $n\delta_n^4\rightarrow\infty$. From this, we have
    \begin{align*}
        \nu_{2p}=\E |x_{11}|^{2p}\leq \delta_n^{2p-4}{n}^{p-2}\nu_4.
    \end{align*}
    Taken together with Lemma \ref{lem:basic}, Proposition \ref{prop.concen} implies that for each $B_l$,
    \begin{align}
        \E |x_1^TB_lx_1-\text{tr}B_l|^p
        &\leq K_p\lVert B_l\rVert^p\{(2c^{-1}\nu_4n)^{p/2}+2c^{-1}\nu_4\delta_n^{2p-4}{n}^{p-1}\}\\
        &\leq K_p\lVert B_l\rVert^p\delta_n^{2p-4}{n}^{p-1}.\label{eq:interm}
    \end{align}
   Set $W_l=N^{-1}(x_1^TB_lx_1-\text{tr}B
    _l)$,  $\lVert W\rVert_q:=(\E |W|^q)^{1/q}$,  then from equation (\ref{eq:interm}) we have
    \begin{align*}
        \lVert W\rVert_q\leq  K_p\lVert B_l\rVert^p\delta_n^{2p-4}{n}^{p-1}.
    \end{align*}
    For an integer $q$, Hölder's inequality implies that 
    \begin{align*}
        \left\lVert \prod_{l=1 }^qW_l\right\rVert_1\leq  \prod_{l=1 }^q\lVert W_l\rVert_q.
    \end{align*}
    Collecting the terms yields    \begin{align*}
       \left| \prod_{l=1 }^qW_l\right|
       \leq \left\lVert \prod_{l=1 }^qW_l\right\rVert_1\leq  K_q(2c)^{q-1}N^{-1}\delta_n^{2q-4}\prod_{l=1}^q\lVert B_l\rVert
    \end{align*}
    as required.
\end{proof}

\section{Convergence of $M_n^1$}
\subsection{Proof of Lemma \ref{lem:Mn1.prelim}}
\begin{enumerate}
    \item \textbf{(i) boundedness of $|\beta_j|$, $|\Tilde{\beta}_j|$, $|b_j|$. }First, we introduce a useful definition and result.
    \begin{definition}\label{def:nevanlinna}
If $f: \mathbbm{C}^+\rightarrow\mathbbm{C}^+$ maps the upper half plane to itself, then $f$ is called a Nevanlinna function.
\end{definition}
\begin{lemma}\label{lem:nevanlinna}
    The following hold true for Nevanlinna functions:
    \begin{enumerate}
        \item for a function $g$, if the function $z\rightarrow zg(z)$ is Nevanlinna, then for $z\in \mathbbm{C}^+$,
        \begin{align*}
            \left|\frac{1}{1+g(z)}\right|\leq \frac{|z|}{\text{Im}z}
        \end{align*}
        \item if matrices $B$, $T$ are positive definite, and $x\in\mathbbm{R}^n$ then for each of the function $g$:
        \begin{align*}
            \frac{1}{N}x^TT^{1/2}(B-z)^{-1}T^{1/2}x,\quad \frac{1}{N}\text{Tr}(B-z)^{-1}T, \quad \frac{1}{N}\E \text{Tr}(B-z)^{-1}T,
        \end{align*}
        $z\rightarrow zg(z)$ is Nevanlinna.
    \end{enumerate}
\end{lemma}
\begin{proof}[Proof of Lemma \ref{lem:nevanlinna}]
   (a)  Note that
        \begin{align*}
            \left|\frac{1}{1+g(z)}\right|=\frac{|z|}{|z+zg(z)|},
        \end{align*}
        where 
        \begin{align*}
            \text{Im}(z+zg(z))>\text{Im}z.
        \end{align*}
        Hence 
        \begin{align*}
            |z+zg(z)|^{-1}\leq (\text{Im}z)^{-1}.
        \end{align*}

        (b) From the spectral decomposition $B=U\Lambda U^T$, we can transform the functions $g$ into
        \begin{gather}\label{eq:transform.g}
\frac{1}{N}\text{Tr}\left[\text{diag}\left(\frac{1}{\lambda_j-z}\right)U^TT^{1/2}xx^TT^{1/2}U\right],\\
    \frac{1}{N}\text{Tr}\left[\text{diag}\left(\frac{1}{\lambda_j-z}\right)U^TTU\right], \quad\frac{1}{N}\E \text{Tr}\left[\text{diag}\left(\frac{1}{\lambda_j-z}\right)U^TTU\right].
\end{gather}
For any $\lambda>0$ and $z\in\mathbbm{C}^+$, we have
\begin{align*}
    \text{Im}\frac{z}{\lambda-z}=\frac{\text{Im}z(\lambda-\Bar{z})}{|\lambda_j-z|^2}=\frac{\lambda\text{Im}z}{|\lambda_j-z|^2}>0.
\end{align*}
Therefore, the functions $z\rightarrow zg(z)$ are Nevanlinna.
\end{proof}

As a direct corollary to Lemma \ref{lem:nevanlinna}, we have that
 \begin{align}\label{eq:bdd.app}
        |\beta_j|,\quad |\Tilde{\beta}_j|,\quad |b_j|,\quad |\psi_j|\leq |z|v^{-1}.
    \end{align}

\textbf{(ii) boundedness of $|b_j(z)-\psi_j(z)|$. }
Introduce 
\begin{align}\label{eq:def.rhoj}
    \rho_j(z)=r_j^TD_j^{-1}(z)r_j-N^{-1}\E [\text{Tr}(D_j^{-1}(z)T_j)].
\end{align}
Note that 
\begin{align*}
   \beta_j(z)-b_j(z)&=\frac{N^{-1}\E [\text{Tr}(D_j^{-1}(z)T_j)]-r_j^TD_j^{-1}(z)r_j}{(1+r_j^TD_j^{-1}(z)r_j)(1+N^{-1}\E [\text{Tr}(D_j^{-1}(z)T_j)])}\\
   &=-\rho_j(z)\beta_j(z)b_j(z).
\end{align*}
Therefore,
\begin{align}
   \beta_j(z)&=b_j(z)-\rho_j(z)\beta_j(z)b_j(z).\label{eq:beta.j.b.j}
\end{align}
Combining the property (\ref{eq:D.Dj})
with (\ref{eq:beta.j.b.j}) we have
\begin{align}
    \E [\text{Tr}(D_j^{-1}(z)T_j)-\text{Tr}(D^{-1}(z)T_j)]
    &=\E [\beta_j(z)\text{Tr}(r_j^TD_j^{-1}(z)T_jD_j^{-1}(z)r_j)]\nonumber\\
    &=b_j(z)\E [\text{Tr}(r_j^TD_j^{-1}(z)T_jD_j^{-1}(z)r_j)]\nonumber\\
    &\quad -b_j(z)\E [\beta_j(z)\rho_j(z)r_j^TD_j^{-1}(z)T_jD_j^{-1}(z)r_j].\label{eq:djtj.dtj}
\end{align}
Apply Lemma \ref{lem:concen}, then
\begin{align}
    &|\E [\beta_j(z)\rho_j(z)r_j^TD_j^{-1}(z)T_jD_j^{-1}(z)r_j]|\nonumber\\
    &\leq  |\E [\beta_j(z)\rho_j(z)\left(r_j^TD_j^{-1}(z)T_jD_j^{-1}(z)r_j-N^{-1}\E \text{Tr}(D_j^{-1}(z)T_jD_j^{-1}T_j)\right)]|\nonumber\\
    &+|\E [\beta_j(z)\rho_j(z)N^{-1}\E \text{Tr}(D_j^{-1}(z)T_jD_j^{-1}T_j)]|\leq C|z|v^{-4}N^{-1/2}.\label{eq:bd.betaj.rhoj}
\end{align}
As a result, by boundedness of $b_j$, $\psi_j$ derived in (\ref{eq:bdd.app}),
\begin{align}
    |b_j(z)-\psi_j(z)|&=\left|\frac{1}{N}b_j(z)\psi_j(z)\E [\text{Tr}(D_j^{-1}(z)T_j)-\text{Tr}(D^{-1}(z)T_j)]\right|\\
    &\leq|z|^3v^{-3}\left|\frac{1}{N^2}\E [\text{Tr}(D_j^{-1}(z)T_jD_j^{-1}(z)T_j)]\right|+C|z|^4v^{-7}N^{-3/2}\\
    &\leq C|z|^4v^{-7}N^{-1}.
\end{align}
\textbf{(iii) boundedness of $\lVert R_j\rVert$, $\lVert R\rVert$. }
Recall $R_j(z)=zI-\frac{1}{N}\sum_{i\neq j}\psi_i(z)T_i$, $R=\frac{1}{N}\sum_{j=1}^N\psi_j(z)T_j-zI$, $\psi_j(z)=\frac{1}{1+N^{-1}\E [\text{Tr}(D^{-1}(z)T_j)]}$.
Assume $B_n=U\Lambda U^T$, $z=\eta+vi$, then direct calculations show that
\begin{align*}
    \E [\text{Tr}(D^{-1}({z})T_j)]&= \E [\text{Tr}(U(\Lambda-{z}I)^{-1}U^TT_j)]\\
     &=\E [\text{Tr}(U(\Lambda-{\eta}I-viI)^{-1}U^TT_j)]\\
     &= \E [\text{Tr}(U\text{diag}(\frac{\lambda_i-\eta+vi}{(\lambda_i-\eta)^2+v^2})U^TT_j)]\\
     &= \E [\text{Tr}(U\text{diag}(\frac{\lambda_i-\eta}{(\lambda_i-\eta)^2+v^2})U^TT_j)]\\
     &+iv\E [\text{Tr}(U\text{diag}(\frac{1}{(\lambda_i-\eta)^2+v^2})U^TT_j)].
\end{align*}
Consequently, we have
\begin{align}
    R=&\frac{1}{N}\sum_{j=1}^N|\psi_j|^2({1+\E [\text{Tr}(U\text{diag}(\frac{\lambda_i-\eta}{(\lambda_i-\eta)^2+v^2})U^TT_j)]})T_j-\eta I\nonumber\\
    &-iv\left(I+\frac{1}{N}\sum_{j=1}^N|\psi_j|^2({1+\E [\text{Tr}(U\text{diag}(\frac{1}{(\lambda_i-\eta)^2+v^2})U^TT_j)]})T_j\right)\label{eq:R.decomp}\\
    :=&A+Bi,\nonumber
\end{align}
where $A$ and $B$ are symmetric, and $B$ is negative definite. 
From the fact that $R^*R=A^2+B^2$, we get
\begin{align*}
    \lambda_{\max}(R^*R)\geq \lambda_{\max}(A^2,B^2)\geq \lambda_{\max}(B^2)=\lambda_{\max}(B)^2\geq v^2.
\end{align*}
Finally, we have $\lVert R^{-1}\rVert\leq v^{-1}$. Similarly, $\lVert R_j^{-1}\rVert\leq v^{-1}$, $j=1,\ldots,N$.



\item From the definition of $g_2^{(r)}$ in (\ref{eq:random.stiel}), we write
\begin{align*}
    N^{-1}\E [\text{Tr}(D^{-1}(z)T_j)]=\sum_{r=1}^kl_{rj}^2g_2^{(r)}.
\end{align*}
Applying Lemma \ref{lem:det.equiv}, it is straightforward to verify that
\begin{align*}
    \left|\psi_j(z)-\deteqv{b}_j(z)\right|=o(1).
\end{align*}
Taken together with the boundedness result (\ref{eq:bd.beta}), we get
\begin{align*}
    \left|b_j(z)-\deteqv{b}_j(z)\right|=o(1).
\end{align*}
\item Without loss of generality, let $j=1$. First, note that
\begin{align*}
   \Tilde{\beta}_1(z)-b_1(z) 
   &=\Tilde{\beta}_1(z)b_1(z)N^{-1}\left(\E [\text{Tr}(D_1^{-1}(z)T_1)]-\text{Tr}(D_1^{-1}(z)T_1)\right).
\end{align*}
Introduce the notations $D_{1i}=D-r_1r_1^T-r_ir_i^T$, $K_i=\text{Tr}T_1(D_1^{-1}-D_{1i}^{-1})$ for $i>1$, then
\begin{align*}
    \text{Tr}(D_1^{-1}(z)T_1)-\E [\text{Tr}(D_1^{-1}(z)T_1)]
    &=\sum_{i=2}^N(\E_i-\E_{i-1})\text{Tr}(D_1^{-1}(z)T_1)\\
    &=\sum_{i=2}^N(\E_i-\E_{i-1})\text{Tr}T_1(D_1^{-1}(z)-D_{1i}^{-1}(z))\\
    &=\sum_{i=2}^N(\E_i-\E_{i-1})K_i.
\end{align*}
By Lemma \ref{lemma:bound.trA}, we have that
\begin{align*}
    |K_i|\leq Cv^{-1}.
\end{align*}
Since $\{(\E_i-\E_{i-1})K_i\}$ are uncorrelated, we further get
\begin{align*}
    \E \left|\text{Tr}(D_1^{-1}(z)T_1)-\E [\text{Tr}(D_1^{-1}(z)T_1)]\right|^2
    &=\E \left|\sum_{i=2}^N(\E_i-\E_{i-1})K_i\right|^2\\
    &=\sum_{i=2}^N\E \left|(\E_i-\E_{i-1})K_i\right|^2\\
    &\leq 4\sum_{i=2}^N\E \left|K_i\right|^2\leq CNv^{-2}.
\end{align*}
Collecting the terms, together with the boundedness result (\ref{eq:bd.beta}), we obtain
\begin{align*}
    \E |\Tilde{\beta}_j(z)-b_j(z)|^2
 &\leq C|\Tilde{\beta}_j(z)|^2|b_j(z)|^2N^{-2}\E \left|\text{Tr}(D_1^{-1}(z)T_1)-\E [\text{Tr}(D_1^{-1}(z)T_1)]\right|^2\\
 &\leq C|z|^4v^{-6}N^{-1}.
\end{align*}
\item Let us write $\epsilon_j(z)$ as
\begin{align*}
    \epsilon_j(z)=N^{-1}x_j^TT_j^{1/2}D_j^{-1}(z)T_j^{1/2}x_j-N^{-1}\text{Tr}(T_j^{1/2}D_j^{-1}(z)T_j^{1/2}).
\end{align*}
Define $C_j:=T_j^{1/2}D_j^{-1}(z)T_j^{1/2}$, then
\begin{align*}
    \epsilon_j(z)=N^{-1}x_j^TC_jx_j-N^{-1}\text{Tr}(C_j).
\end{align*}
For $q\geq 1$, further take $B_{2k-1}=C_j$, $B_{2k}=C_j^*$, $k=1,...,q$, then the concentration result (\ref{eq:martingale.concentr}) implies
\begin{align*}
  \E |\epsilon_j(z)|^{2q}
   &= \E \prod_{l=1}^{2q}\left(N^{-1}x_j^TB_lx_j-N^{-1}\text{Tr}(B_l)\right)\\
    &\leq  \E \left|\E \left[\prod_{l=1}^{2q}\left(N^{-1}x_j^TB_lx_j-N^{-1}\text{Tr}(B_l)\right)\bigg| x_i,i\neq j\right]\right|\\
   &\leq KN^{-1}\delta_n^{4q-4}\prod_{l=1}^{2q}\lVert B_l\rVert\\
   &\leq KN^{-1}\delta_n^{4q-4}v^{-2q}.
\end{align*}
In a similar vein,
\begin{align*}
    \E |\gamma_j(z)|^{2q}\leq KN^{-1}\delta_n^{4q-4}v^{-4q}.
\end{align*}
\end{enumerate}
\subsection{Additional justifications for equation (\ref{eq:beta.eps})}
\begin{lemma}\label{app:lem.der}
    Denote $f_n(z_1,z_2)=\sum_{j=1}^N\E_{j-1}[\E_j(\Tilde{\beta}_j(z_1)\epsilon_j(z_1))\E_j(\Tilde{\beta}_j(z_2)\epsilon_j(z_2))]$. If there exists $\sigma_n'$ such that
    \begin{align*}
       \frac{f_n(z_1,z_2)}{ \sigma_n'^{2}(z_1,z_2)}\convp 1,
    \end{align*}
    then we also have
    \begin{align*}
       \frac{\Phi_n(z_1,z_2)}{ \sigma_n^{2}(z_1,z_2)}\convp 1,
    \end{align*}
    where $ \sigma_n^{2}(z_1,z_2)=\frac{\partial^2}{\partial z_2\partial z_1}\sigma_n'^{2}(z_1,z_2)$.
\end{lemma}
Before we prove this result, we first introduce the following useful lemma. 
\begin{lemma}[Lemma 2.3 in \cite{baisilverstein2004}]\label{app:lem.bai.der}
    Let $f_1,f_2,\ldots$ be analytic in $D$, a connected open set of $\mathbbm{C}$, satisfying $|f_n(z)|\leq M$ for every $n$ and $z$ in $D$, and $f_n(z)$ converges, as $n\rightarrow \infty$ for each $z$ in a subset of $D$ having a limit point in $D$. Then there exists a function $f$, analytic in $D$ for which $f_n(z)\rightarrow f(z)$ and $f_n'(z)\rightarrow f'(z)$ for all $z\in D$. Moreover, on any set bounded by a contour interior to $D$ the convergence is uniform and $\{f_n'(z)\}$ is uniformly bounded by $2M/\epsilon$, where $\epsilon$ is the distance between the contour and the boundary of $D$.
\end{lemma}
\begin{proof}[Proof of Lemma \ref{app:lem.der}]
    By Lemma \ref{lem:Mn1.prelim}, we have that for every $n$, $|f_n(z_1,z_2)|$ has an upper bound that only depends on $z_1$, $z_2$ and $v_0$. 

    We shall adapt the proof strategy on page 571 in \cite{baisilverstein2004}. The main idea is to apply Lemma \ref{app:lem.bai.der} separately to the cases where $\text{Im}z > v_0$ and $\text{Im}z < -v_0$. Suppose $\sigma_n'^{-2}(z_k,z_l)f_n(z_k,z_l)\convp 1$ for each $z_k$, $z_l\in \{z_i\}\subset \underline{D}=\{z:v_0<|\text{Im}z|<K\}$ for an arbitrary $K>v_0$ and that the sequence $\{z_i\}$ has two limit points, one with $\text{Im}z > v_0$ and the other with $\text{Im}z < -v_0$. By a diagonalization argument, we can find a subsequence of $n\in \mathbbm{N}$ such that $\sigma_n'^{-2}(z_k,z_l)f_n(z_k,z_l)$ converges simultaneously for each pair $z_k$, $z_l$. Now, for each $z_l\in\{z_i\}$, we apply Lemma \ref{app:lem.bai.der} to conclude that on each of $\{z:v_0<\text{Im}z<K\}$ and $\{z:-K<\text{Im}z<-v_0\}$, we have $\sigma_n'^{-2}(z,z_l)f_n(z,z_l)\convp 1$ and  $\frac{\partial}{\partial z}\sigma_n'^{-2}(z,z_l)f_n(z,z_l)\convp 0$. In other words,
    \begin{align*}
        \frac{\frac{\partial}{\partial z}f_n(z,z_l)}{\frac{\partial}{\partial z}\sigma_n'^{2}(z,z_l)}\rightarrow 1. 
    \end{align*}
    Here, the convergence is uniform.
    We conclude the proof by applying Lemma \ref{app:lem.bai.der} once more on the remaining variable. 
\end{proof}
\subsection{Additional justifications for equation (\ref{eq:b.eps})}
\begin{lemma}\label{lem:tilde.beta.b}
    Under Assumption \ref{assump_trunc},  we have
    \begin{align*}
        \sum_{j=1}^N\E_{j-1}[\E_j(\Tilde{\beta}_j(z_1)\epsilon_j(z_1))\E_j(\Tilde{\beta}_j(z_2)\epsilon_j(z_2))]-\sum_{j=1}^N\E_{j-1}[\E_j(b_j(z_1)\epsilon_j(z_1))\E_j(b_j(z_2)\epsilon_j(z_2))]\convp 0.
    \end{align*}
\end{lemma}
\begin{proof}
For each $j$, we write
\begin{align*}
    &\E_{j-1}[\E_j(\Tilde{\beta}_j(z_1)\epsilon_j(z_1))\E_j(\Tilde{\beta}_j(z_2)\epsilon_j(z_2))]-\E_{j-1}[\E_j(b_j(z_1)\epsilon_j(z_1))\E_j(b_j(z_2)\epsilon_j(z_2))]\\
   =& \E_{j-1}[\E_j((\Tilde{\beta}_j(z_1)-b_j(z_1))\epsilon_j(z_1))\E_j(\Tilde{\beta}_j(z_2)\epsilon_j(z_2))]\\
   &+\E_{j-1}[\E_j(b_j(z_1)\epsilon_j(z_1))\E_j((\Tilde{\beta}_j(z_2)-b_j(z_2))\epsilon_j(z_2))]\\
   :=&A_1+A_2.
\end{align*}
We shall bound $A_1$ and $A_2$ separately. From the identities (\ref{eq:bd.beta}) and (\ref{eq:beta.tilde.b}) established in Lemma \ref{lem:Mn1.prelim}, we have
\begin{align*}
    |\E A_1|\leq &\E \left[\E_j|(\Tilde{\beta}_j(z_1)-b_j(z_1))\epsilon_j(z_1)|\E_j|\Tilde{\beta}_j(z_2)\epsilon_j(z_2)|\right]\\
    \leq&C_{z_2,v_0}\E \left[|\Tilde{\beta}_j(z_1)-b_j(z_1)|\big|\epsilon_j(z_1)\E_j|\epsilon_j(z_2)|\big|\right]\\
    \leq &N^{-1/2}C_{z_1,z_2,v_0}\E ^{1/4}|\epsilon_j(z_1)|^4\E ^{1/4}|\epsilon_j(z_2)|^4.
\end{align*}
    Along with the concentration result derived in Lemma \ref{lem:concen}, it is straightforward to verify
\begin{align*}
    |\E A_1|\leq C_{z_1,z_2,v_0}\delta_n^2N^{-1},
\end{align*}
and similarly
\begin{align*}
    |\E A_2|\leq C'_{z_1,z_2,v_0}\delta_n^2N^{-1}.
\end{align*}
Collecting the terms gives
\begin{align*}
    \E &\left|\sum_{j=1}^N\E_{j-1}[\E_j(\Tilde{\beta}_j(z_1)\epsilon_j(z_1))\E_j(\Tilde{\beta}_j(z_2)\epsilon_j(z_2))]\right.\\
    &\quad\quad\quad\quad\quad\quad\quad\quad\left.-\sum_{j=1}^N\E_{j-1}[\E_j(b_j(z_1)\epsilon_j(z_1))\E_j(b_j(z_2)\epsilon_j(z_2))]\right|\rightarrow 0,
\end{align*}
and by Markov's inequality we complete the proof.
\end{proof}

\subsection{Proof of Lemma \ref{lem:A123}}
First, we establish a useful lemma.
\begin{lemma}\label{app:lem.prop}
Recall $D_{ij}$, $\beta_{ij}$, $b_{ij}$ defined in equation (\ref{eq:beta.ij}). Further define
\begin{align}
    \Tilde{\beta}_{ij}(z)=\frac{1}{1+N^{-1}\text{tr}(D_{ij}^{-1}T_i)},\quad \epsilon_{ij}=r_i^{-1}D_{ij}^{-1}r_i-N^{-1}\text{tr}(D_{ij}^{-1}T_i).
\end{align}
    Under Assumption \ref{assump_trunc}, for each $z\in\mathbbm{C}^+$ with $\text{Im}z>v$, 
    \begin{align}
         &\E \left|\beta_{ij}(z)-\psi_i(z)\right|^2\leq C|z|^8v^{-14}N^{-1}
    \end{align}
\end{lemma}
\begin{proof}
    Applying Lemma \ref{lem:Mn1.prelim} to $B_n-r_jr_j^T$ instead of $B_n$, we immediately have
    \begin{align*}
    &|\beta_{ij}|,|\deteqv{\beta}_{ij}|\leq |z|v^{-1},\\
         &\E \left|\Tilde{\beta}_{ij}(z)-b_{ij}(z)\right|^2\leq C|z|^4v^{-6}N^{-1}\\
        &|b_{ij}(z)-b_j(z)|\leq C|z|^4v^{-7}N^{-1}.
    \end{align*}
    By direct calculation, 
\begin{align*}
    |\beta_{ij}-\deteqv{\beta}_{ij}|=|\beta_{ij}\deteqv{\beta}_{ij}\epsilon_{ij}|\leq |z|^2v^{-2}|\epsilon_{ij}|.
\end{align*}
As a result, by the concentration result established in Lemma \ref{lem:concen},
\begin{align*}
    \E |\beta_{ij}-\deteqv{\beta}_{ij}|^2\leq |z|^4v^{-4}\E |\epsilon_{ij}|^2=C|z|^4v^{-4}N^{-1}.
\end{align*}
    Taken together with the bound on $|b_j(z)-\psi_j(z)|$ established in Lemma \ref{lem:Mn1.prelim}, we conclude that
    \begin{align*}
        \E \left|\beta_{ij}(z)-\psi_i(z)\right|^2\leq C|z|^8v^{-14}N^{-1}.
    \end{align*}
\end{proof}

\subsubsection*{Bounding  $\text{tr}(A_2(z)M)$}

Let $M$ be any (possibly random) matrix with non-random bound on its spectral norm, denoted by $\TVert{M}$, then
\begin{align*}
    \E |\text{tr}A_2(z)M|&\leq\sum_{i\neq j} \E \left|(\beta_{ij}(z)-\psi_i(z))r_i^TD_{ij}^{-1}MR_j^{-1}r_i\right|\\
    &\leq\sum_{i\neq j}\E ^{1/2}\left|\beta_{ij}(z)-\psi_i(z)\right|^2\E \left|r_i^TD_{ij}^{-1}MR_j^{-1}r_i\right|.
\end{align*}
where
\begin{align*}
   &\E \left|r_i^TD_{ij}^{-1}MR_j^{-1}r_i\right|\\
   \leq &\E ^{1/2}\left|r_i^TD_{ij}^{-1}MR_j^{-1}r_i-N^{-1}\text{Tr}D_{ij}^{-1}MR_jT_j\right|^2+\E \left|N^{-1}\text{Tr}D_{ij}^{-1}MR_jT_j\right|\leq C_z
\end{align*}
for a constant $C_z$ that depends on $z$.
Therefore, together with Lemma \ref{app:lem.prop}, we have
\begin{align}
    \E |\text{tr}A_2(z_1)M|\leq O(N^{1/2}).
\end{align}

\subsubsection*{Bounding  $\text{tr}(A_3(z)M)$}
Recall
\begin{align*}
    A_3(z)=\frac{1}{N}\sum_{i\neq j}\psi_i(z)R_j^{-1}T_i(D_{ij}^{-1}-D_j^{-1}).
\end{align*}
By Lemma \ref{lemma:bound.trA}, for any (possibly random) matrix $M$ with non-random bound on its spectral norm, we obtain
\begin{align*}
    |\text{tr}A_3(z)M|
    &\leq \frac{1}{N}\sum_{i\neq j}\left|\text{tr}(D_{ij}^{-1}-D_j^{-1})\psi_i(z)R_j^{-1}M\right|\\
    &\leq \frac{1}{N}\sum_{i\neq j}v^{-1}|\psi_i(z)| \cdot \lVert R_j^{-1}\rVert \cdot\TVert{M}=O(1).
\end{align*}
\subsubsection*{Bounding  $\text{tr}(A_1(z)M)$}
Let $M$ be a \textbf{non-random} matrix. Recall
\begin{align*}
    A_1(z)=\sum_{i\neq j}\psi_i(z)R_j^{-1}(r_ir_i^T-N^{-1}T_i)D_{ij}^{-1}.
\end{align*}
Define $H_{ij}:=T_i^{1/2}D_{ij}^{-1}MR_j^{-1}T_i^{1/2}$. Note that $H_{ij}$ does not depend on $x_i$, thus
\begin{align*}
    \E |\text{tr}(A_1(z)M)|
    &= \E \left|\sum_{i\neq j}\psi_i(z)\left(r_i^TD_{ij}^{-1}MR_j^{-1}r_i-N^{-1}\text{tr}(T_iD_{ij}^{-1}MR_j^{-1})\right)\right|\\
    &=\E \left| \sum_{i\neq j}\psi_i(z)\left(N^{-1}x_i^TT_i^{1/2}D_{ij}^{-1}MR_j^{-1}T_i^{1/2}x_i-N^{-1}\text{tr}(T_i^{1/2}D_{ij}^{-1}MR_j^{-1}T_i^{1/2})\right)\right|\\
    &=\E \left| \sum_{i\neq j}\psi_i(z)\left(N^{-1}x_i^TH_{ij}x_i-N^{-1}\text{tr}H_{ij}\right)\right|\\
    &\leq \sum_{i\neq j}\E ^{1/2}\left| \psi_i(z)\right|^2\E ^{1/2}\left[\E \left[\left|\left(N^{-1}x_i^TH_{ij}x_i-N^{-1}\text{tr}H_{ij}\right)\right|^2\bigg|x_{-i}\right]\right]\\
    &\leq O(N^{1/2}),
\end{align*}
where the last step follows from the concentration result in  (\ref{eq:martingale.concentr}) and the boundedness of $\psi_i(z)$ established in Lemma \ref{lem:Mn1.prelim}.

\subsection{Proof of Lemma \ref{lem:concen.spectral.S}}
    \begin{proof}

    Let 
    \begin{align*}
        Y&=\begin{bmatrix}
            0&\sum_{r=1}^k\Sigma_r^{1/2}X_r^TL_r\\
            \sum_{r=1}^kL_r^{1/2}X_r\Sigma_r&0
        \end{bmatrix},
    \end{align*}
    then
    \begin{align*}
        Y^2&=\begin{bmatrix}
            B_n&0\\
            0&\underline{B_n}
        \end{bmatrix},
    \end{align*}
    where $\underline{B_n}=(\sum_{r=1}^kL_r^{1/2}X_r\Sigma_r)(\sum_{r=1}^kL_r^{1/2}X_r\Sigma_r)^T$. Hence
    \begin{align*}
        \text{Tr}f(Y^2)&=2\text{Tr}f(B_n)+(n-N)f(0).
    \end{align*}
    Since Gaussian random variables are logarithmic Sobolev, we derive the concentration result by applying Theorem 1.1 from \cite{Guionnet2000}. This application is adapted from the methodology used to prove Corollary 1.8 in \cite{Guionnet2000}, specifically for functions $f$ such that $g(x)=f(x^2)$ is Lipschitz.
   \end{proof} 
\subsection{Additional justifications for equation (\ref{eq:cu.mn1.conv})}\label{app:subsec.cu.mn1.conv}
Our proof relies on the following result. 
\begin{lemma}\label{lem:convd}
    Let $X_n$ be a sequence of random variables. If for each pair of $\epsilon$, $\eta>0$, there exists $Y_n$, such that 
    \begin{align*}
        Y_n\convd X,\quad \mathbbm{P}(|X_n-Y_n|>\epsilon)<\eta,
    \end{align*}
    then
    \begin{align*}
        X_n\convd X.
    \end{align*}
\end{lemma}
 Since $g$ is analytic and bounded respectively on $\mathcal{C}_u$ and $\overline{\mathcal{C}_u}$, it is also uniformly continuous. Therefore, for any $\epsilon>0$, there exists $\delta_1>0$ such that for all $|z_1-z_2|<\delta_1$, 
 \begin{align}\label{eq:bd.g}
     |g(z_1)-g(z_2)|<\epsilon.
 \end{align}
    Recall by definition of tightness of $\{M_n^1(z)\}$ on $\mathcal{C}_u$, for any positive $\epsilon$, $\eta>0$, there exists $\delta_2\in\left[0,1\right]$ such that, for any $|z_1-z_2|\leq \delta_2$, we have
\begin{align}\label{eq:bd.mn1}
    \mathbbm{P}\left[\left|M_n^1(z_1)-M_n^1(z_2)\right|>\epsilon\right]<\eta.
\end{align} 
Let $\delta=\min\{\delta_1,\delta_2\}$, and partition $[-r,r]$ into subintervals each of length less than $\delta$. Denote the partition by $\{-r=x_0<\ldots< x_m=r\}$, with the length of each subinterval as $\Delta x=2r/[2r/\delta]$, then we immediately have
\begin{align*}
    \frac{\sum_{i=1}^m \Delta x g(x_i+iv_0)M_n^1(x_i+iv_0)}{\sqrt{\sum_{i=1}^m\sum_{j=1}^m(\Delta x )^2g(x_{i}+iv_0)g(x_{j}-iv_0)\sigma_n^2(x_i+iv_0,{x}_j-iv_0)}}\rightarrow CN(0,1).
\end{align*}
To conclude the proof, from Lemma \ref{lem:convd} it suffices to prove that for any $\epsilon'$, $\eta'>0$, there exists appropriately chosen $\epsilon$, $\eta$, and the corresponding $\delta$, $m$, such that
\begin{align*}
 &\mathbbm{P}\left[ \left|\frac{\int_{\mathcal{C}_u}g(z)M_n^1(z)dz}{\sqrt{\int_{\mathcal{C}_u}\int_{\overline{\mathcal{C}_u}} g(z_1)g(z_2)\sigma^2_n(z_1,z_2) dz_1 dz_2}}\right.\right.  \\
&\qquad\qquad\qquad-\left.\left.\frac{\sum_{i=1}^m \Delta x g(x_i+iv_0)M_n^1(x_i+iv_0)}{\sqrt{\sum_{i=1}^m\sum_{j=1}^m(\Delta x )^2g(x_{i}+iv_0)g(x_{j}-iv_0)\sigma_n^2(x_i+iv_0,{x}_j-iv_0)}}\right| > \epsilon' \right] < \eta',
\end{align*}
which is true from properties (\ref{eq:bd.g}) and (\ref{eq:bd.mn1}).

\section{Bias calculation}
We begin by establishing some preliminary results. 
\begin{lemma}\label{app:lem.bias}
    Recall $\beta_j(z)$, $b_j(z)$ and $\epsilon_j(z)$ defined in Section \ref{sec:conv.mn1}. Under the assumptions of Lemma \ref{lem:concen}, we have
    \begin{align}
        &\E [\beta_j(z)-b_j(z)]=b_j^3(z)\E [\epsilon_j^2(z)]+o(N^{-1})\label{eq:beta.j.b.j.e}\\
        &\E |\beta_j(z)-\psi_j(z)|^2=O(N^{-1}).\label{eq:betaj.psij.2mom}
    \end{align}
\end{lemma}
\begin{proof}
    Recall $\rho_j(z)$ defined in (\ref{eq:def.rhoj}). Applying the property established in (\ref{eq:beta.j.b.j}) gives
\begin{align}
   \beta_j(z)
   &=b_j(z)-\rho_j(z)b_j^2(z)+\beta_j(z)b_j^2(z)\rho_j^2(z)\label{eq:beta.j.b.j2}\\
   &=b_j(z)-\rho_j(z)b_j^2(z)+b_j^3(z)\rho_j^2(z)-\beta_j(z)b_j^3(z)\rho_j^3(z).\nonumber
\end{align}
Hence
\begin{align*}
     \E [\beta_j(z)-b_j(z)]=b_j^3(z)\E [\rho_j^2(z)]-b_j^3(z)\E [\beta_j(z)\rho_j^3(z)].
\end{align*}
By definition of $\epsilon_j$ and $\rho_j$, we further write 
\begin{align*}
     \E [\beta_j(z)-b_j(z)]=&b_j^3(z)\E [\epsilon_j^2(z)]
    -b_j^3(z)\E [\beta_j(z)\rho_j^3(z)]\\
    &+N^{-2}b_j^3(z)\E \left[\left(\text{Tr}(D_j^{-1}(z)T_j)-\E [\text{Tr}(D_j^{-1}(z)T_j)\right)^2\right]\\
    :=&H_1+H_2+H_3.
\end{align*}
By the boundedness of $b_j$ and $\beta_j$ established in Lemma \ref{lem:Mn1.prelim} and the concentration results in Lemma \ref{lem:concen}, we bound $H_2$ by
\begin{align*}
     |H_2|\leq |b_j(z)|^3\E ^{1/2}|\beta_j(z)|^2\E ^{1/2}|\rho_j(z)|^6\leq CN^{-3/2}=o(N^{-1}).
\end{align*}
Similarly, by the boundedness of $b_j$ and Lemma \ref{lem.0.13}, we bound $H_3$ by 
\begin{align*}
    |H_3|\leq CN^{-2}=o(N^{-1}).
\end{align*}
It follows then that
\begin{align*}
     \E [\beta_j(z)-b_j(z)]=&b_j^3(z)\E [\epsilon_j^2(z)]+o(N^{-1}).
\end{align*}
By definition, we also have
\begin{align*}
    \beta_j-\psi_j&=-\beta_j\psi_j(r_j^TD_j^{-1}r_j-N^{-1}\E [\text{Tr}(D^{-1}(z)T_j)])\\
    &=-\beta_j\psi_j(r_j^TD_j^{-1}r_j-N^{-1}\text{Tr}(D_j^{-1}(z)T_j))\\
    &\quad-\beta_j\psi_j(N^{-1}\text{Tr}(D_j^{-1}(z)T_j)-N^{-1}\E [\text{Tr}(D_j^{-1}(z)T_j)])\\
    &\quad-\beta_j\psi_j(N^{-1}\E [\text{Tr}(D_j^{-1}(z)T_j)]-N^{-1}\E [\text{Tr}(D^{-1}(z)T_j)]).
\end{align*}
We conclude the proof by appealing to Lemma \ref{lem:concen}, Lemma \ref{lem.0.13} and the properties (\ref{eq:djtj.dtj}) and (\ref{eq:bd.betaj.rhoj}).
\end{proof}
\subsection{Calculations of $d_{n0}$}\label{subsec:sec:dn0.comp}
\subsubsection*{Computing $\mathcal{J}_1$}
Applying (\ref{eq:beta.j.b.j2}) yields
\begin{align*}
     \mathcal{J}_1
     &=-\sum_{j=1}^N\E \beta_j(z)\left(\text{Tr}r_j^TD_j^{-1}R^{-1}r_j -\frac{1}{N}\E \text{Tr}R^{-1}T_jD_j^{-1}\right)\\
     &=+\sum_{j=1}^Nb_j^2(z)\E \rho_j(z)\text{Tr}r_j^TD_j^{-1}R^{-1}r_j\\
     &\quad  -\sum_{j=1}^Nb_j^2(z)\E \beta_j(z)\rho_j^2(z)\left(\text{Tr}r_j^TD_j^{-1}R^{-1}r_j-\frac{1}{N}\text{Tr}R^{-1}T_jD_j^{-1}\right)   \\
     &\quad -\frac{1}{N}\sum_{j=1}^Nb_j^2(z)\E \beta_j(z)\rho_j^2(z)\left(\text{Tr}R^{-1}T_jD_j^{-1} -\E \text{Tr}R^{-1}T_jD_j^{-1} \right)\\
     &\quad-\frac{1}{N}\sum_{j=1}^N(b_j^2(z)\E \beta_j(z)\rho_j^2(z)+b_j(z)-\E \beta_j(z))\E \text{Tr}R^{-1}T_jD_j^{-1} \\
     &=+\sum_{j=1}^Nb_j^2(z)\E \rho_j(z)\left(r_j^TD_j^{-1}R^{-1}r_j-\frac{1}{N}\E \text{Tr}R^{-1}T_jD_j^{-1}\right)\\
     &\quad  -\sum_{j=1}^Nb_j^2(z)\E \beta_j(z)\rho_j^2(z)\left(r_j^TD_j^{-1}R^{-1}r_j-\frac{1}{N}\text{Tr}R^{-1}T_jD_j^{-1}\right)   \\
     &\quad -\frac{1}{N}\sum_{j=1}^Nb_j^2(z)\E \beta_j(z)\rho_j^2(z)\left(\text{Tr}R^{-1}T_jD_j^{-1} -\E \text{Tr}R^{-1}T_jD_j^{-1} \right)\\
     &:=\mathcal{J}_{11}+\mathcal{J}_{12}+\mathcal{J}_{13}.
\end{align*}
By definitions of $\rho_j$ in (\ref{eq:def.rhoj}) and $\epsilon_j$ in (\ref{eq:eps}), we have
\begin{align*}
    \mathcal{J}_{11}
    &=\sum_{j=1}^Nb_j^2(z)\E \rho_j(z)\left(r_j^TD_j^{-1}R^{-1}r_j-\frac{1}{N}\E \text{Tr}R^{-1}T_jD_j^{-1}\right)\\
    &=\sum_{j=1}^Nb_j^2(z)\E \epsilon_j(z)\left(r_j^TD_j^{-1}R^{-1}r_j-\frac{1}{N}\E \text{Tr}R^{-1}T_jD_j^{-1}\right)\\
    &\quad +\sum_{j=1}^Nb_j^2(z)\E \left(N^{-1}\text{Tr}(D_j^{-1}(z)T_j)-N^{-1}\E [\text{Tr}(D_j^{-1}(z)T_j)]\right)\\
    &\quad\quad\quad\quad\quad\quad\quad\quad\quad\quad \times\left(r_j^TD_j^{-1}R^{-1}r_j-\frac{1}{N}\E \text{Tr}R^{-1}T_jD_j^{-1}\right)\\
    &:=\mathcal{J}_{111}+\mathcal{J}_{112}.
\end{align*}
Apply Lemma \ref{lemma:epsilon.exp}, then
\begin{align*}
    \mathcal{J}_{111}=\frac{2}{N^2}\sum_{j=1}^Nb_j^2(z)\E [\text{Tr}R^{-1}T_jD_j^{-1}T_jD_j^{-1}]+o(1).
\end{align*}
By Cauchy-Schwartz, the boundedness of $b_j$ established in Lemma \ref{lem:Mn1.prelim}, Lemma \ref{lem.0.13} and Lemma \ref{lem:concen}, we have
\begin{align*}
    |\mathcal{J}_{112}|\leq C/N\rightarrow 0.
\end{align*}
By Cauchy-Schwartz and Lemma \ref{lem:concen} , we also have
\begin{align*}
    |\mathcal{J}_{12}|&\leq \sum_{j=1}^N \left(\E |\beta_j(z)|^4\right)^{1/4}\left(\E |\rho_j(z)|^8\right)^{1/4}\left(\E |r_j^TD_j^{-1}R^{-1}r_j-\frac{1}{N}\text{Tr}R^{-1}T_jD_j^{-1}|^2\right)^{1/2}\\
    &\leq N\cdot C\cdot N^{-2}\cdot N \cdot N^{-1}\cdot \sqrt{N}\rightarrow 0.
\end{align*}
A similar argument gives
\begin{align*}
    |\mathcal{J}_{13}|&\leq C/N\rightarrow 0.
\end{align*}
Collecting the terms, we obtain
\begin{align*}
    \mathcal{J}_1=\frac{2}{N^2}\sum_{j=1}^Nb_j^2(z)\E [\text{Tr}R^{-1}T_jD_j^{-1}T_jD_j^{-1}]  +o(1).
\end{align*}
In Lemma \ref{lem:Mn1.prelim}, we established that 
\begin{align}
    |b_j(z)-\psi_j(z)|\leq \frac{C}{N} \label{eq:bd.bj.psij}
\end{align}
for some constant $C$ that depends on $z$.
As a result, together with Lemma \ref{lem.0.13}, we conclude that
\begin{align*}
    \mathcal{J}_1=
    \frac{2}{N^2}\sum_{j=1}^N\psi_j^2(z)\E [\text{Tr}R^{-1}T_jD^{-1}T_jD^{-1}] +o(1).
\end{align*}

\subsubsection*{Computing $\mathcal{J}_2$}
It follows from (\ref{eq:D.Dj}) and (\ref{eq:beta.j.b.j}) that
\begin{align*}
    \mathcal{J}_2
    &=- \frac{1}{N}\sum_{j=1}^N\E \beta_j(z)\left(\E \text{Tr}R^{-1}T_jD_j^{-1}-\E \text{Tr}R^{-1}T_jD^{-1}\right)\\
    &=- \frac{1}{N}\sum_{j=1}^N\E \beta_j(z)\E \beta_j(z)r_j^TD_j^{-1}R^{-1}T_jD_j^{-1}r_j\\
    &=- \frac{1}{N}\sum_{j=1}^N(b_j(z)-b_j(z)\E \rho_j\beta_j(z))\E (b_j(z)-b_j(z)\rho_j\beta_j(z))r_j^TD_j^{-1}R^{-1}T_jD_j^{-1}r_j\\
    &=- \frac{1}{N}\sum_{j=1}^Nb_j^2(z)\E r_j^TD_j^{-1}R^{-1}T_jD_j^{-1}r_j\\
    &\quad + \frac{2}{N}\sum_{j=1}^Nb_j^2(z)\E \rho_j\beta_j(z)r_j^TD_j^{-1}R^{-1}T_jD_j^{-1}r_j\\
    &\quad- \frac{1}{N}\sum_{j=1}^Nb_j^2(z)\E \rho_j\beta_j(z)\E \rho_j\beta_j(z)r_j^TD_j^{-1}R^{-1}T_jD_j^{-1}r_j\\
    &:=\mathcal{J}_{21}+\mathcal{J}_{22}+\mathcal{J}_{23}.
\end{align*}
Let us start by bounding $\mathcal{J}_{22}$. Rewrite $\mathcal{J}_{22}$ as
\begin{align*}
   \mathcal{J}_{22}
   &=\frac{2}{N}\sum_{j=1}^Nb_j^2(z)\E \rho_j\beta_j(z)\left(r_j^TD_j^{-1}R^{-1}T_jD_j^{-1}r_j-\frac{1}{N}\text{Tr}D_j^{-1}R^{-1}T_jD_j^{-1}T_j\right)\\
   &\quad+\frac{2}{N^2}\sum_{j=1}^Nb_j^2(z)\E \rho_j\beta_j(z)\text{Tr}D_j^{-1}R^{-1}T_jD_j^{-1}T_j.
\end{align*}
Then by the boundedness of $b_j$ and $\beta_j$ established in Lemma \ref{lem:Mn1.prelim} and the concentration results established in Lemma \ref{lem:concen}, we have
\begin{align*}
   &| \mathcal{J}_{22}|\\
   \leq&\frac{2}{N}\sum_{j=1}^N|b_j(z)|^2\E^{1/4} |\rho_j|^4\E^{1/4} |\beta_j(z)|^4\E^{1/2} |r_j^TD_j^{-1}R^{-1}T_jD_j^{-1}r_j-\frac{1}{N}\text{Tr}D_j^{-1}R^{-1}T_jD_j^{-1}T_j|^2\\
   \quad&+\frac{2C}{N}\sum_{j=1}^N|b_j(z)|^2\E^{1/4} |\rho_j|^4\E^{1/4} |\beta_j(z)|^4\leq C/\sqrt{N} \rightarrow 0.
\end{align*}
A similar argument gives
\begin{align*}
     | \mathcal{J}_{23}|=o(1) 
\end{align*}
Collecting the terms, we obtain
\begin{align*}
    \mathcal{J}_2=- \frac{1}{N^2}\sum_{j=1}^Nb_j^2(z)\E \text{Tr}R^{-1}T_jD_j^{-1}T_jD_j^{-1}+o(1).
\end{align*}
Finally, it follows from (\ref{eq:bd.bj.psij}) and Lemma \ref{lem.0.13} that
\begin{align*}
    \mathcal{J}_2=- \frac{1}{N^2}\sum_{j=1}^N\psi_j^2(z)\E \text{Tr}R^{-1}T_jD^{-1}T_jD^{-1}+o(1).
\end{align*}
\subsubsection*{Computing $\mathcal{J}_3$}
We make use of (\ref{eq:djtj.dtj}) and (\ref{eq:bd.betaj.rhoj}) to write
\begin{align}
    b_j(z)-\psi_j(z)&=-\frac{1}{N}b_j(z)\psi_j(z)\E [\text{Tr}(D_j^{-1}T_j)-\text{Tr}(D^{-1}(z)T_j)]\nonumber\\
    &=-\frac{1}{N^2}b_j^2(z)\psi_j(z)\E [\text{Tr}(D_j^{-1}T_jD_j^{-1}T_j)]+o(N^{-1}).\label{eq:bj.psij}
\end{align}
Together with (\ref{eq:beta.j.b.j.e}) established in Lemma \ref{app:lem.bias}, we have
\begin{align*}
    \E [\beta_j(z)-\psi_j(z)]
    =&-\frac{1}{N^2}b_j^2(z)\psi_j(z)\E [\text{Tr}(D_j^{-1}T_jD_j^{-1}T_j)]\\
    &+b_j^3(z)\E [\epsilon_j^2(z)]+o(N^{-1}).
\end{align*}
It then follows that
\begin{align*}
    \mathcal{J}_3
    &=- \frac{1}{N}\sum_{j=1}^N\E (\beta_j(z)-\psi_j(z))\E \text{Tr}R^{-1}T_jD^{-1}\\
    &=- \frac{1}{N}\sum_{j=1}^Nb_j^3(z)\E [\epsilon_j^2(z)]\E \text{Tr}R^{-1}T_jD^{-1}\\
    &\quad +\frac{1}{N^3}\sum_{j=1}^Nb_j^2(z)\psi_j(z)\E [\text{Tr}(D_j^{-1}T_jD_j^{-1}T_j)]\E \text{Tr}R^{-1}T_jD^{-1}+o(1).
\end{align*}
With (\ref{eq:bd.bj.psij}), we substitute $b_j$ with $\psi_j$ and write
\begin{align*}
    \mathcal{J}_3
    &=- \frac{1}{N}\sum_{j=1}^N\psi_j^3(z)\E [\epsilon_j^2(z)]\E [\text{Tr}R^{-1}T_jD^{-1}]\\
    &\quad +\frac{1}{N^3}\sum_{j=1}^N\psi_j^3(z)\E [\text{Tr}D_j^{-1}T_jD_j^{-1}T_j]\E [\text{Tr}R^{-1}T_jD^{-1}]+o(1).
\end{align*}
Lemma \ref{lemma:epsilon.exp} implies
\begin{align*}
    \E [\epsilon_j(z)^2]=N^{-2}2\E [\text{Tr}(D_j^{-1}(z)T_jD_j^{-1}(z)T_j)]+o(N^{-1}).
\end{align*}
Together with Lemma \ref{lem.0.13}, we conclude that
\begin{align*}
     \mathcal{J}_3
    &=- \frac{1}{N^3}\sum_{j=1}^N\psi_j^3(z)\E [\text{Tr}D^{-1}T_jD^{-1}T_j]\E [\text{Tr}R^{-1}T_jD^{-1}]+o(1).
\end{align*}
\subsection{Proof of Lemma \ref{lem:conv.dnr}}
We follow a similar argument leading to the evaluation of $ \text{Tr}(\E_{j}A_1(z_1)T_j\E_{j}D_j^{-1}(z_2)\Sigma_r)$ in (\ref{eq:A1..}). The main difference lies in focusing on $D_j$ instead of $\E_{j}D_j$. Similar to the decomposition in (\ref{eq:decom.rj.dj}), let
\begin{align}
    R^{-1}(z)-D^{-1}(z)
    &=\sum_{j=1}^N\beta_j(z)R^{-1}r_jr_j^TD_j^{-1}(z) - \frac{1}{N}\sum_{j=1}^N\psi_j(z)R^{-1}T_jD^{-1}\\
    &=\sum_{j=1}^N(\beta_j(z)-\psi_j(z))R^{-1}r_jr_j^TD_j^{-1}(z) \\
    &\quad+\sum_{j=1}^N\psi_j(z)R^{-1}(r_jr_j^T-\frac{1}{N}T_j)D_j^{-1}(z) \\
     &\quad+\frac{1}{N}\sum_{j=1}^N\psi_j(z)R^{-1}T_j\left(D_j^{-1}(z)-D^{-1}(z)\right) \\
     &:=G_2(z)+G_1(z)+G_3(z),
\end{align}
Notice that since we are no longer working with $\E_jD_j$, we sum over all indices from 1 to $N$ instead of those smaller than $j$.
Applying Lemma \ref{lem:A123}, we follow a similar argument leading to (\ref{eq:wjr.interm}) to obtain
\begin{align*}
    \text{Tr}(D^{-1}\Sigma_aD^{-1}M)=\text{Tr}(R^{-1}\Sigma_aR^{-1}M)-\text{Tr}(G_1(z)\Sigma_aD^{-1}M)+a_1(z),
\end{align*}
where $\E |a_1(z)|\leq O(N^{1/2})$. Adapting the arguments leading to equations (\ref{eq:A1..}) and (\ref{eq:alpha.4}), we conclude that
\begin{align*}
    \text{Tr}(G_1(z)\Sigma_aD^{-1}M)
    &=-\frac{1}{N^2}\sum_{j=1}^N\psi_j^2(z)\E \text{Tr}D^{-1}(z)T_jD^{-1}\Sigma_a\E \text{Tr}R^{-1}MR^{-1}T_j+a_2(z),
\end{align*}
with $\E |a_2(z)|\leq O(N^{1/2})$. 

\section{Proof of Lemma \ref{lem:main}}
Similar to our proof of Lemma \ref{cor:sos.gamma.gp}, Let $\mathcal{C}$ be a contour containing the interval (\ref{eq:open.int}), with endpoints at $(\pm r,\pm v_0)$.  We split $\mathcal{C}$ into the union of $\mathcal{C}_u$, $\overline{\mathcal{C}_u}$ and $\mathcal{C}_j$, where $\mathcal{C}_u=\{z=x+iv_0,|x|\leq r\}$, $\mathcal{C}_j=\{z=\pm r+iy,|y|\leq v_0\}$. In Section \ref{sec:bias}, we proved that for any $z\in\mathcal{C}_u\cup \overline{\mathcal{C}_u}$, 
\begin{align*}
    M_n^2(z)-\mu_n(z)\convp 0.
\end{align*}
Above, $\mu_n(z)$ is defined in (\ref{eq:mu}). Therefore, to complete the proof of Lemma \ref{lem:main}, it suffices to prove that 
\begin{align*}
    \lim_{v_0\rightarrow 0}\limsup_n\int_{C_j}|M_n^2(z)|^2dz\rightarrow 0,\quad \lim_{v_0\rightarrow 0}\limsup_n\int_{C_j}|\mu_n(z)|^2dz\rightarrow 0.
\end{align*}

Our proof relies on the following results.
\begin{lemma}\label{lem:g.cj}
    Let $D(z)$, $D_j(z)$ be defined as in (\ref{eq:d.dj}). For each of the functions $g$:
        \begin{align*}
            x_j^TT_j^{1/2}D_j^{-1}T_j^{1/2}x_j,\quad \text{Tr}D_j^{-1}T_j,\quad \E \text{Tr}D_j^{-1}T_j \quad \E \text{Tr}D^{-1}T_j,
        \end{align*}
    we have
    \begin{align*}
        \frac{1}{|1+g(z)|}\leq \frac{1}{2}\quad \text{a.s.}
    \end{align*}
    for all $|z|\geq 3C_u$, with $C_u$ defined in Lemma \ref{lem:concen.spectral.S}.
\end{lemma}
\begin{proof}
    Recall from the proof of Lemma \ref{lem:nevanlinna} that we can transform each of the functions $g$ as (\ref{eq:transform.g}). 
From Lemma \ref{lem:trace}, we have for each $g$,
\begin{align*}
    |g(z)|\leq k^2s_L^2s_{\Sigma}^2\left(\sum_{i=1}^n\frac{1}{|\lambda-z|^2}\right)^{1/2}.
\end{align*}
Together with Lemma \ref{lem:concen.spectral.S}, we obtain
\begin{align*}
   |g(z)| \leq \frac{k^2s_L^2s_{\Sigma}^2}{|z|-C_u}
   \leq \frac{1}{2}\quad \text{a.s.}.
\end{align*}
\end{proof}

\begin{corollary}\label{cor:bd.cj}
For all $z=\eta+vi\in C_j$, we have
\begin{gather}
     |\beta_j|,\quad |\Tilde{\beta}_j|,\quad |b_j|,\quad |\psi_j|\leq 1/2 \quad \text{a.s.},\label{eq:bdd.cj.beta}\\
     \lVert R^{-1}\rVert\leq 2|\eta|^{-1}\label{eq:R.inv.cj}
\end{gather}
for $\eta$ large enough.
\end{corollary}
\begin{proof}
Applying Lemma \ref{lem:g.cj}, we immediately have the boundedness results in (\ref{eq:bdd.cj.beta}).
Recall from (\ref{eq:R.decomp}) that
    \begin{align*}
    R=&\frac{1}{N}\sum_{j=1}^N|\psi_j|^2({1+\E [\text{Tr}(U\text{diag}(\frac{\lambda_i-\eta}{(\lambda_i-\eta)^2+v^2})U^TT_j)]})T_j-\eta I\\
    &-iv\left(I+\frac{1}{N}\sum_{j=1}^N|\psi_j|^2({1+\E [\text{Tr}(U\text{diag}(\frac{1}{(\lambda_i-\eta)^2+v^2})U^TT_j)]})T_j\right).
\end{align*}
As a result, $\lambda(R^*R)\geq \eta/2$ for $\eta$ large enough.
\end{proof}

\begin{lemma}\label{lem:D.inv.bd.cj}
    Let $z=\eta+iv$, with $v>n^{-2}$, then
    \begin{align*}
        \E \lVert D^{-1}\rVert\leq 2\eta^{-1}
    \end{align*}
for $\eta$ large enough.
\end{lemma}
\begin{proof}
    Let $Q_n=\{\lambda_{max}(S)\leq C_u+\epsilon\}$. By Lemma \ref{lemma: concen.eig.sofs}, we have that for any $\epsilon>0$ and positive integer $t$, 
\begin{align*}
    \mathbbm{P}[Q_n^c]=o(n^{-t}).
\end{align*}
Thus
\begin{align*}
    \E \lVert D^{-1}\rVert=\E [\lVert D^{-1}\rVert\mathbbm{1}_{Q_n}]+\E [\lVert D^{-1}\rVert\mathbbm{1}_{Q_n^c}]\leq 2\eta^{-1}
\end{align*}
for $\eta$ large enough.
\end{proof}

\subsection{Bounding $\int_{C_j}|\mu_n(z)|^2dz$}
Recall 
\begin{align}
      \mu_n(z)=\deteqv{d}_{n0}(z)+n/N\sum_{r=1}^k\nu_r(z)\sum_{s=1}^kh_{N+1}^{rs}(z,z)\Xi_0^s(z),
  \end{align}
where the quantities $\deteqv{d}_{n0}$, $\nu_r$, $\Xi_0^a$ are defined in (\ref{eq:dnr.bias.deteqv}), (\ref{eq:nu.r}) and (\ref{eq:xi.0.a}). We proceed to bound each part individually. 

Applying property (\ref{eq:bdd.cj.beta}), we have
\begin{align*}
    h_{N+1}^{ab}\leq s_L^4,\quad h^{abc}\leq s_L^6. 
\end{align*}
Applying property (\ref{eq:R.inv.cj}), we further have
\begin{align}\label{eq:xi.bd}
    \Xi_0^{a},\quad \Xi_1^{ab},\quad \Xi_2^{ab}, \quad \Xi_{3}^{abc}\leq 4s_{\Sigma}^2\eta^{-2}
\end{align}
for $\eta$ large enough. Thus, for each $r=0,\ldots,k$,
\begin{align*}
    |\deteqv{d}_{nr}|\leq C_{\eta,s_L,s_{\Sigma}}
\end{align*}
by definition of $\deteqv{d}_{nr}$ in (\ref{eq:dnr.bias.deteqv}), with $C_{\eta,s_L,s_{\Sigma}}$ being a constant that only depends on $\eta$, $s_L$ and $s_{\Sigma}$.
Similarly, from the arbitrarily small upper bound established in (\ref{eq:xi.bd}), we also have
\begin{align*}
    |\nu_r(z)|\leq C'_{\eta,s_L,s_{\Sigma}}.
\end{align*}
Collecting the terms yields 
\begin{align*}
    |\mu_n(z)|\leq C''_{\eta,s_L,s_{\Sigma}}.
\end{align*}
Hence
\begin{align*}
    \lim_{v_0\rightarrow 0}\limsup_n\int_{C_j}|\mu_n(z)|^2dz\rightarrow 0.
\end{align*}

\subsection{Bounding $\int_{C_j}|M_n^2(z)|^2dz$}
Our strategy involves further dividing $C_j$ into $C_{j1}=\{z=\pm r+iy,|y|\leq n^{-2}\}$ and $C_{j2}=\{z=\pm r+iy,n^{-2}\leq|y|\leq v_0\}$

We begin by bounding the integral on $C_{j1}$. Recall that $M_n^2(z)=\E \text{Tr}(B_n-z)^{-1}-\text{Tr}(\deteqv{B}_n-z)^{-1}$. Let $z=\eta+vi$. From Lemma \ref{lem:concen.spectral.S}, it is straightforward to verify that for $\eta$ and $n$ large enough, we have
\begin{align*}
    |M_n^2(z)|\leq 4n\eta^{-1}.
\end{align*}
Hence
\begin{align*}
    \lim_{v_0\rightarrow 0}\limsup_n\int_{C_{j1}}|M_n^2(z)|^2dz\rightarrow 0.
\end{align*}

As for the integral on $C_{j2}$, we follow similar arguments for bounding $\mu_n(z)$ in the previous section. Recall from equation (\ref{eq:mn2.interm}) that
\begin{align*}     M_n^2(z)=&d_{n0}+\sum_{r=1}^kn(\E {g}_2^{(r)}-\deteqv{g}_2^{(r)})\sum_{s=1}^k\frac{1}{N}\sum_{j=1}^Nl_{sj}^2l_{rj}^2\deteqv{b}_j(z)\psi_j(z)\frac{1}{n}\text{Tr}(R^{-1}\Sigma_s\deteqv{R}^{-1}).
\end{align*}
Based on the calculations of $\mathcal{J}_1$, $\mathcal{J}_2$, $\mathcal{J}_3$ in Section \ref{subsec:sec:dn0.comp}, together with Lemma \ref{lem:D.inv.bd.cj}, we obtain
\begin{align*}
    |d_{n0}|\leq C_{\eta,s_L,s_{\Sigma}}.
\end{align*}
Similarly, for $r=1,\ldots,k$, we have
\begin{align*}
    |d_{nr}|\leq C'_{\eta,s_L,s_{\Sigma}}.
\end{align*}
As a result, it follows from the system of equations established in (\ref{eq:ng}), the arbitrarily small nature of $\lVert R^{-1}\rVert$ and the boundedness results established in Corollary \ref{cor:bd.cj} that
\begin{align*}
    |n(\E {g}_2^{(r)}-\deteqv{g}_2^{(r)})|\leq C''_{\eta,s_L,s_{\Sigma}}.
\end{align*}
Finally, we conclude that 
\begin{align*}
    \lim_{v_0\rightarrow 0}\limsup_n\int_{C_{j2}}|M_n^2(z)|^2dz\rightarrow 0.
\end{align*}

\section{Useful results}
\begin{theorem}[Theorem 35.12 of \cite{billingsley2017probability}]\label{thm:martingale}
Suppose for each $n$, $Y_{n1}$,...,$Y_{nr_n}$ is a real martingale difference sequence with respect to the increasing $\sigma$-field $\{F_{nj}\}$ having second moments. If as $n\rightarrow\infty$,
\begin{flalign*}
    (i)&& \sum_{j=1}^{r_n}\E [Y_{nj}^2|F_{n,jk-1}]\convp \sigma^2>0,&\quad&
\end{flalign*}
and for each $\epsilon>0$,
\begin{flalign*}
   (ii)&& \sum_{j=1}^{r_n}\E [Y_{nj}^2\mathbbm{1}(|Y_{nj}|\geq \epsilon)]\rightarrow 0 &\quad&
\end{flalign*}
then
\begin{align*}
    \sum_{j=1}^{r_n}Y_{nj}\convd N(0,\sigma^2).
\end{align*}
\end{theorem}
\begin{lemma}[Lemma 0.11 in \cite{bai2019central}]\label{lemma:bound.trA}
    Let $z\in\mathbbm{C}^+$ with $v=\text{Im}z$. Let $A$ and $B$ be $N\times N$ complex matrices, with $B$ a Hermitian matrix. Then for $\tau\in\mathbbm{R}$ and $r\in\mathbbm{C}^N$, we have
    \begin{align*}
        \left|\text{Tr}\left((B-zI)^{-1}-(B+\tau rr^*-zI)^{-1}\right)A\right|\leq \lVert A\rVert v^{-1}.
    \end{align*}
\end{lemma}
\begin{proof}
    Let $C=B-zI$. From the Woodbury formula, we have
    \begin{align*}
        (C+\tau rr^*)^{-1}-C^{-1}=\tau \frac{C^{-1}rr^*C^{-1}}{1+\tau r^*C^{-1}r}.
    \end{align*}
    Using the spectral decomposition $B=\sum \lambda_j  q_jq_j^*$ with $\lambda_j\in\mathbbm{R}$ and $r_j=q_j^*r$, we have
    \begin{align*}
        \lVert C^{-1}r\rVert^2 
        &= r^*(B-\bar{z}I)^{-1}(B-{z}I)^{-1}r=\sum_j\frac{r_j^2}{|\lambda_j-z|^2},\\
        \text{Im}r^*C^{-1}r
        &=\sum_jr_j^2\text{Im}\frac{1}{\lambda_j-z}=(\text{Im}z)\sum_j\frac{r_j^2}{|\lambda_j-z|^2}.
    \end{align*}
    Hence
    \begin{align*}
        \left|\text{tr}\left[C^{-1}-(C+\tau rr^*)^{-1}\right]A\right|=\left|\frac{\tau r^*C^{-1}AC^{-1}r}{1+\tau r^*C^{-1}r}\right|\leq \frac{\lVert A\rVert \lVert C^{-1}r\rVert^2 }{|\text{Im}r^*C^{-1}r|}\leq \frac{\lVert A\rVert }{\text{Im}z}.
    \end{align*}
\end{proof}
\begin{lemma}[Lemma 0.13 from \cite{bai2019central}]\label{lem.0.13}
    Let $M$ be a nonrandom matrix, for each $j=1,...,n$, we have
    \begin{align*}
        \E |\text{Tr}D_j^{-1}M-\E \text{Tr}D_j^{-1}M|^2\leq C\lVert M\rVert^2.
    \end{align*}
\end{lemma}
\begin{proof}
    Following a similar argument as in (\ref{eq:tele}), we have
    \begin{align*}
      \text{Tr}D^{-1}M-\E \text{Tr}D^{-1}M
         &=\sum_{j=1}^N\E_j\text{Tr}D^{-1}M-\E_{j-1}\text{Tr}D^{-1}M\\
         &=\sum_{j=1}^N(\E_j-\E_{j-1})\text{Tr}(D^{-1}-D_j^{-1})M,\\
         &=-\sum_{j=1}^N(\E_j-\E_{j-1})\beta_j(z)r_j^TD_j^{-1}MD_j^{-1}r_j.
    \end{align*}
    By the mutual independence of $(\E_j-\E_{j-1})\beta_j(z)r_j^TD_j^{-1}MD_j^{-1}r_j$ and the basic algebraic fact that $(x+y)^2\leq 2(x^2+y^2)$, we further have 
    \begin{align*}
        \E |\text{Tr}D^{-1}M-\E \text{Tr}D^{-1}M|^2
        &=\sum_{j=1}^N\E |(\E_j-\E_{j-1})\beta_j(z)r_j^TD_j^{-1}MD_j^{-1}r_j|^2\\
        &\leq4\sum_{j=1}^N\E |\beta_j(z)r_j^TD_j^{-1}MD_j^{-1}r_j|^2\\
        &\leq8\sum_{j=1}^N\E |\beta_j(z)|^2|r_j^TD_j^{-1}MD_j^{-1}r_j-N^{-1}\text{Tr}D_j^{-1}MD_j^{-1}T_j|^2\\
        &+8\sum_{j=1}^N\E |N^{-1}\text{Tr}D_j^{-1}MD_j^{-1}T_j|^2.
    \end{align*}
    Together with Lemma \ref{lem:concen} and the boundedness of $\beta_j$ established in Lemma \ref{lem:Mn1.prelim}, we obtain
    \begin{align*}
        \E |\text{Tr}D^{-1}M-\E \text{Tr}D^{-1}M|^2\leq C\lVert M\rVert^2,
    \end{align*}
    for some constant $C$.
    Applying the same arguments for $D_j(z)$, $j=1,\ldots,N$, we conclude that
    \begin{align*}
        \E |\text{Tr}D_j^{-1}M-\E \text{Tr}D_j^{-1}M|^2\leq C\lVert M\rVert^2.
    \end{align*}
\end{proof}

\begin{lemma}[Equation (1.15) in \cite{baisilverstein2004}]\label{lemma:epsilon.exp}
    Let $x_1=(x_{1i})_1^n$ be a vector with i.i.d. entries satisfying $\E x_{11}=0$, $\E x_{11}^2=\sigma^2$, $\E x_{11}^4<\infty$, then for $n\times n$ matrices $A=(a_{ij})$ and $B=(b_{ij})$, we have
    \begin{align*}
    &\E(x_1^TAx_1-\text{Tr}A)(x_1^TBx_1-\text{Tr}B)\\
       =&(\E|x_{11}|^4-3\sigma^4)\sum_{i=1}^na_{ii}b_{ii}+\sigma^4(\text{Tr}AB+\text{Tr}AB^T)+(\sigma^2-1)^2\text{Tr}A\text{Tr}B.
    \end{align*}   
    Under Assumption \ref{assump_trunc}, if additionally $B=B^T$, the above relationship simplifies to
    \begin{align*}
       \E(x_1^TAx_1-\text{Tr}A)(x_1^TBx_1-\text{Tr}B)=2\text{Tr}AB+o(1)\sum_{i=1}^na_{ii}b_{ii}.
    \end{align*}
\end{lemma}
\begin{proof}
    Since $\E x^TAx=\sigma^2\text{Tr}A$, we get
    \begin{align*}
    \E(x_1^TAx_1-\text{Tr}A)(x_1^TBx_1-\text{Tr}B)= \E(x_1^TAx_1)(x_1^TBx_1)-2\sigma^2\Tr A\Tr B +\Tr A\Tr B.
    \end{align*}
    It remains to study 
    \begin{align*}
     \E(x_1^TAx_1)(x_1^TBx_1)= \sum_{ijkl}a_{ij}b_{kl}\E x_{1i}x_{1j}x_{1k}x_{1l}.
    \end{align*}
    Denote each summand by $C_{ijkl}$, then by independence, the sum reduces to 
    \begin{align*}
       \E(x_1^TAx_1)(x_1^TBx_1)=T_{ii}+T_{ij}+T_{ik}+T_{il},
    \end{align*}
    where
    \begin{align*}
        &T_{ii}=\sum_{i=j=k=l}C_{ijkl}=\E x_{1i}^4\sum_ia_{ii}b_{ii},\\
        &T_{ij}=\sum_{i=j\neq k=l}C_{ijkl}=\sigma^4\sum_{i\neq k}a_{ii}b_{kk}=\sigma^4(\Tr A\Tr B-\sum_ia_{ii}b_{ii}),\\
        &T_{ik}=\sum_{i=k\neq j=l}C_{ijkl}=\sigma^4\sum_{i\neq j}a_{ij}b_{ij}=\sigma^4(\Tr A B^T-\sum_ia_{ii}b_{ii}),\\
        &T_{il}=\sum_{i=l\neq j=k}C_{ijkl}=\sigma^4\sum_{i\neq j}a_{ij}b_{ji}=\sigma^4(\Tr AB-\sum_ia_{ii}b_{ii}).
    \end{align*}
\end{proof}
\begin{lemma}\label{lem:trace}
    For rectangular matrices $A$, $B$, $C$, $D$, we have
    \begin{align*}
        |\text{Tr}(ABCD)|\leq \lVert A\rVert \lVert C\rVert (\text{Tr}BB^*)^{1/2}(\text{Tr}DD^*)^{1/2}.
    \end{align*}
\end{lemma}

\end{document}